\documentclass[11pt,english]{amsart}

\usepackage{babel}
\usepackage[margin=0.9in]{geometry}
\usepackage{amsmath,amssymb,amsfonts,latexsym,cancel}
\usepackage{rawfonts}
\usepackage{pictexwd}

\usepackage{caption}
\usepackage{epsfig}
\usepackage{pst-grad} 
\usepackage{pst-plot} 

\newcommand{\R}{\mathbb{R}}

\newcommand{\eps}{\varepsilon}

\theoremstyle{plain}
\newtheorem{defi}{Definition}[section]
\newtheorem{proposition}[defi]{Proposition}
\newtheorem{theorem}[defi]{Theorem}

\newtheorem{lemma}[defi]{Lemma}

\newtheorem{remark}[defi]{Remark}
\newtheorem{disc}[defi]{Discussion}

\theoremstyle{definition}

\theoremstyle{remark}

\numberwithin{equation}{section}
\numberwithin{figure}{section}
\numberwithin{table}{section}

\begin{document}

\title[Touchdown Localization]{Quantitative touchdown localization for the \\ MEMS problem
with variable dielectric permittivity}

\author{Carlos Esteve}
\address{Universit\'{e} Paris 13, Sorbonne Paris Cit\'{e}, CNRS UMR 7539, Laboratoire Analyse, G\'{e}om\'{e}trie et Applications, 93430, Villetaneuse, France.
{\tt E-mail address: esteve@math.univ-paris13.fr}
}

\author{Philippe Souplet}
\address{Universit\'{e} Paris 13, Sorbonne Paris Cit\'{e}, CNRS UMR 7539, Laboratoire Analyse, G\'{e}om\'{e}trie et Applications, 93430, Villetaneuse, France.
 {\tt E-mail address: souplet@math.univ-paris13.fr}
}

\date{\today}

\begin{abstract} 
We consider a well-known model for micro-electromechanical systems (MEMS)
with variable dielectric permittivity, based on a parabolic equation with singular nonlinearity.
We study the touchdown or quenching phenomenon.
Recently, the question whether or not touchdown can occur at zero points of the permittivity profile~$f$,
which had long remained open, was answered negatively in~\cite{GS}
for the case of interior points,
 and we then showed in~\cite{ES2} 
 that touchdown can actually be ruled out in subregions of $\Omega$ 
where $f$ is positive but suitably small.

The goal of this paper is to further investigate the touchdown localization problem
and to show that, in one space dimension, one can obtain quite quantitative conditions.
Namely, for large classes of typical, one-bump and two-bump permittivity profiles,
we find good lower estimates of the ratio $\rho$ between $f(x)$ and its maximum,
below which no touchdown occurs outside of the bumps. 
The ratio $\rho$ is rigorously obtained as the solution of a suitable {\bf finite-dimensional optimization problem}
(with either three or four parameters), which is then numerically estimated.
Rather surprisingly, it turns out that the values of
 the ratio $\rho$ are not ``small'' but actually {\bf up to the order $\sim0.3$,}
 which could hence be quite appropriate for robust use in practical MEMS design.
 
The main tool for the reduction to the finite-dimensional optimization problem is a quantitative type~I, temporal 
 touchdown estimate. The latter is proved by maximum principle arguments, applied to a multi-parameter family of refined, nonlinear auxiliary functions with cut-off. 
 
\end{abstract}

\maketitle

\section{Introduction}

\subsection{Mathematical problem and physical background}

We consider the problem
\begin{equation}\label{quenching problem}
\left\lbrace \begin{array}{rrl}
u_t - u_{xx} = f(x)(1-u)^{-p}, & x\in \Omega, & t>0, \\
u = 0, & x\in \partial \Omega, & t>0, \\
u(0,x) = 0, & x \in \Omega, &
\end{array} \right.
\end{equation}
where $\Omega=(-R,R)\subset \mathbb{R}$, $p>0$ and 
\begin{equation}\label{hypf}
\hbox{$f\geq 0$ is a H\"older continuous function in $\overline{\Omega}$. }
\end{equation}
Problem (\ref{quenching problem}) with $p=2$ is a known model for micro-electromechanical devices (MEMS)
and has received a lot attention in the past 15 years.
An idealized version of such device consists of two conducting plates, connected to an electric circuit.
The lower plate is rigid and fixed while the upper one is elastic and fixed only at the boundary.
Initially the plates are parallel and at unit distance from each other.
When a voltage (difference of potential between the two plates) is applied, the upper
plate starts to bend down and, if the voltage is large enough,
the upper plate eventually touches the lower one. This is called {\it touchdown} 
phenomenon.
Such device can be used for instance as an actuator, a microvalve (the touching-down part closes the valve),
or a fuse.

In the mathematical model, $u=u(t,x)$ measures the vertical deflection of the upper plate
and the function $f(x)$ represents the dielectric permittivity of the material
(and is also proportional to the -- constant -- applied voltage).
 As a key feature, the permittivity $f$ may be inhomogeneous
and this can be used to trigger the properties of the device.
We refer to~\cite{EGG} and the references therein for the full details of the model derivation.

It is well known that problem (\ref{quenching problem}) admits a unique maximal classical solution~$u$.
We denote its maximal existence time by $T=T_f\in (0,\infty]$. Moreover,
under some largeness assumption on $f$, it is known that the maximum of $u$ reaches the value $1$
at a finite time, so that $u$ ceases to exist in the classical sense, i.e. $T<\infty$.
This property, known as quenching, is the mathematical counterpart of the touchdown phenomenon. 

A point $x = x_0$ is called a {\it touchdown} or {\it quenching point} if there exists a sequence $\{(x_n , t_n )\} \in \Omega\times (0, T)$ such that
$$x_n\to x_0,\ \ t_n\uparrow T\ \ \hbox{and}\ \ u(x_n,t_n)\to 1 \ \hbox{as}\ n\to\infty.$$
The set of all such points is called the {\it touchdown} or {\it quenching set}, denoted by $\mathcal{T}=\mathcal{T}_f\subset\overline{\Omega}$. 

 In the past decades, MEMS problems, including system (\ref{quenching problem}) and the related touchdown issues,
 have received considerable attention in the physical and engineering as well as 
in the mathematical communities. We refer to \cite{EGG}, \cite{JAP-DHB02} for more details on the physical background, 
and to, e.g., \cite{G97}, \cite{FT00} \cite{PT01}, \cite{YG-ZP-MJW05}, \cite{FMPS07}, \cite{GG07}, \cite{GG08}, \cite{GHW08}, \cite{G08}, 
\cite{G08A},  \cite{KMS08},  \cite{GK}, \cite{G14}, \cite{GS} for mathematical studies.
See also \cite{Phi87}, \cite{L89}, \cite{G1}, \cite{FLV} for earlier mathematical work on the case of constant $f$.

\subsection{Motivation} 

 The question whether or not touchdown can occur at zero points of the permittivity profile~$f$,
raised in \cite{G08A}, \cite{GG08}, \cite{G08}, \cite{EGG}, was answered negatively in~\cite{GS} for the case of interior points.
 This is by no means obvious since, for the analogous blowup problem $u_t - \Delta u = f(x)u^p$
with $f(x)=|x|^\sigma$, examples of solutions with single-point blowup at the origin have been constructed in \cite{FT00}, \cite{GS11}
for suitable $\sigma>0$, $p>1$ and suitable initial data $u_0\ge 0$.
We then showed in~\cite{ES2} that touchdown can actually be ruled out in subregions of $\Omega$ 
where $f$ is positive but suitably small. The following theorem collects the two smallness criteria given in~\cite{ES2}.

\noindent \textbf{Theorem.} \textit{ 
Let $p>0$, $\Omega\subset \mathbb{R}^n$ a smooth bounded domain and $f$ a function satisfying \eqref{hypf} and 
\begin{equation*}
\left\lbrace \begin{array}{rrl}
&&T_f\le M,\quad \|f\|_\infty\le M,\quad f\ge r\chi_B, \\
\noalign{\vskip 1mm}
&&\hbox{where $M, r>0$ and $B\subset\Omega$ is a ball of radius $r$.}
\end{array} \right.
\end{equation*}
There exists $\gamma_0>0$ depending only on $p,\Omega,M,r$ such that:
\begin{itemize}
\item[{\it (i)}] For any $x_0\in \Omega$, if $f(x_0) <\gamma_0 {\hskip 1pt} \text{dist}^{p+1} (x_0,\partial\Omega)$, then $x_0$ is not a touchdown point.
\smallskip
\item[{\it (ii)}] For any $\omega\subset\subset \Omega$, if $\displaystyle\sup_{x\in \overline{\Omega}\setminus \omega} f(x) <
\gamma_0 {\hskip 1pt} \text{dist}^{p+1}(\omega,\partial\Omega)$, then the touchdown set 
 is contained in $\omega$. 
 \end{itemize}}

 Motivated by practical considerations of MEMS design, our aim in this article is to further investigate the 
touchdown localization problem
and to show that in one space dimension, where analytic computations can be made more precise,
one can obtain quite quantitative conditions.
 Namely, we look for a lower estimate of the ratio $\rho$ between $f$ and its maximum,
below which no touchdown occurs on a subregion of $\Omega$.
Rather surprisingly, it turns out that in the physical case $p=2$, under suitable assumptions on $f$,
our methods yield values of
 the ratio $\rho$ which are not ``small'' but can actually be up to the order $\rho\sim 0.3$, 
 which could hence be quite appropriate for robust practical use.

\subsection{Reduction to a finite-dimensional optimization problem and quantitative results} 

In order to give good estimates of the  ratio $\rho$, we shall consider two typical situations, which roughly correspond to a ``one-bump'' or a ``two-bump'' shape for the profile~$f$. 
 The touchdown is ruled out in a subinterval respectively located between a bump and 
an endpoint of $\Omega$, or between two bumps.

The idea behind this is that the plate can be covered with two dielectric materials, 
one with a high permittivity and the other with a lower permittivity. We then seek for a ratio between the two permittivities, allowing to rule out touchdown in the low permittivity region.

We point out that, as a consequence of our method, the ratio $\rho$ is rigorously obtained as the solution of a suitable {\it finite-dimensional optimization problem}, with either three or four parameters. 
Such kind of reduction in nonlinear parabolic problems is new, as far as we know. 

In spite of the rather awkward shape of the optimization problem,
it turns out to yield quite reasonable practical values of the threshold ratio $\rho$
in concrete cases.
 Before presenting our rigorous statements, let us illustrate 
 the results for the physical case $p=2$ by some concrete examples,
that can be deduced from them by a relatively 
simple numerical procedure applied to the finite-dimensional optimization problem
(see also Table~\ref{table intro} below for more applications). 
The following two figures represent some typical permittivity profiles $f(x)$ and the localization of the corresponding touchdown sets,
 in the one-bump and two-bump cases, respectively. 
The touchdown sets are localized in a neighborhood of the bumps, represented by the fat lines.


\begin{figure}[h]
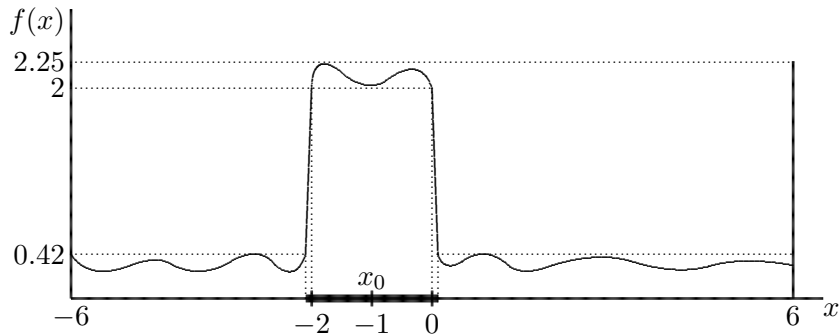

\[
\beginpicture
\setcoordinatesystem units <0.8cm,0.7cm>
\setplotarea x from -6 to 8, y from -1 to 5.5

\setdots <0pt>
\linethickness=1pt
\putrule from -5 0 to 7.5 0
\putrule from -5 0 to -5 5.5
\putrule from 7 0 to 7 4.5

\setquadratic
\plot -5 0.84  -4.6 0.53  -4 0.63 / 
\plot -4 0.63  -3.6 0.73  -3.3 0.63 /
\plot -3.3 0.63  -2.8 0.52  -2.3 0.73 /
\plot -2.3 0.73  -1.9 0.84  -1.6 0.63 /
\plot -1.6 0.63  -1.3 0.52  -1.1  0.84 /

\plot -1 4  -0.84 4.45  -0.4 4.2 /
\plot -0.4 4.2  -0.05 4.05  0.25 4.15 /
\plot 0.25 4.15  0.72 4.35  1 4 /

\plot 1.1 0.84  1.25 0.63  1.5 0.7 /
\plot 1.5 0.7  1.9 0.84  2.3 0.6 /  
\plot 2.3 0.6 2.6 0.52  3 0.63 /  
\plot 3 0.63  3.7 0.78  4.3 0.71 /
\plot 4.3 0.71  5 0.55  5.6 0.61 /
\plot 5.6 0.61  6.23 0.71  7 0.63 /

\setlinear
\plot -1.1 0.84  -1 4 / 
\plot 1 4  1.1 0.84 /

\setdots <2pt>
\setlinear 
\plot -1.1 0  -1.1 0.84 / 
\plot -1 0  -1 4 /
\plot 1 0  1 4 /
\plot 1.1 0  1.1 0.84 /
\plot -5 4  1 4 /
\plot -5 4.5  7 4.5 /
\plot -5 0.84  -1.1 0.84 / 
\plot 1.1 0.84  7 0.84 /

\put {$x$} [lt] at 7.5 -.1
\put {$6$} [ct] at 7 -.15
\put {$0$}   [ct] at 1 -.3
\put {$-1$}   [ct] at 0 -.3
\put {$x_0$}  [ct] at 0 .5
\put {$-2$}  [ct] at -1 -.3
\put {$-6$}  [ct] at -5 -.15
\put {$f(x)$} [rc] at -5.1 5.25
\put {$2.25$}  [rc] at -5.1 4.5
\put {$2$} [rc] at -5.1 4
\put {$0.42$} [rc] at -5.1 0.84 

\setdots <0pt>
\linethickness=3pt
\putrule from -1.1 0 to 1.1 0 

\linethickness=1pt
\putrule from -1 .15 to -1 -.15 
\putrule from 1 .15 to 1 -.15 
\putrule from 0 .15 to  0 -.15 

\endpicture
\] 
\caption{An illustration of the localization of the touchdown set in the one-bump case for $p=2$}
\label{fig1}
\end{figure}

\goodbreak

\begin{figure}[h]
\[
\beginpicture
\setcoordinatesystem units <0.5cm,0.7cm>
\setplotarea x from -11 to 11, y from -1 to 5.5

\setdots <0pt>
\linethickness=1pt
\putrule from -10 0 to 10.5 0
\putrule from -10 0 to -10 5.5
\putrule from 10 0 to 10 4.5

\setquadratic
\plot -10 0.68  -9 0.73  -8 0.63 / 
\plot -8 0.63  -7 0.53  -6.3 0.58 /
\plot -6.3 0.58  -5.6 0.66  -5 0.63 /
\plot -5 0.63  -4.5 0.61  -4 0.68 /
\plot -4 0.68  -3.1 0.79  -2.4 0.63 /
\plot -2.4 0.63  -1.8 0.57  -1.4 0.63 /
\plot -1.4 0.63  -1.25 0.7  -1.1 0.84 /

\plot -1 4  -0.84 4.45  -0.5 4.25 /
\plot -0.5 4.25  -0.28 4.12  -0.07 4.3 /
\plot -0.07 4.3  0.13 4.45  0.32 4.35 /
\plot 0.32 4.35  0.51 4.2  0.7 4.25 / 
\plot 0.7 4.25  0.9 4.3  1 4 /

\plot 1.1 0.84  1.3 0.63  1.8 0.7 /
\plot 1.8 0.7  2.3 0.8  3 0.67 /
\plot 3 0.67  3.7 0.63  3.9 0.84 /

\plot 4 4  4.3 4.45  4.7 4.25 /
\plot 4.7 4.25  5.05 4.1  5.45 4.23 /
\plot 5.45 4.23  5.8 4.33  6 4 /

\plot 6.1 0.84  6.3 0.71  6.6 0.74 /
\plot 6.6 0.74  7.2 0.82  7.8 0.68 /
\plot 7.8 0.68  8.8 0.58  10 0.79 /

\setlinear
\plot -1.1 0.84  -1 4 /
\plot 1 4  1.1 0.84 /
\plot 3.9 0.84  4 4 /
\plot 6 4  6.1 0.84 /
 
\setdots <2pt>
\setlinear 
\plot -1.1 0  -1.1 0.84 /
\plot -1 0  -1 4 /
\plot 1 0  1 4 /
\plot 1.1 0  1.1 0.84 /
\plot 3.9 0  3.9 0.84 /
\plot 4 0  4 4 /
\plot 6 0  6 4 /
\plot 6.1 0  6.1 0.84 /
\plot -10 4  6 4 /
\plot -10 4.5  10 4.5 /
\plot -10 0.84  -1.1 0.84 /
\plot 1.1 0.84  3.9 0.84 /
\plot 6.1 0.84  10 0.84 /

\put {$x$} [lt] at 10.5 -.1
\put {$10$} [ct] at 10 -.15
\put {$-10$} [ct] at -10 -.15
\put {$6$}   [ct] at 6 -.3
\put {$x_0$}  [ct] at 5 .5
\put {$4$}   [ct] at 4 -.3
\put {$1$}  [ct] at 1 -.3
\put {$x_1$} [ct] at 0 .5
\put {$-1$}  [ct] at -1 -.3
\put {$f(x)$} [rc] at -10.1 5.25
\put {$2.25$}  [rc] at -10.1 4.5
\put {$2$} [rc] at -10.1 4
\put {$0.42$} [rc] at -10.1 0.84

\setdots <0pt>
\linethickness=3pt
\putrule from -1.1 0 to 1.1 0
\putrule from 3.9 0 to 6.1 0 

\linethickness=1pt
\putrule from -1 .15 to -1 -.15 
\putrule from 1 .15 to 1 -.15 
\putrule from 4 .15 to  4 -.15 
\putrule from 6 .15 to  6 -.15

\endpicture
\] 
\caption{An illustration of  the localization of the touchdown set in the two-bump case for $p=2$}
\label{fig2}
\end{figure}

\noindent { Here, the profile $f$ satisfies 
$$\|f\|_\infty \leq 2.25,\qquad f(x)\geq 2\ \hbox{ in $\omega$},\qquad 
f(x)\leq 0.42\ \hbox{ in $\Omega\setminus\tilde\omega$,}$$  
with respectively $\Omega=(-6,6)$,  $\omega=(-2, 0)$, $\tilde\omega=(-2.1, 0.1)$
and $\Omega=(-10,10)$, $\omega=(-1,1)\cup (4,6)$, $\tilde\omega= (-1.1,1.1)\cup (3.9,6.1)$.
The touchdown set is then contained in $\tilde\omega$. }

 Let us now state our first two main theorems. We begin with the ``one-bump'' case.
Here we recall that the error function is defined as 
$$\mbox{erf}(x)=\dfrac{2}{\sqrt{\pi}}\displaystyle\int_0^x e^{-t^2} dt$$
and we also set $\overline\cot \,s=\cot s$ if $0<s\le \pi/2$, and $\overline\cot \,s=0$ if $s>\pi/2$.

  \goodbreak
  
\begin{theorem}\label{global result dim 1}
Consider problem (\ref{quenching problem}) with $p>0$ and $\Omega=(-R,R)$.
Let $x_0\in \Omega$ and assume 
\begin{equation}\label{hypmu1d0}
\mu>\mu_1(p):= \dfrac{p^p}{(p+1)^{p+1}} \dfrac{\pi^2}{2},\qquad 
d_0:=R-|x_0|-1>\frac{p+1}{\sqrt{p\mu}}\,\overline\cot\bigl[\sqrt{p\mu}\bigr].
\end{equation}
For each $d\in(0,d_0)$, there exists $\rho=\rho(p,\mu,\|f\|_\infty,d_0,d)\in (0,1)$ 
such that, if $f$ satisfies
\begin{equation}\label{condition 1 on f}
f\ge \mu \quad\hbox{ in $\bigl(x_0-1,x_0+1\bigr)$}
\quad\hbox{and}\quad
f<\rho\mu \quad\hbox{ in $ D:= [x_0+1+d,R]$,}
\end{equation}
then $T_f<\infty$ and there are no touchdown points in $D$.

In addition, $\rho$ can be chosen as the solution of the following optimization problem:
\begin{equation}\label{simplified problem}
\rho = \dfrac{1}{2}\, \underset{(\tau,\beta,K)\in \mathcal{A}}{\rm Sup} \left( \dfrac{\beta-d}{\beta}\right)^{p+1} 
\dfrac{S(t_0(\tau),\beta)}{K+\tau^{-p}} \min \Bigl\{ H\bigl(t_0(\tau),\beta\bigr), G\bigl(t_0(\tau),\beta,K\bigr) \Bigr\}
\end{equation}
with
$$
\mathcal{A} = \biggl\{(\tau,\beta,K) \in (0,1)\times (d,d_0)\times (0,\infty);\,\tau\ge
\textstyle\frac{\mu}{2\mu-\mu_1},\,K\geq \textstyle\frac{p}{\mu\beta^2}-\textstyle\frac{1}{p+1},\,\delta(\beta,K)\leq 1\biggr\}.
$$
Here we set
$$t_0(\tau) = \textstyle\frac{1-\tau^{p+1}}{(p+1)\|f\|_\infty},\qquad \, L=1+(p+1)K,\qquad \,\Gamma=\sqrt{\textstyle\frac{(p+1)L}{pK\mu\beta^2}},
\qquad \,A = \arctan\Gamma,\qquad \alpha = 1+ \textstyle\frac{p}{L}$$
and the functions $S, H, G, \delta$ are defined by
\begin{equation}\label{H and G}
\begin{array}{ll}
&S(t,\beta) =  e^{-\frac{\pi^2t}{4(d_0+1)^2}}\left[1-e^{-\frac{d_0(d_0-\beta)}{t}}\right], \qquad
\delta(\beta,K) = A(1+K)\sqrt{\frac{p+1}{pLK\mu}},\\
\noalign{\vskip 3mm}
&H(t,\beta) = \displaystyle\inf_{0<x<1}  \dfrac{\text{\rm erf} \left( \frac{1}{\sqrt{t}}\left(1+\frac{\beta}{2}x\right)\right) 
- \text{\rm erf} \left(\frac{\beta}{2\sqrt{t}}x\right)}{(1-x)^{p+1}}, \\
\noalign{\vskip 3mm}
&G(t,\beta,K) = (\Gamma^2+1)^{-\alpha/2}\ \underset{0<x<1}{\inf} \dfrac{\text{\rm erf} \left( \frac{2-(1-x)\delta}{2\sqrt{t}}\right) + 
\text{\rm erf} \left( \frac{(1-x)\delta}{2\sqrt{t}}\right) }{\cos^\alpha (Ax)}.
\end{array}
\end{equation}
\end{theorem}

\goodbreak
For the ``two-bump'' case, we have:

\goodbreak

\begin{theorem}\label{global result dim 1a}
Consider problem (\ref{quenching problem}) with $p>0$ and $\Omega=(-R,R)$.
Let $x_0, x_1\in \Omega$ be such that $|x_0|\ge |x_1|$ and assume \eqref{hypmu1d0}.
For each $d\in(0,d_0)$, there exists $\rho=\rho(p,\mu,\|f\|_\infty,d_0,d)\in (0,1)$ 
such that, if $x_0-x_1>2(1+d)$ and $f$ satisfies
\begin{equation}\label{condition 2 on f}
f\ge \mu  \quad\hbox{ in $\bigl(x_1-1,x_1+1\bigr)\cup \bigl(x_0-1,x_0+1\bigr)$}
\end{equation}
and 
\begin{equation}\label{condition 2 on fb}
f<\rho\mu  \quad\hbox{ in $ D:= \bigl[x_1+1+d,x_0-1-d\bigr]$,}
\end{equation}
then $T_f<\infty$ and there are no touchdown points in $D$. 
In addition, $\rho$ can be chosen as the solution of the optimization problem (\ref{simplified problem}). 
\end{theorem}

 Of course, in order to apply Theorems~\ref{global result dim 1}-\ref{global result dim 1a},
 it is not necessary in practice to determine the value of $\rho$ itself.
Any number $\overline\rho<\rho$ can be used instead of $\rho$ in assumptions (\ref{condition 1 on f}) and (\ref{condition 2 on fb}).
It therefore suffices to evaluate the function in the RHS of (\ref{simplified problem}) for suitable choices of $(\tau,\beta,K)\in \mathcal{A}$.

 In Table~\ref{table intro}, for the physical case $p=2$ and physically reasonable values of the parameters,
  we present some numerical lower estimates of the threshold ratio $\rho$ (see the column $\overline{\rho}_1$). 
 They show that Theorems~\ref{global result dim 1}-\ref{global result dim 1a}
allow to reach ratios up to the order of 
$$\rho\sim 0.17,$$
which seems quite satisfactory in view of robust practical
conception of MEMS, in which one would like to prevent touchdown in specific parts of the device
by proper design of the permittivity profile.
As for the results in the column $\overline{\rho}_2$ of Table~\ref{table intro}, they even give values up to
$$\rho\sim 0.3.$$
However, they are based on a more complicated optimization problem,
whose lengthy statement is therefore postponed to Section~2 (see Theorem~\ref{global result dim 1c}).
Figures~\ref{fig1} and \ref{fig2} above are based on the second  line in Table~\ref{table intro} (using $\overline{\rho}_2$ as lower estimate for $\rho$).  Note that in the example of Figure~\ref{fig2}, we are applying the localization criteria from the two-bump and one-bump cases at the same time (between and at the exterior of the two bumps).

\begin{table}[h] 
\begin{center}
\begin{tabular}{|cccc|c|c|}
\hline
$\mu$ & $\|f\|_\infty$   & $d$   & $d_0$   & $\overline{\rho}_1$   & $\overline{\rho}_2$   \\
\hline   
1   &   1.1   &   0.1   &   5   &   \textbf{0.1050}   &   \textbf{0.2249}   \\
\hline
2   &   2.25     &   0.1  &    4  &   \textbf{0.1182}   &   \textbf{0.2111}   \\
\hline
3   &   3.5   &   0.01  &   5   &   \textbf{0.1554}   &   \textbf{0.2698}   \\
\hline
6   &   6.2   &   0.01  &   10   &   \textbf{0.1682}   &   \textbf{0.2856}   \\
\hline
10   &   10   &   0.005  &   10   &   \textbf{0.1732}   &   \textbf{0.2921}   \\
\hline
\end{tabular}
\end{center}
\vskip 1mm
\caption{Lower estimates of the threshold ratio $\rho$ for $p=2$, using 
Theorems~\ref{global result dim 1}-\ref{global result dim 1a} (column $\overline{\rho}_1$) 
and Theorems~\ref{global result dim 1c}-\ref{global result dim 1c2}
 (column $\overline{\rho}_2$).}
\label{table intro}
\label{tableex1}
\end{table}

 The evaluation of $\overline{\rho}_1$ and $\overline{\rho}_2$ in Table~\ref{table intro} 
is done with the help of the computational tool $Matlab$,
and this can be done with very good accuracy.
See Section~6 for details on the numerical procedure.
In particular we stress that we use a ``monotone'' discretization scheme to evaluate the infima in (\ref{H and G}),
which guarantees that the discrete infima are not larger than the exact ones.
In this way, the only possible sources of errors in excess on $\overline\rho$ are the round-off machine errors and the numerical errors in 
the $Matlab$ evaluations (for instance those of $\text{\rm erf}$).
In principle this can be guaranteed with any reasonably prescribed safety margin.


   \medskip

\begin{disc}\label{disc1}
 (a) It is a natural question whether the above touchdown localization behavior
could be true whenever $\rho<1$ in  (\ref{condition 1 on f}), (\ref{condition 2 on fb}). 
Actually,  we show in~\cite{ES2} that this is {\bf not} the case. Indeed, among other things, we construct examples
showing that, for some class of symmetric ``M''-shaped profiles in $\Omega=(-R,R)$, where $f(0)$ is less but
close enough to the maximum of $f$,
touchdown does occur at the origin, and only there.
Moreover, interestingly, the touchdown set is then located {\bf far away} from the points of maximum of $f$.
 This is illustrated in Figure~\ref{SinglePointTouchdown}.
Although the function $f$ is also a two-bump profile, it clearly presents a reverse situation to that described in 
Theorems~\ref{global result dim 1}-\ref{global result dim 1a} and illustrated in Figure~\ref{fig2}. 
This shows that the threshold ratio $\rho$ cannot exceed a certain value $\rho_0$ less than~$1$ 
and that there is an intermediate range $(\rho_0,1]$ where results 
of the type of Theorems~\ref{global result dim 1}-\ref{global result dim 1a} cannot hold. 

\begin{figure}[h]
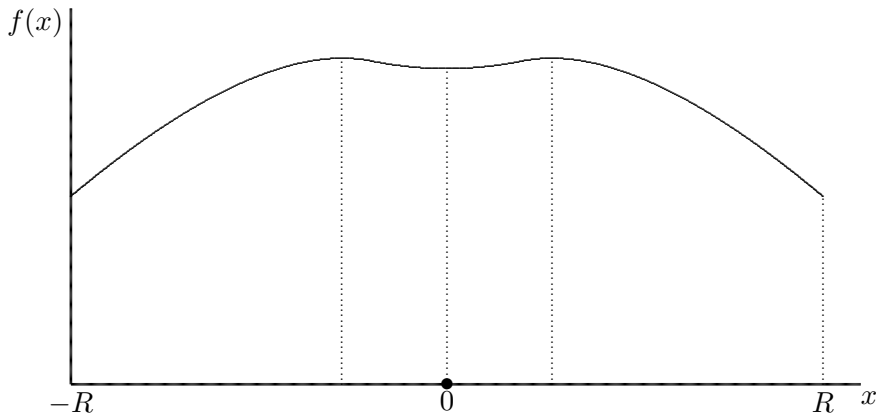

\[
\beginpicture
\setcoordinatesystem units <1cm,1cm>
\setplotarea x from -6 to 6, y from -1 to 5

\setdots <0pt>
\linethickness=1pt
\putrule from -5 0 to 5.5 0
\putrule from -5 0 to -5 5

\setquadratic

\plot -5 2.5  -2.75 4  -1 4.3 /
\plot 5 2.5  2.75 4  1 4.3 /
\plot -1 4.3  0 4.2  1 4.3 /

\setdots <2pt>
\setlinear
\plot 0 0  0 4.2 /
\plot -1.4 0  -1.4 4.3 /
\plot 1.4 0  1.4 4.3 /
\plot 5 0  5 2.5 /

\put {$x$} [lt] at 5.5 -.1
\put {$R$} [ct] at 5 -.1
\put {$0$}   [ct] at 0 -.1
\put {$-R$}  [ct] at -5 -.1
\put {$f(x)$} [rt] at -5.1 5

\put {$\bullet$} [cc] at 0 0

\endpicture
\]
\caption{Touchdown only at the origin for an ``M''-shaped profile
corresponding to the intermediate range of $\rho$ (cf.~Discussion~\ref{disc1}(a)).} 
\label{SinglePointTouchdown}
\end{figure}


(b)  However,  the values of $\rho$ found in Theorems~\ref{global result dim 1}, \ref{global result dim 1a}, \ref{global result dim 1c}--\ref{global result dim 1b}
are probably not optimal and we have no indication 
what the value of the optimal threshold $\rho$ should be.
This seems to be a difficult problem.

(c) For notational simplicity, we have chosen intervals of radius~1 in conditions (\ref{condition 1 on f}), (\ref{condition 2 on f}), but by a straightforward scaling argument, one can see that this entails no loss of generality.

(d) When $\Omega$ is a ball and $f$ is monotonically decreasing with respect to $|x|$, 
the touchdown set is reduced to the origin (see \cite{G1}, \cite{G08}).
We stress that in Theorems~\ref{global result dim 1}-\ref{global result dim 1a} we do not make any kind of monotonicity or symmetry  assumptions on $f$.  On the other hand, for general nonmonotone profiles $f$, it remains an open -- and probably difficult -- problem to determine the finer structure of the touchdown set beyond the localization properties in Theorems~\ref{global result dim 1}-\ref{global result dim 1a}.
Numerical simulations in~\cite{EGG} suggest that it need not consist of isolated points but might contain intervals.

(e) Consider the reference case when $\Omega=(-1,1)$ and $f=k=Const.$,
and recall that there exists a number $k^*$ such that $T<\infty$ if $k>k^*$ and $T=\infty$ if $k<k^*$.
This is the so-called pull-in voltage, for a unit profile (see e.g.~\cite{EGG}). We have 
$$\frac{2p^p}{(p+1)^{p+1}} \le k^*\le \frac{p^p}{(p+1)^{p+1}}\dfrac{\pi^2}{4}$$
(see \cite[Proposition 2.2.1 and 2.2.2]{EGG} and their proofs).
In particular for $p=2$, we have $0.30\sim \frac{8}{27} \le k^*\le\frac{\pi^2}{27}\sim 0.37$, 
whereas Theorems~\ref{global result dim 1}-\ref{global result dim 1a} require $\mu\geq\frac{2\pi^2}{27}\sim 0.73.$
Although this condition is a bit more restrictive, we note that $ 0.73$
still remains of the same order of magnitude as the reference pull-in voltage.
On the other hand, Theorem~\ref{global result dim 1b} below applies whenever $\mu>\frac{p^p}{(p+1)^{p+1}}\frac{\pi^2}{4}$ ($\sim 0.37$ for $p=2$).
\end{disc}

 The article is structured as follows. In Section 2 we present further localization results, 
either giving more precise estimates of $\rho$ 
(Theorems~\ref{global result dim 1c}-\ref{global result dim 1c2}),
or requiring weaker assumptions on $\mu, d_0$ (Theorem~\ref{global result dim 1b}).
In Section 3, we state and prove a type I touchdown estimate (Proposition~\ref{type 1 quantitative}),
which is a key ingredient in the proof of the localization results.  
Here is where our basic, multi-parameter auxiliary function is introduced, which will eventually lead to the optimization problem.
The proofs of Theorems~\ref{global result dim 1}, \ref{global result dim 1a} and \ref{global result dim 1b} are given in Section 4, 
as a consequence of a no-touchdown criterion (Lemma~\ref{basic supersol}), combined with the previous type I estimate.
In Section 5, we prove Theorems~\ref{global result dim 1c}-\ref{global result dim 1c2}, as a consequence of a more precise type I estimate, obtained by refining various ingredients from the proof of Theorems~\ref{global result dim 1}-\ref{global result dim 1a}.
In Section 6, we describe the numerical procedures that we use to handle the optimization problem.
Finally, the article is completed by two appendices. The first one provides some useful quantitative comparison estimates for heat semigroups,
and the second one is devoted to establishing the optimality of the  cut-off functions 
$a(r)$ appearing in our proofs of localization.


\section{Further quantitative results}\label{further results}

\subsection{Improved ratio}

 As announced in introduction, the following theorem allows one to obtain better estimates for the ratio $\rho$
 (cf.~the last column of Table~\ref{table intro}), at the expense of a more complicated optimization problem.
 Although the statement may seem somewhat lengthy, we stress that this result allows for quite good estimates of $\rho$
 (cf.~Table~\ref{table intro}).
In what follows, we set
$$\mu_0(p):= \dfrac{p^p}{(p+1)^{p+1}} \dfrac{\pi^2}{4}.$$
\goodbreak

\begin{theorem}\label{global result dim 1c}
Consider problem (\ref{quenching problem}) with $p>0$ and $\Omega=(-R,R)$. 
Assume 
\begin{equation}\label{Th 6.2 hypothesis}
\mu > \max \left\{\mu_0(p), \frac{\arctan^2(\sqrt{p+1})}{p}\right\}, 
\qquad 0<d<\sqrt{\frac{p+1}{p\mu}} < d_0:=R-|x_0|-1.
\end{equation}
If $f$ satisfies
\begin{equation}\label{condition 1 on f thm21}
f\ge \mu \ \ \hbox{in $J:= \bigl[x_0-1,x_0+1\bigr]$}
\ \quad\hbox{and}\ \quad
f<\rho\mu \ \ \hbox{in $D:= \overline\Omega\setminus  \bigl(x_0-1-d,x_0+1+d\bigr)$,}
\end{equation}
then $T_f<\infty$ and there are no touchdown points in $D$.
Here $\rho$ is given by the solution of the following optimization problem:
\begin{equation}\label{defrho1c} 
\rho= \underset{ (\beta,K,\tau, \lambda) \in \mathcal{A}}{\rm Sup}\  
\min \left\lbrace \dfrac12\left(\dfrac{\beta-d}{\beta}\right)^{p+1}S(t_0(\tau),\beta)G^*(\tau,t_0(\tau),\beta,K, \lambda),\
\dfrac{1}{p+1} \dfrac{\tau^{p+1}}{(T-t_0(\tau))\mu},\  \lambda \right\rbrace
\end{equation}
with
\begin{equation}\label{condition on parameters 2c}
\mathcal{A}=
\biggl\{(\beta,K,\tau, \lambda);\ \beta\in (d, d_0),\  K\in (0,p],\ \tau,\lambda \in (0,1),\
K\mu\beta^2 \leq \frac{p(p+2)-K}{p},\ \delta_1+\delta_2 
\leq 1 \biggr\},
\end{equation}
\begin{equation}\label{defG*} 
G^*( \tau,t,\beta,K, \lambda) =  
 \inf_{r\in (r_0,1+\beta)} \Bigl(1+p\mu S(t,\beta)\Lambda(t,r)\Bigr)\,
\dfrac{{\rm erf}\Bigl(\frac{r+1}{2\sqrt{t}}\Bigr)+{\rm erf}\Bigl(\frac{1-r}{2\sqrt{t}}\Bigr)}{W_{\tau,K, \lambda}(r) a_{\beta,K}(r)},
\end{equation}
\vskip .5mm
\begin{equation}\label{defLambda}
\Lambda (t,r) = \frac{1}{2} \int_0^t  \bigl(1-Y(s)\bigr)^{-\frac{p}{p+1}-1}\left[ \text{\rm erf} \left(\frac{r+1}{2\sqrt{ t-s}}\right) + \text{\rm erf} \left(\frac{1-r}{2\sqrt{ t-s}} \right) \right] ds,
\end{equation}
where the function $S$ is defined in (\ref{H and G}) and we set
$Y(s)= S(s,0)\, \text{\rm erf} \left(\frac{1}{\sqrt{s}}\right) \frac{(p+1)\mu}{2}\, s$,
\begin{equation}\label{W q=1}
W_{\tau,K,\lambda} (r) =
K\bigl(1-(1-\tau)\tilde{u}(r)\bigr) + \bigl(1-(1-\tau)\tilde{u}(r)\bigr)^{-p}, 
\end{equation}
\begin{equation}\label{DefTildeu}
\tilde{u}(r) = \dfrac{\lambda \mu}{\|f\|_\infty}+\Bigl(1-\dfrac{\lambda \mu}{\|f\|_\infty}\Bigr) 
\frac{1}{\cosh\bigl(\sqrt{ c_p\|f\|_\infty}\, (r-1-d)_+\bigr)}, 
\end{equation} 
\begin{equation}
a_{\beta,K}(r)= \left\lbrace \begin{array}{ll}
D_1 \cos^\alpha\bigl(A_0\sqrt{K\mu}(r-r_0)\bigr), &  r\in [r_0,r_1), \\
\noalign{\vskip 2mm}
D_2 \cos^{p+1} \left(\sqrt{K\mu}(r-1)+A_3\right), &  r\in [r_1,1], \\
\noalign{\vskip 1mm}
\left(1-\beta^{-1}(r-1)\right)^{p+1}, & r\in (1,1+\beta],
\end{array}\right.
\end{equation}\label{deft0}  
$$t_0(\tau) = \frac{1-\tau^{p+1}}{(p+1)\| f\|_\infty},\qquad \overline{T} = \frac{1}{(p+1)(\mu-\mu_0(p))},
\qquad c_p = \dfrac{(p+1)^{p+1}}{p^p},\qquad  L=p(p+2)-K,$$
$$
A_0 =\sqrt{\textstyle\frac{p(1+K)+K L}{p(1+K)^2}},\ \
A_1 = \arctan\Bigl( \sqrt{\textstyle\frac{p(1+K)}{ L}+K} \Bigr),\ \
A_2 = \arctan\Bigl( \sqrt{\textstyle\frac{p}{ L}}\Bigr),\ \ A_3 = \arctan\Bigl(\textstyle\frac{1}{\sqrt{K\mu\beta^2}}\Bigr),$$
$$\delta_1(K)=\dfrac{A_1}{A_0\sqrt{K\mu}},\quad
\delta_2 (\beta,K) = \dfrac{A_3-A_2}{\sqrt{K\mu}},\quad 
r_0 = 1-\delta_1-\delta_2,\quad r_1=1-\delta_2, \quad \alpha = \frac{p+1}{(1+K)A_0^2},$$
$$D_2= \left(1+\frac{1}{K\mu\beta^2}\right)^{\frac{p+1}{2}}, 
\quad D_1 =  D_2 \frac{D_{11}^{\alpha}}{D_{12}^{p+1}}, \quad
D_{11} = \textstyle{\sqrt{1+K + \frac{p(1+K)}{ L}},} \quad
D_{12} = \textstyle{\sqrt{1 + \frac{p}{ L}}}.$$
\end{theorem}

We will see in our numerical examples that formula \eqref{defrho1c} in practice simplifies 
to $\rho\approx \frac12 G^*$ once the parameters $\tau,\beta,K,\lambda$ have been selected 
(cf.~Remark~\ref{disclimitations}, and see Figure~\ref{FigGr} for a plot of the RHS of \eqref{defG*}
 as a function of $r$).

On the other hand, whereas, in Theorems~\ref{global result dim 1}-\ref{global result dim 1a}, 
excluding touchdown on a single interval required the smallness condition to be only assumed in that interval,
this is no longer the case here, due to additional arguments in the proof
 (see Remark~\ref{global statement remark}(i) for details).
 For simplicity we thus made the smallness assumption in the global form \eqref{condition 1 on f thm21}.

We now give the corresponding global result in the case of multi-bump permittivity profiles
(note that Theorems~\ref{global result dim 1}-\ref{global result dim 1a}
 were applicable to multi-bump profiles as well, in view of their local character).
 
\begin{theorem}\label{global result dim 1c2} 
Consider problem (\ref{quenching problem}) with $p>0$ and $\Omega=(-R,R)$,
let $-R<x_m<\dots<x_0<R$ for some $m\ge 1$, and assume
$$d_0:= \displaystyle\min_{0\le i\le m} R-|x_i|-1>0,\qquad 
d_1:=\displaystyle\min_{0\le i\le m-1}\textstyle \frac12 |x_i-x_{i+1}|-1>0.$$
Then the result of Theorem~\ref{global result dim 1c} remains valid with
$$J:= \displaystyle\bigcup_{0\le i\le m} \bigl[x_i-1,x_i+1\bigr],\qquad
D:= \overline\Omega\setminus  \displaystyle\bigcup_{0\le i\le m} \bigl(x_i-1-d,x_i+1+d\bigr),$$
and $d_0$ replaced with $d_2:=\min(d_0,d_1)$
 in formulas \eqref{Th 6.2 hypothesis} and \eqref{condition on parameters 2c}.
\end{theorem}

\subsection{Weaker conditions on $\mu$ and $d_0$}

We next state a variant of Theorems~\ref{global result dim 1}-\ref{global result dim 1a},
which is valid under less restrictive conditions on $\mu$  and for any $d_0>0$.
It actually allows one to handle values of $\mu$ which are close to the reference pull-in voltage
(cf.~Discussion~\ref{disc1}(e)).
For instance, in the case $p=2$, for which the reference pull-in voltage is known to be comprised between
$8/27 \sim 0.30$ and $\pi^2/27\sim 0.37$,
the following theorem only requires $\mu>\pi^2/27\sim 0.37$, 
whereas Theorems~\ref{global result dim 1}-\ref{global result dim 1a} 
and \ref{global result dim 1c}-\ref{global result dim 1c2} 
respectively required $\mu>2\pi^2/27\sim 0.73$ and $\mu>\pi^2/18\sim 0.55$ 
(along with additional restrictions on $d_0$).
The corresponding optimization problem for $\rho$ is similar to that 
in Theorems~\ref{global result dim 1}-\ref{global result dim 1a} 
(and simpler than in Theorems~\ref{global result dim 1c}-\ref{global result dim 1c2}), 
but it now has four instead of three parameters. 

\begin{theorem}\label{global result dim 1b}

Consider problem (\ref{quenching problem}) with $p>0$ and $\Omega=(-R,R)$. 	
Let either $x_0\in \Omega$ or $x_0, x_1\in \Omega$ with $|x_0|\ge |x_1|$,
and assume
\begin{equation}\label{extra conditions b1}
\mu > \mu_0(p)=\dfrac{p^p}{(p+1)^{p+1}}\dfrac{\pi^2}{4},\qquad d_0:=R-|x_0|-1>0.
\end{equation}
Then the conclusions of Theorems~\ref{global result dim 1}-\ref{global result dim 1a} remain valid, where
$\rho$ can now be chosen as the solution of the following optimization problem.
\begin{equation}\label{defrho1b}
\rho= \underset{ ( \tau,\beta,K,\eta) \in \mathcal{A}}{\rm Sup}\  
\min \left\lbrace \rho_1( \tau,\beta,K,\eta),\dfrac{1}{p+1} \dfrac{\tau^{p+1}}{(T-t_0(\tau))\mu} \right\rbrace
\end{equation}
where
$$
\rho_1( \tau,\beta,K,\eta) =  \dfrac12\left( \dfrac{\beta-d}{\beta}\right)^{p+1}
 \min \biggl\{ \dfrac{S(\overline T,0)}{K+\eta^{-p}} G(\overline T, \beta,K,\eta), 
\dfrac{S(t_0,\beta)}{K+\tau^{-p}}
\min \Bigl[  H\bigl(t_0,\beta\bigr), G\bigl(t_0,\beta,K,\eta\bigr)\Bigr] \biggr\},
$$
\begin{equation}\label{condition on parameters 2}
\mathcal{A}=
\biggl\{( \tau,\beta,K,\eta)\in (d,d_0)\times(0,\infty)\times (0,1)^2, \ K\ge \frac{p\eta}{\mu\beta^2}-\frac{1}{(p+1)\eta^p},
\ \delta(\beta,K,\eta)\leq 1 \biggr\}.
\end{equation}
Here we set
$$t_0(\tau) = \textstyle\frac{1-\tau^{p+1}}{(p+1)\| f\|_\infty},\ \overline{T} = \textstyle\frac{1}{(p+1)(\mu-\mu_0(p))}, \ \,
L=1+(p+1)K\eta^p,\ \, \Gamma=\sqrt{\textstyle\frac{(p+1)\eta L}{pK\mu\beta^2}},\ \,
 A = \arctan\Gamma,\ \, \alpha =1+ \textstyle\frac{p}{L},$$
the function $H$ and $S$ are defined by (\ref{H and G}) and the functions $G, \delta$ are defined by
\begin{equation}\label{H and G2}
\begin{array}{ll}
&\delta(\beta,K,\eta) = A(1+K\eta^p)\sqrt{\frac{(p+1)\eta}{pLK\mu}},\\
\noalign{\vskip 2mm}
&G(t,\beta,K,\eta) = (\Gamma^2+1)^{-\alpha/2}\ \underset{0<x<1}{\inf} \dfrac{\text{\rm erf} \left( \frac{2-(1-x)\delta}{2\sqrt{t}}\right) + 
\text{\rm erf} \left( \frac{(1-x)\delta}{2\sqrt{t}}\right) }{\cos^\alpha (Ax)}.
\end{array}
\end{equation}
\end{theorem}

 In Table~\ref{tableex2} below, for the physical case $p=2$, we present some examples of numerical lower estimates of the threshold ratio $\rho$, 
based on Theorem~\ref{global result dim 1b}, for values of $\mu$ close to the reference pull-in voltage
 ($0.37<\mu<0.73$), for which Theorems~\ref{global result dim 1}-\ref{global result dim 1a} are not applicable.

\begin{table}[h]
\begin{center}
\begin{tabular}{|cccc|c|}
\hline
$\mu$ & $\|f\|_\infty$   & $d$   & $d_0$   & $\overline{\rho}$   \\
\hline   
0.7   & 0.8              & 0.01  & 8       & \textbf{0.0815}                  \\
\hline
0.6     & 0.65       & 0.05  & 10    & \textbf{0.0714}                  \\
\hline
0.5     & 0.6          & 0.001 & 6     &   \textbf{0.0137}                  \\
\hline
0.5     & 0.5            & 0.01 & 7    & \textbf{0.0228}                  \\
\hline
\end{tabular}
\end{center}
\caption{Examples for Theorem~\ref{global result dim 1b} with $p=2$.}
\label{general approx table intro}
\label{tableex2}
\end{table}


\section{Type~I estimate and auxiliary optimization problem}

Following the approach in~\cite{GS} and \cite{ES2}, a key ingredient in 
the proofs of Theorems~\ref{global result dim 1}, \ref{global result dim 1a} and
\ref{global result dim 1c}-\ref{global result dim 1b} is the following type~I estimate for $u$ away from the boundary, which we here refine in a nontrivial way in order to allow good quantitative estimates.

\subsection{The estimate and the auxiliary optimization problem for Theorems~\ref{global result dim 1} and \ref{global result dim 1a}}

\begin{proposition}\label{type 1 quantitative}
Let $R>1$, $x_0\in\Omega=(-R,R)$ with $d_0:=R-1-|x_0|>0$, and $\mu > \mu_0(p)$.
Assume that $f$ satisfies 
\begin{equation}\label{condition 3 on f}
f\ge \mu  \quad\hbox{ in $(x_0-1,x_0+1)$}
\end{equation}
and let $u$ be the solution of problem~(\ref{quenching problem}).
Let $d\in (0,d_0)$, $\tau\in(0,1)$.
Then the touchdown time $T$ of $u$ verifies 
\begin{equation} \label{def-t0Tbar}
 T>t_0:=\dfrac{1-\tau^{p+1}}{(p+1)\| f\|_\infty} 
\end{equation}
and $u$ satisfies the  type~I estimate
\begin{equation} \label{typeIest}
\big[1 - u\big(t,x_0\pm(1+d)\big)\big]^{p+1} \geq  (p+1) \overline\eps \mu (T-t)
\quad\hbox{ for all $t\in [t_0,T)$,}
\end{equation}
where
\begin{equation}\label{def-epsilonbar}
\overline\eps= \sup_{(\beta,K)\in\mathcal{A}_1}\  \hat\eps(\beta,K),\qquad
\hat\eps(\beta,K) = \dfrac{1}{2}\left(\dfrac{\beta-d}{\beta}\right)^{p+1}
 \dfrac{S(t_0,\beta)}{K+\tau^{-p}} \min \bigl\{ H(t_0,\beta), G(t_0,\beta,K) \bigr\},
\end{equation}
\begin{equation}\label{defCalA2}
\mathcal{A}_1:=
\left\{ (\beta,K)\in (d,d_0)\times (0,\infty),\ \ K\ge \frac{p}{\mu\beta^2}-\frac{1}{p+1},
\ \ \delta(\beta,K) \leq 1\right\},
\end{equation}
and $H,G,S,\delta$ are defined in (\ref{H and G}).
\end{proposition}

\begin{remark}
 Proposition~\ref{type 1 quantitative} (and the analogous Propositions~\ref{type 1 quantitative1} and \ref{type 1 quantitative1b} below), 
are of course useful only if $\mathcal{A}_1\neq\emptyset$
(since otherwise $\overline\eps=-\infty$ and (\eqref{typeIest}) is void). The condition $\mathcal{A}_1\neq\emptyset$
 will be checked when we apply these propositions in the proofs of 
  Theorems~\ref{global result dim 1}, \ref{global result dim 1a} and \ref{global result dim 1c}-\ref{global result dim 1b}.
 \end{remark}

For the proof of Proposition~\ref{type 1 quantitative} (and of Propositions~\ref{type 1 quantitative1} and \ref{type 1 quantitative1b}),
our strategy is to use a parametri\-zed auxiliary function of the form 
\begin{equation}\label{defJLambda}
J(t,x) = u_t - \eps \mu a_{\beta,K,\eta}(x -x_0)h(u)
\quad\hbox{ in $\Sigma := [t_0, T)\times I_\beta$}
\end{equation}
\begin{equation}\label{defhKq}
h(u)=(1-u)^{-p} + K(1-u)^{q},
\end{equation}
(cf.~Lemma~\ref{basic computation}), 
where $I_\beta$ is a subinterval of $\Omega$, 
namely $I_\beta:= [ x_0-1-\beta, x_0+1+\beta]$ with $\beta>d$, and $a_{\beta,K,\eta}$ 
is a suitable family of  cut-off functions with $a(1+\beta)=0$ (see in particular Figure~\ref{fig3}).
We shall assume for simplicity that $a$ is an even function, with 
\begin{equation}\label{amonotone}
\hbox{$a>0$ and $a'\le 0$ on $[0,1+\beta)$.}
\end{equation}
A key feature in order to reach good values of the threshold ratio $\rho$
is the possibility to optimize with respect to the various parameters which appear in the function $J$, namely:
$$\beta, t_0, q, K, \eps,$$
as well as $\eta$, which defines the subregions of $\Sigma$ in (\ref{DefSubsetsSigma}).

\subsection{Basic computation for the function $J$}

The basic computation for the function $J$ is contained in the following lemma.
This computation was already done in~\cite{GS} and \cite{ES2} for specific choices of the parameters $q$ and $K$
 (and we recall that earlier versions of functions of type $J$, without cut-off and perturbation terms,  
go back to \cite{Sp}, \cite{FM}, \cite{G1}).
In this paper, varying these parameters will be useful in the proof of Propositions~\ref{type 1 quantitative}, \ref{type 1 quantitative1} and \ref{type 1 quantitative1b}.

\begin{lemma}\label{basic computation}
Let $\omega$ be a subdomain of $\Omega$ and let 
$a\in C^2(\omega)$ be a positive function.
Let $u$ be the solution of  (\ref{quenching problem}), $t_0\in [0,T)$,
and let $J$ be given by (\ref{defJLambda}) in $(t_0,T)\times\omega$, 
where 
\begin{equation}\label{function h}
h(u) = (1-u)^{-p} + K(1-u)^{q},\quad\  0\le u<1, 
\end{equation}
with $q\in [0,1]$ and 
\begin{equation}\label{basic computation 1a}
0<K<\dfrac{p(p+1)}{q(1-q)}.
\end{equation}
Then 
\begin{equation}\label{basic computation 1}
J_t - J_{xx} - pf(x)(1-u)^{-p-1}J = \mu \eps  \Theta\quad\hbox{ in $(t_0,T)\times \omega$,} 
\end{equation}
where
\begin{equation}\label{basic computation 2}
\Theta = (p+q)K a(x) f(x) (1-u)^{-p+q-1} + ah''(u) u_x^2 + 2h'(u) a_x u_x + h(u) a_{xx}.
\end{equation}
Moreover, we have  $h''(u)>0$ for all $u\in [0,1)$ and
\begin{equation}\label{basic computation 3}
\Theta \geq \underbrace{(p+q) K a(x) f(x) (1-u)^{-p+q-1}}_{\tau_1} + \underbrace{h(u) a_{xx}(x)}_{\tau_2}
 - \underbrace{\dfrac{h'^2(u) a_x^2(x)}{a(x)h''(u)}}_{\tau_3}.
\end{equation}
\end{lemma}

\begin{proof}
We compute
\begin{eqnarray*}
J_t &=& u_{tt} -\eps  a(x)h'(u)u_t, \\
J_x &=&  u_{xt} - \eps \mu \big( a(x)h'(u)u_x + h(u) a_x(x)\big), \\
J_{xx} &=&  u_{xxt} - \eps  \mu \big( a(x)h'(u) u_{xx} + a(x)h''(u) u_x^2 
  + 2h'(u) a_x(x) u_x + h(u) a_{xx}(x) \big).
\end{eqnarray*}
Setting $g(u) = (1-u)^{-p}$ and omitting the variables $x,u$ without risk of confusion, we get
\begin{eqnarray*}
J_t -  J_{xx} &=& (u_t- u_{xx})_t - \eps \mu  ah'(u_t- u_{xx}) 
 +\eps \mu (ah''  u_x^2  + 2h' a_x u_x+ h a_{xx}) \\
               &=& fg'u_t - \eps \mu fah'g + \eps \mu (ah''  u_x^2 + 2h' a_x u_x + h a_{xx}).
\end{eqnarray*}
Using $u_t=J+\eps \mu ah$, we have
$$J_t- J_{xx} - fg'J = \eps \mu \Theta,$$
where
$$\Theta=fa(g'h-h'g)+ah''  u_x^2 + 2h'(u) a_x u_x + h a_{xx}.$$
On the other hand, we have
\begin{equation}\label{hprimeu}
 h'(u)=p(1-u)^{-p-1} - qK(1-u)^{q-1},
\end{equation}
hence
\begin{eqnarray*}
g'h - h'g &=& p(1-u)^{-p-1} [(1-u)^{-p} + K(1-u)^{q}] - (1-u)^{-p} [p (1-u)^{-p-1} - qK(1-u)^{q-1}] \\
          &=& (p+q)K(1-u)^{-p+q-1},   
\end{eqnarray*}
which yields (\ref{basic computation 2}). Also, owing to (\ref{basic computation 1a}), we have
\begin{equation}\label{hprimeprimeu}
\begin{array}{ll}
h'' &= [p(p+1) - q(1-q)K(1-u)^{p+q}](1-u)^{-p-2}
    \geq  [p(p+1)-q(1-q)K](1-u)^{-p-2} > 0.
\end{array}
\end{equation}
Finally, since $a>0$  in $\omega$, we may write
$$\Theta = (p+q)Ka f (1-u)^{-p+q-1} + h a_{xx} + ah'' \left[ u_x^2 + 2\dfrac{h' a_x u_x}{ah''}\right].$$
Since $u_x^2 + 2\dfrac{h' a_x u_x}{ah''}\geq -\dfrac{h'^2  a_x^2}{a^2(h'')^2}$, inequality~(\ref{basic computation 3}) follows.
\end{proof}

\begin{remark}
(a) We observe that no loss of information seems to occur from  inequality~(\ref{basic computation 3}).
Indeed, by (\ref{basic computation 2}), this inequality becomes an equality at any point $x$ such that
$ u_x+\frac{h' a_x}{h''a}=0$ i.e., $[\log (ah'(u)]_x=0$.
But since, in order to apply the maximum principle in the proofs of Proposition~\ref{type 1 quantitative} (and Propositions~\ref{type 1 quantitative1} and \ref{type 1 quantitative1b} below), 
the function $a$ will be required to vanish on $\partial\omega$,
such points $x$ must exist for each $t\in (0,T)$.

(b) The restriction $q\le 1$ is necessary to guarantee that the key term $h''(u) u_x^2$ in (\ref{basic computation 2}) remains positive.
Similarly, the positivity of the key term $\tau_1$ in (\ref{basic computation 3}) imposes $p+q>0$.
Although the values $q\in (-p,0)$ would be also admissible, we shall not consider them.
Indeed, when looking for quantitative estimates in  Section~6, they seem to lead to worse results due to smaller constant $p+q$
(and to more complicated expressions than $q=0$ or $q=1$).
\end{remark}

 We also recall the following simple lemma, that will be used in the sequel.

\begin{lemma}\label{estimates for T} 
Let $u$ be the solution of (\ref{quenching problem}). 

(i) We have $T\geq T_*:= \frac{1}{(p+1)\| f\|_{\infty}}$ and
\begin{equation}\label{compODE}
\|u(t)\|_{\infty} \leq y(t):= 1-\bigl(1-(p+1)\|f\|_\infty t\bigr)^{\frac{1}{p+1}},\quad \hbox{ for all } t\in[0,T_*].
\end{equation}

(ii) Assume that $I:=(x_0-1,x_0+1)\subset\Omega$ and $f\geq \mu\chi_I$, with 
$\mu>\mu_0(p):=\frac{p^p}{(p+1)^{p+1}}\frac{\pi^2}{4}$.
Then 
$$T\leq \overline T:= \dfrac{1}{(p+1)(\mu - \mu_0(p))}<\infty.$$
\end{lemma}

\begin{proof}
Since $y(t)$ is the solution of the ODE
\begin{equation}\label{ODE}
y'(t) = \|f\|_\infty(1-y(t))^{-p}, \quad t\in (0,T_\ast), 
\qquad\hbox{ with $y(0)=0$,}
\end{equation}
and $T_\ast$ is the maximal existence time for $y(t)$,
assertion (i) follows immediately from the comparison principle.

Assertion (ii) follows from a simple eigenfunction argument, see e.g., Lemma~2.2 in~\cite{ES2}.
\end{proof}

\subsection{Construction of the family of  cut-off functions $a(x)$ and parametrized type~I estimate}

The function $J$ needs to satisfy a basic parabolic inequality (cf.~\eqref{Sigma2eta1a} below).
In one space dimension, the study of this parabolic inequality can be made quite precise.
It actually leads to the following, natural and optimal, differential inequality for the function~$a(r)$:
\begin{equation}\label{diffineqa1}
a''(r) \geq a(r)\,F\Bigl(r,\dfrac{a'(r)}{a(r)}\Bigr),\qquad 0\le r<1+\beta,
\end{equation}
where
\begin{equation}\label{diffineqa2}
F(r,\xi)=
\left\lbrace \begin{array}{ll}
\displaystyle\sup_{u\in (0,1)}\left[\dfrac{h'^2(u)}{hh''(u)}\right]\,\xi^2,&r>1, \\
\displaystyle\sup_{u\in (1-\eta,1)}\left[\dfrac{h'^2(u)}{hh''(u)}\,\xi^2 - \dfrac{(p+q)K\mu}{(1-u)^{p+1-q}h(u)}\right],&r<1.
\end{array}\right.
\end{equation}
This is the contents of the following lemma which
gives a family of type~I estimates, corresponding to each admissible value of the parameters.
We note that only the choice $q=0$ and $\eta=1$ will be used in the proof 
of Proposition~\ref{type 1 quantitative} and Theorems~\ref{global result dim 1}-\ref{global result dim 1a}.
Other values of the parameters $q, \eta$ will be used in the  proofs of the results of Section~\ref{further results}.

\begin{lemma}\label{control Theta}
Let $R>1$, $x_0\in\Omega=(-R,R)$ with $d_0:=R-1-|x_0|>0$,
and $\mu > \mu_0(p)$.
Assume that $f$ satisfies \eqref{condition 3 on f}
and let $u$ be the solution of problem~(\ref{quenching problem}).
Let 
\begin{equation}\label{condparamcontrolTheta}
q\in [0,1],\quad \tau\in (0,1),\quad\eta\in (0,1],\quad \beta\in (0,d_0), \quad K>0
\end{equation}
and let $h(u)$ be defined by \eqref{defhKq}.
Set $I_0=(x_0-1,x_0+1)$, $I_\beta=(x_0-1-\beta, x_0+1+\beta)$ and
$$t_0 = t_0(\tau) = \dfrac{1-\tau^{p+1}}{(p+1)\| f\|_\infty}.$$
Assume that there exists a solution $a\in  W^{2,2}([0,1+\beta])$ of \eqref{diffineqa1}-\eqref{diffineqa2},
with \eqref{amonotone} and the boundary conditions 
\begin{equation}\label{diffineqa3}
a'(0)=0,\quad a(1+\beta)=0.
\end{equation}
Assume also that $\eps>0$ satisfies
\begin{equation}\label{epsilon 1b}
\eps \le \eps_1:=
\inf_{x\in I_\beta} \dfrac{e^{t_0\Delta_\Omega}\chi_{I_0}(x)}{h(u(t_0,x)) a(|x-x_0|)}
\end{equation}
where $e^{t\Delta_\Omega}$ denotes the Dirichlet heat semigroup on $\Omega$, and 
\begin{equation}\label{epsilon 1}
\eta=1\quad\hbox{or}\quad\eps \le \eps_2:=
\dfrac{1}{W(q,\eta,K)} \inf_{(t,x)\in [t_0,\overline T]\times I_0} \dfrac{e^{t\Delta_\Omega}\chi_{I_0}(x)}{a(|x-x_0|)},
\end{equation}
with
 $\overline T = \frac{1}{(p+1)(\mu - \mu_0(p))}$ and $W(q,\omega,K) := \displaystyle\sup_{u\in (0,1-\omega)} h(u)$.
Then
\begin{equation}\label{estTypeILemma}
(1-u(t,x))^{p+1} \geq (p+1) {\hskip 0.5pt}\eps {\hskip 0.5pt}\mu{\hskip 0.5pt} a(|x-x_0|) (T-t) \quad\hbox{ in $[t_0,T)\times I_\beta$}. 
\end{equation}
\end{lemma}

 By Lemma~\ref{estimates for T}(i) we have $\|u(t_0)\|_\infty \leq 1-\tau$, so we can estimate
\begin{equation}\label{estimate h t0}
h(u(t_0,x)) \leq W(q,\tau,K).
\end{equation}
We use this estimate in Theorems~\ref{global result dim 1}, \ref{global result dim 1a}, \ref{global result dim 1b} for simplicity. However, in Theorems~\ref{global result dim 1c}-\ref{global result dim 1c2}
 we will use a better upper estimate of $u$ by taking advantage of the smallness hypothesis in \eqref{condition 1 on f thm21}.

By the monotonicity properties of the function $h(u)$,  we note that if $p\geq qK$, then
\begin{equation}\label{Wcalcul}
W(q,\omega,K) =  h(1-\omega)=\omega^{-p}+K \omega^q,\quad \omega\in (0,1).
\end{equation}
This will be the case in the proofs of Theorems~\ref{global result dim 1}, \ref{global result dim 1a} and \ref{global result dim 1b} 
since we are considering $q=0$. 
For Theorems~\ref{global result dim 1c}-\ref{global result dim 1c2} we will also restrict ourselves to this case for simplicity and 
since we have observed numerically that the optimal choice of $K$ is less than $p$.

\begin{proof} Set $\Sigma:= [t_0, T) \times I_\beta$ and let 
$$J(t,x) = u_t -\eps \mu b(x)h(u)\quad\hbox{ in $\Sigma$,}$$
where 
$$h(u)=(1-u)^{-p} + K(1-u)^{q},\qquad b(x)=a(x-x_0).$$
We split the cylinder $\Sigma$ into three subregions as follows:
\begin{equation}\label{DefSubsetsSigma}
\begin{array}{lll}
\Sigma_1    &=& [t_0, T)\times (I_\beta\setminus \bar I_0), \\  
\noalign{\vskip 1mm}
\Sigma_2^\eta &=& \lbrace (t,x)\in [t_0,T) \times  \bar I_0; \ u(t,x)\geq 1-\eta \rbrace , \\
\noalign{\vskip 1mm}
\Sigma_3^\eta &=& \lbrace (t,x)\in [t_0,T) \times  \bar I_0; \ u(t,x)< 1-\eta \rbrace.
\end{array}
\end{equation}

By Lemma~\ref{basic computation} (still valid for $a\in  W^{2,2}([0,1+\beta])$), we have
\begin{equation}\label{Sigma2eta1a0}
J_t -  J_{xx} - pf(x)(1-u)^{-p-1}J \ge \eps \tilde \Theta\quad\hbox{ a.e. in $(t_0,T)\times I_\beta$,} 
\end{equation}
where
\begin{equation}\label{Sigma2eta1a01}
\tilde \Theta=(p+q)Kf(x) b(x) (1-u)^{-p+q-1} +   b''(x)h - \dfrac{|b'(x)|^2h'^2}{ b(x)h''}.
\end{equation}
Since our only lower assumption on $f$ is $f\ge \mu\chi_{ I_0}$ 
the property $\tilde\Theta\geq 0$ a.e. in $\Sigma_1$ (resp. $\Sigma_2^\eta$) amounts to requiring
\begin{equation*}
\dfrac{(p+q)K\mu\chi_{(-1,1)}( r)}{(1-u)^{p+1-q} h(u)} +\dfrac{a''(r)}{a(r)}  
\ge \left[\dfrac{a'(r)}{a(r)}\right]^2 \dfrac{{h'}^2(u)}{hh''(u)}
\end{equation*}
for a.e. $( r,u)\in (1,1+\beta)\times (0,1)$ (resp., $[0,1)\times [1-\eta,1)$).
Since we assumed that $a$ solves \eqref{diffineqa1}-\eqref{diffineqa2}, this precisely guarantees that
\begin{equation}\label{Sigma2eta1a}
J_t - J_{xx} - p f(x) (1-u)^{-p-1} J \geq  0 \quad \hbox{ a.e. in $\Sigma_1\cup\Sigma_2^\eta$.}
\end{equation}

On the other hand, we claim that
\begin{equation}\label{low bound u_t}
u_t\geq \mu e^{t\Delta_\Omega} \chi_{I_0} \qquad \text{in } [0,T)\times \Omega.
\end{equation}
The claim follows from the comparison principle applied to the function $v=u_t$, which solves the problem
\begin{equation}\label{eqnut}
\left\lbrace\begin{array}{ll}
v_t -  v_{xx} = pf(x) (1-u)^{-p-1}v, & \mbox{ in } (0,T)\times \Omega, \\
v=0, & \mbox{ in } [0,T)\times \partial\Omega, \\
v(0,x)=f(x), &\mbox{ in } \Omega
\end{array}\right.
\end{equation}
(see Lemma~3.4 in~\cite{ES2}). 
In view of \eqref{epsilon 1b}, \eqref{low bound u_t}, we have
\begin{equation}\label{J for t_0}
J(t_0,x) \geq \mu e^{t_0\Delta_\Omega}\chi_{I_0}(x) - \eps \mu h(u(t_0,x))a(x-x_0)\geq 0 \quad \mbox{ in } I_\beta.
\end{equation}

Next observe that if $\eta=1$, then the subregion $\Sigma_3^\eta$ is empty.
If $\eta\in (0,1)$, using assumption~(\ref{epsilon 1}) and $T\le \overline T$  by Lemma~\ref{estimates for T}(ii),
we have
\begin{equation}\label{J in sigma 3}
J(t,x) = u_t -\eps \mu b(x) h(u) 
      \geq  \mu e^{t\Delta_\Omega}\chi_{I_0}(x) - \eps \mu a(x-x_0)W(q,\eta,K)\geq 0 \quad\mbox{ in } \Sigma_3^\eta.
\end{equation}
Now, as a 
consequence of (\ref{J in sigma 3}) and $\Sigma = \Sigma_1\cup \Sigma_2^\eta \cup \Sigma_3^\eta$, we have
\begin{equation}\label{J<0 set}
\lbrace (t,x)\in \Sigma ; \quad J(t,x)<0\rbrace \subset \Sigma_1\cup \Sigma_2^\eta.
\end{equation}
Also, since $b=0$ on $\partial I_\beta$, we have
\begin{equation}\label{J on the boundary}
J\geq 0 \quad \mbox{on } [t_0,T)\times \partial I_\beta.
\end{equation}
 On the other hand, using standard parabolic regularity, we observe that 
$J\in C^1\big([t_0,T)\times \overline{I_\beta}\big)$, with $J_{xx}\in L^2(\Sigma)$.
It then follows from (\ref{Sigma2eta1a}), (\ref{J for t_0}), (\ref{J<0 set}), (\ref{J on the boundary}) and the maximum principle 
(see, e.g., \cite[Proposition 52.8 and Remark 52.11(a)]{QS}) that
$$J\geq 0 \quad \mbox{ in } \Sigma.$$
Integrating in time 
we obtain
$$(1-u(t,x))^{p+1} \geq (p+1) \eps \mu b(x) (T-t) \quad \mbox{ in } \Sigma,$$
which concludes the proof of the lemma.
\end{proof}

Our next task is to identify appropriate solutions of the differential inequality  \eqref{diffineqa1}-\eqref{diffineqa2}.
First of all, without loss of generality, we may assume the normalization condition
\begin{equation}\label{diffeqa4}
a(1)=1.
\end{equation}
Indeed, the inequality \eqref{diffineqa1}, the boundary conditions \eqref{diffineqa3} and 
(owing to assumptions~(\ref{epsilon 1b}), (\ref{epsilon 1})) the estimate~(\ref{estTypeILemma}) 
are not affected by multiplication of $a$ by a positive constant.

Now, it can be shown (see Proposition~\ref{optimal prolongation lemma}) that, 
among all possible solutions of \eqref{diffineqa1}, \eqref{diffineqa2}, \eqref{diffineqa3},
the optimal choice is to actually look for a solution of the corresponding ODE:
\begin{equation}\label{diffeqa1}
a''(r) = a(r)\,F\Bigl(r,\dfrac{a'(r)}{a(r)}\Bigr),\quad r_0\le r<1+\beta,
\end{equation}
for some $r_0\in [0,1)$, with boundary conditions
\begin{equation}\label{diffeqa2}
a'(r_0)=0,\quad a(1+\beta)=0,
\end{equation}
and to extend it to be constant on the remaining part of the interval:
\begin{equation}\label{diffeqa3}
a(r)=a(r_0),\quad 0\le r\le r_0.
\end{equation}
Indeed, Proposition~\ref{optimal prolongation lemma} shows that, fixing the reference value in \eqref{diffeqa4},
the ratios $\eps_1, \eps_2$ in (\ref{epsilon 1b}), (\ref{epsilon 1}) are largest when the function $a(r)$ is chosen in this way.
In order not to interrupt the main line of argument, Proposition~\ref{optimal prolongation lemma}
 and its proof are postponed to Appendix 2.

Next, it turns out that, for the special values $q=0$ and $q=1$, the solution of  \eqref{diffeqa4}-\eqref{diffeqa3} can be explicitly computed
whenever it exists. We start with the case $q=0$
(keeping $\eta\in (0,1]$ for future use in the proof of Theorem~\ref{global result dim 1b}.)
As for the more complicated case $q=1$, it will be studied in Section~\ref{construction a q=1} for the proof of  Theorems~\ref{global result dim 1c}-\ref{global result dim 1c2}.

\begin{lemma}\label{Explicita0}
Let $\beta>0$, $q=0$, $\eta\in (0,1]$, $K>0$ and assume
\begin{equation}\label{assumption a q=0}  
K\ge \frac{p\eta}{\mu\beta^2}-\frac{1}{(p+1)\eta^p}
\end{equation}
Let $h, F$ be given by \eqref{defhKq}, \eqref{diffineqa2} and set
$$m=\frac{p}{(p+1)(1+K\eta^p)},\quad M=\frac{pK\mu}{\eta (1+K\eta^p)},\quad
\delta_0
=\frac{\arctan\left(\textstyle \frac{p+1}{\beta} \sqrt{\frac{1-m}{M}} \right)}{\sqrt{M(1-m)}}.$$
There exist $r_0\in [0,1)$ and a solution $a\in  W^{2,2}([0,1+\beta])$ of \eqref{diffeqa4}-\eqref{diffeqa3}, \eqref{amonotone} 
if and only if $\delta_0\le 1$. The couple $(r_0,a)$ is then unique and it is given by
\begin{equation}\label{function a q=0}  
a_{\beta,K,\eta}(r)= \left\lbrace \begin{array}{ll}
D, &  r \in [0, r_0), \\
\noalign{\vskip 1mm}
D\cos^{\frac{1}{1-m}} \bigl(\sqrt{M(1-m)} (r-r_0)\bigr), &  r\in [r_0,1], \\
\noalign{\vskip 1mm}
\left(\dfrac{1+\beta -r}{\beta}\right)^{p+1}, & r\in (1,1+\beta],
\end{array}\right.
\end{equation}
where $r_0 = 1-\delta_0$ and $D=\Bigl[1+\frac{1-m}{M}\left(\frac{p+1}{\beta}\right)^2\Bigr]^{\frac{1}{2(1-m)}}$.
\end{lemma}

\begin{figure}
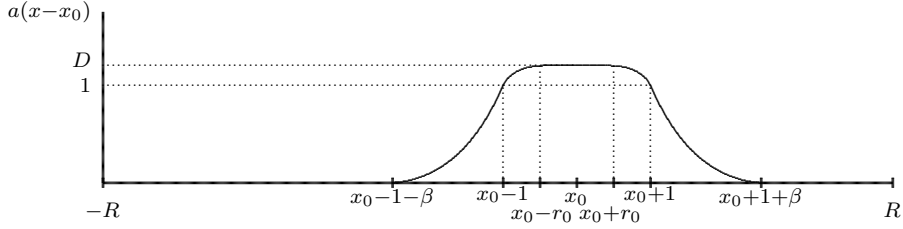

\[
\beginpicture
\setcoordinatesystem units <0.7cm,0.65cm>
\setplotarea x from -10 to 6, y from -1 to 4

\setdots <0pt>
\linethickness=1pt
\putrule from -9 0 to 6 0
\putrule from -9 0 to -9 3.5

\linethickness=1pt
\putrule from -9 -0.1 to -9 0.1
\putrule from 6 -0.1 to 6 0.1
\putrule from -1.4 -0.1 to -1.4 0.1
\putrule from 1.4 -0.1 to 1.4 0.1
\putrule from -.7 -0.1 to -.7 0.1
\putrule from .7 -0.1 to .7 0.1
\putrule from -3.5 -0.1 to -3.5 0.1
\putrule from 3.5 -0.1 to 3.5 0.1

\putrule from 0 -.1 to 0 .1
\put {$_{R}$} [ct] at 6 -.4
\put {$_{-R}$} [ct] at -9 -.4
\put {$_{x_0}$} [ct] at 0 -.2
\put {$_{x_0-1}$} [cb] at -1.4 -.4
\put {$_{x_0+1}$} [cb] at 1.4 -.4
\put {$_{x_0-1-\beta}$} [cb] at -3.5 -.5
\put {$_{x_0+1+\beta}$} [cb] at 3.5 -.5
\put {$_{x_0-r_0}$} [cb] at -.65 -.8
\put {$_{x_0+r_0}$} [cb] at .65 -.8
\put {$_{a(x-x_0)}$} [rc] at -9.2 3.5
\put {$_{1}$} [rc] at -9.2 2
\put {$_{D}$} [rb] at -9.2 2.38

\setquadratic
\plot 1.4 2  2.3 0.55  3.5 0 /
\plot 0.6 2.4  1.1 2.3  1.4 2 /

\plot -1.4 2  -2.3 0.55  -3.5 0 /
\plot -0.6 2.4  -1.1 2.3  -1.4 2 /

\setlinear
\plot -0.6 2.4  0.6 2.4 / 

\setdots <2pt>
\setlinear 
\plot -1.4 0  -1.4 2 /
\plot 1.4 0  1.4 2 /
\plot -.7 0  -.7 2.4 /
\plot .7 0  .7 2.4 /
\plot -9 2  1.4 2 /
\plot -9 2.4  0 2.4 /

\endpicture
\]
\caption{The function $a_{\beta,K,\eta}(x-x_0)$.} 
\label{a(x) figure}
\end{figure}


\begin{proof} 
\textit{Step 1: Determination of $a$ in $[1,1+\beta]$.} 
Since
$$hh'' =p(p+1)(1-u)^{-2p-2} \bigl[1 + K(1-u)^{p}\bigr], \quad h'^2=p^2(1-u)^{-2p-2},$$
we have
\begin{equation}\label{condaexterior}
\displaystyle\sup_{u\in (0,1)}\left[\dfrac{h'^2(u)}{hh''(u)}\right]=\dfrac{p}{p+1}.
\end{equation}
We are thus left with the ODE
$$aa''=\dfrac{p}{p+1}{a'}^2,$$
with boundary conditions $a(1)=1$, $a(1+\beta)=0$.
It is easy to solve this by setting $\phi=a^{1-\sigma}$, $\sigma=p/(p+1)\in (0,1)$, hence $\phi'=(1-\sigma)a^{-\sigma}a'$ and
$$\phi''=(1-\sigma)a^{-\sigma-1}[aa''-\sigma(a')^2]=0.$$
This leads to $\phi(r)=(1+\beta-r)/\beta$, hence $a(r)=\phi^{p+1}(r)$ given by the last part of~(\ref{function a q=0}).

\textit{Step 2: Determination of $a$ in $[0,1]$ and existence condition.} 
For $r\in [0,1]$, we compute
\begin{eqnarray*}
F(r,\xi)
&=&\displaystyle\sup_{u\in (1-\eta,1)}\left[\dfrac{h'^2(u)}{hh''(u)}\,\xi^2 - \dfrac{pK\mu}{(1-u)^{p+1}h(u)}\right] 
=\displaystyle\sup_{y\in (0,\eta)}\left[\dfrac{p\xi^2 }{(p+1)\bigl(1 + Ky^p\bigr)}- \dfrac{pK\mu}{y\bigl(1+ Ky^p\bigr)}\right] \\
&=& \displaystyle\sup_{y\in (0,\eta)} Z(\xi,y),\quad\hbox{ where }Z(\xi,y)=\Bigl(\dfrac{\xi^2 }{p+1}- \dfrac{K\mu}{y}\Bigr)\dfrac{p}{1+ Ky^p}.
\end{eqnarray*}
Computing
$$\frac{\partial Z}{\partial y}
=\dfrac{K\mu}{y^2}\dfrac{p}{1+ Ky^p}
-\Bigl(\dfrac{\xi^2 }{p+1}- \dfrac{K\mu}{y}\Bigr)\dfrac{p^2Ky^{p-1}}{(1+ Ky^p)^2},$$
we see that 
$$\frac{\partial Z}{\partial y}\ge 0 
\Leftrightarrow\frac{ \mu(1+ Ky^p)}{y}\ge\Bigl(\dfrac{\xi^2}{p+1}-\frac{K\mu}{y}\Bigr)py^{p}
\Leftrightarrow \xi^2\le \dfrac{(p+1)\mu}{py}\Bigl[\frac{1}{y^{p}}+K(p+1)\Bigr].$$
Consequently,  for $r\in [0,1]$, we have
\begin{equation}\label{Frxismall}
F(r,\xi)
=\Bigl(\dfrac{\xi^2 }{p+1}- \dfrac{K\mu}{\eta}\Bigr)\dfrac{p}{1+ K\eta^p},
\quad\hbox{ for }  \xi^2\le \dfrac{(p+1)\mu}{p\eta}\Bigl[\frac{1}{\eta^{p}}+K(p+1) \Bigr].
\end{equation}
Furthermore, owing to \eqref{diffineqa2} and \eqref{condaexterior}, any solution of 
\eqref{diffeqa1} on some interval $[r_0,1]$ with $a>0$ and $a'\le 0$ must satisfy
$$\Bigl(\frac{a'}{a}\Bigr)'=\frac{a''}{a}-\Bigl(\frac{a'}{a}\Bigr)^2 \le 0$$
hence, by the $C^1$ continuity conditions 
\begin{equation}\label{C1contcond}
a(1_-)=1,\qquad a'(1_-)=a'(1_+)=-\frac{p+1}{\beta},
\end{equation}
we must have
\begin{equation}\label{boundaprimea}
\Bigl|\frac{a'(r)}{a(r)}\Bigr|\le \Bigl|\frac{a'(1)}{a(1)}\Bigr|=\frac{p+1}{\beta},\quad r_0\le r\le 1.
\end{equation}
By \eqref{Frxismall}, \eqref{boundaprimea} and assumption \eqref{assumption a q=0}, we are thus left with the ODE
\begin{equation}\label{resolphi00}
aa''=m{a'}^2-Ma^2,\quad r\le 1,
\end{equation}
with 
$$m=\frac{p}{(p+1)(1+K\eta^p)} \in (0,1),\quad M=\frac{pK\mu}{\eta (1+K\eta^p)}.$$
Setting 
$\phi=a^{1-m}$, 
\eqref{resolphi00} is equivalent to
\begin{equation}\label{resolphi0}
\phi''=(1-m)a^{-m-1}[aa''-m(a')^2]=-(1-m)M\phi.
\end{equation}
Since all solutions of \eqref{resolphi0} are given by cosine functions, and since we are looking for a solution $\phi$ such that $\phi'(1)<0$,
$\phi'$ must have a first zero on the left of $r=1$. 
Since we also impose $\phi>0$, $\phi'\le 0$ on $[r_0,1]$  and $\phi'(r_0)=0$, this first zero must coincide with $r_0$ and
the solution must be of the form
\begin{equation}\label{resolphi0b}
\phi(r)=D_0\cos\bigl[\sqrt{(1-m)M}(r-r_0)\bigr],
\end{equation}
for some $D_0>0$, and we must have
\begin{equation}\label{resolphi1}
\sqrt{(1-m)M}(1-r_0)<\pi/2.
\end{equation}
This, along with \eqref{C1contcond}, yields
\begin{equation}\label{resolphi2}
1=\phi(1)=D_0\cos\bigl[\sqrt{(1-m)M}(1-r_0)\bigr]
\end{equation}
and
$$
-(1-m)(p+1)\beta^{-1}=\phi'(1)=-D_0\sqrt{(1-m)M}\sin\bigl[\sqrt{(1-m)M}(1-r_0)\bigr],
$$
hence
\begin{equation}\label{resolphi3}
\tan\bigl[\sqrt{(1-m)M}(1-r_0)\bigr]=\frac{p+1}{\beta}\sqrt{\frac{1-m}{M}}.
\end{equation}
It follows that
\begin{equation}\label{resolphi3b}
 1-r_0=\delta_0:=\frac{1}{\sqrt{M(1-m)}}\arctan\left( \frac{p+1}{\beta} \sqrt{\frac{1-m}{M}} \right)
\end{equation}
(which in particular guarantees \eqref{resolphi1}) and, by \eqref{resolphi2}, \eqref{resolphi3}, we have
$$
D_0=\dfrac{1}{\cos\bigl[\sqrt{(1-m)M}(1-r_0)\bigr]}=\sqrt{1+\tan^2\bigl[\sqrt{(1-m)M}(1-r_0)\bigr]},$$
hence
\begin{equation}\label{resolphi4}
D_0=\sqrt{1+\frac{1-m}{M}\left(\frac{p+1}{\beta}\right)^2}.
\end{equation}

 From the above, we see that the existence of a solution  of \eqref{diffeqa4}-\eqref{diffeqa3}, \eqref{amonotone} 
for some $r_0\in [0,1)$ is equivalent to $\delta_0\le 1$.
Under this condition, we deduce from \eqref{resolphi0b} and \eqref{resolphi4} that 
$a(r)$ is given by (\ref{function a q=0}). .
 Moreover, in view of \eqref{resolphi3b}, the couple $(r_0,a)$ is then unique.
The lemma is proved.
\end{proof}

\subsection{Proof of Proposition~\ref{type 1 quantitative}}

To complete the proof of Proposition~\ref{type 1 quantitative}, it essentially remains to express the infima in 
(\ref{epsilon 1b}) and (\ref{epsilon 1}) in terms of the error function.

This relies on quantitative 
comparison estimates for the Dirichlet heat semigroups,
given in Proposition~\ref{prop semigroup comparison}, combined with the following elementary lemma.

\begin{lemma}\label{lemTransfErr}
Denote by $e^{t\Delta_\mathbb{R}}$ the heat semigroup on $\mathbb{R}$.
Let the functions $G, H, \delta$ be defined by (\ref{H and G}).
Let $\beta, K>0$ satisfy $K\ge \frac{p}{\mu\beta^2}-\frac{1}{p+1}$ and $\delta(\beta,K)\le 1$, and 
let $a(x)$ be defined by (\ref{function a q=0}) with $\eta=1$.
We have
\begin{equation}\label{quotient inside B q=0}
\displaystyle\inf_{0<r<1} \dfrac{e^{t\Delta_\mathbb{R}}\chi_{(-1,1)}(r)}{a(r)}
=\dfrac{1}{2} G(t,\beta,K)
\end{equation}
and
\begin{equation}\label{quotient outside B} 
\displaystyle\inf_{1<r<1+\beta} \dfrac{e^{t\Delta_\mathbb{R}}\chi_{(-1,1)}(r)}{a(r)} 
=\dfrac{1}{2} H(t,\beta),\quad t>0.
\end{equation}
Moreover, 
\begin{equation}\label{Gmonot}
t\mapsto G(t,\beta,K)  \ \hbox{ is nonincreasing for $t>0$.}
\end{equation}
\end{lemma}

\begin{proof}
Denote by $L_1$ the LHS of (\ref{quotient inside B q=0}). 
We note that $a(r)$ is constant in the interval $\left[0, r_0\right)$.
Since $e^{t\Delta_\mathbb{R}}\chi_{(-1,1)}(r)$ is even and monotonically decreasing for $r>0$, we thus need only consider the interval  $\left[r_0,1\right)$.
Using the definition of $a(r)$ in this interval and the heat kernel on the real line, we obtain
$$L_1 = \dfrac{1}{\sqrt{4\pi t}} \inf_{r_0<r<1}\, \dfrac{1}{a(r)}\displaystyle\int_{-1}^1  e^{-\frac{(r-y)^2}{4t}}dy.$$
After the change of variable  $Z = \frac{r-y}{2\sqrt{t}}$, and then $x=\frac{r-r_0}{1-r_0}$, we obtain
\begin{equation*}
L_1 = \dfrac{1}{\sqrt{\pi}} \inf_{r_0<r<1} 
\dfrac{1}{a(r)} \displaystyle\int_{\frac{r-1}{2\sqrt{t}}}^{\frac{r+1}{2\sqrt{t}}} 
e^{-Z^2}dZ 
= \dfrac{1}{2}\inf_{ 0<x<1} 
\dfrac{\mbox{erf} \left( \frac{1+r_0+(1-r_0) x}{2\sqrt{t}}\right) 
+ \mbox{erf} \left(\frac{1-r_0-(1-r_0) x}{2\sqrt{t}}\right)}{a((1-r_0) x+r_0)}.
\end{equation*}
 Comparing \eqref{function a q=0}, where $\eta=1$, with (\ref{H and G}), we deduce
$$L_1=\dfrac{1}{2}\inf_{0<x<1} 
\dfrac{\mbox{erf} \left( \frac{ 2-(1-x)\delta_0}{2\sqrt{t}}\right) 
+ \mbox{erf} \left(\frac{(1-x)\delta_0}{2\sqrt{t}}\right)}{ D \cos^\alpha (Ax)} =\frac12 G(t,\beta,K).$$ 

Now denote by $L_2$ the LHS of (\ref{quotient outside B}). 
Using the definition of $a(r)$ in this interval and the heat kernel on the real line, we find
$$ L_2 = \dfrac{1}{\sqrt{4\pi t}} \inf_{1<r<1+\beta} \left(\dfrac{\beta}{1+\beta-r}\right)^{p+1}\displaystyle\int_{-1}^1 e^{-\frac{(r-y)^2}{4t}}dy.$$
After the changes of variables  $X = \dfrac{r-1}{\beta}$, 
$Y = \dfrac{y-1}{\beta}$, and then $Z=\dfrac{\beta(Y-X)}{2\sqrt{t}}$, we obtain
\begin{eqnarray*}
L_2 & = & \dfrac{1}{\sqrt{4\pi t}} \inf_{0<X<1} \dfrac{\beta\displaystyle\int_{-\frac{2}{\beta}}^0 e^{-\frac{(X-Y)^2}{4t}\beta^2} dY}{(1-X)^{p+1}} 
=  \dfrac{1}{\sqrt{\pi }} \inf_{0<X<1} \dfrac{\displaystyle\int_{X\frac{\beta}{2\sqrt{t}}}^{\left(\frac{2}{\beta}+X\right)\frac{\beta}{2\sqrt{t}}} 
e^{-Z^2} dZ}{(1-X)^{p+1}}\\
& = & \dfrac{1}{2}\inf_{0< X <1} \dfrac{\mbox{erf} \left( (\frac{2}{\beta}+X)\frac{\beta}{2\sqrt{t}}\right) 
- \mbox{erf} \left(X\frac{\beta}{2\sqrt{t}}\right)}{(1-X)^{p+1}} =\frac12 H(t,\beta).
\end{eqnarray*}

Let us finally verify (\ref{Gmonot}). For any $t>0$ and $x\in [-1,1]$, using the change of variables $x-y=z\sqrt{4 t}$, we obtain
$$e^{t\Delta_\mathbb{R}}\chi_{(-1,1)}(x) 
= \dfrac{1}{\sqrt{4\pi t}} \displaystyle\int_{-1}^1 e^{-\frac{(x-y)^2}{4t}}dy
= \dfrac{1}{\sqrt{\pi}}\displaystyle\int_{\frac{x-1}{\sqrt{4t}}}^{\frac{x+1}{\sqrt{4t}}} e^{-z^2}dz.$$
Since $x-1\le 0\le x+1$, this quantity is nonincreasing as $t$ increases and 
property~(\ref{Gmonot}) then immediately follows from (\ref{quotient inside B q=0}).
\end{proof}

\begin{proof}[Proof of Proposition~\ref{type 1 quantitative}]
Let  $d\in (0,d_0)$, $\tau\in (0,1)$ and let $(\beta,K)\in\mathcal{A}_1$, $\eta=1$.
 Estimate~(\ref{def-t0Tbar}) follows from Lemma~\ref{estimates for T}(i).  
We may apply Lemma~\ref{control Theta}  for $q=0, \eta=1$ with the function $a(r)$ given 
by Lemma~\ref{Explicita0}; see Figure~\ref{fig3}.


\begin{figure}[h]
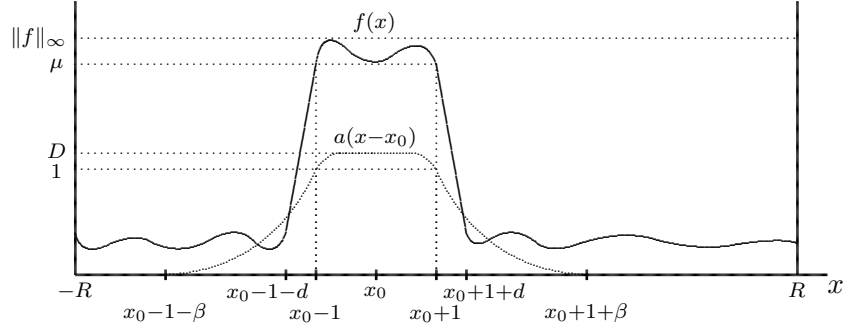

\[
\beginpicture
\setcoordinatesystem units <0.8cm,0.7cm>
\setplotarea x from -6 to 8, y from -1 to 5.5

\setdots <0pt>
\linethickness=1pt
\putrule from -5 0 to 7.5 0
\putrule from -5 0 to -5 5.2
\putrule from 7 0 to 7 5.2

\setquadratic
\plot -5 0.8  -4.8 0.5  -4.4 0.6 /
\plot -4.4 0.6  -4 0.7  -3.7 0.6 /
\plot -3.7 0.6  -3.2 0.5  -2.7 0.7 /
\plot -2.7 0.7  -2.3 0.8  -2 0.6 /
\plot -2 0.6  -1.7 0.5  -1.5  0.8 /

\plot -1 4  -0.8 4.45  -0.4 4.2 /
\plot -0.4 4.2  -0.05 4.05  0.25 4.15 /
\plot 0.25 4.15  0.72 4.35  1 4 /

\plot 1.5 0.8  1.65 0.6  1.9 0.68 /
\plot 1.9 0.68  2.3 0.8  2.7 0.58 /  
\plot 2.7 0.58 3 0.5  3.4 0.6 /  
\plot 3.4 0.6  4.1 0.75  4.7 0.65 /
\plot 4.7 0.65  5.4 0.53  6 0.58 /
\plot 6 0.58  6.63 0.65  7 0.6 /

\setlinear
\plot -1.5 0.8  -1 4 /
\plot 1 4  1.5 0.8 /

\setdots <2.5pt>
\setlinear 
\plot -1 0  -1 4 /
\plot 1 0  1 4 /
\plot -5 4  1 4 /
\plot -5 4.5  7 4.5 /

\put {$x$} [lt] at 7.5 -.1
\put {$_{R}$} [ct] at 7 -.15
\put {$_{-R}$}  [ct] at -5 -.15
\put {$_{f(x)}$} [rc] at 0.35 4.8
\put {$_{\|f\|_\infty}$}  [ct] at -5.6 4.7
\put {$_{\mu}$} [ct] at -5.3 4.1

\setdots <0pt>

\putrule from 0 -.1 to 0 .1
\putrule from 3.5 -.1 to 3.5 .1
\putrule from 1 -.1 to 1 .1
\putrule from 1.5 -.1 to 1.5 .1
\putrule from -3.5 -.1 to -3.5 .1
\putrule from -1 -.1 to -1 .1
\putrule from -1.5 -.1 to -1.5 .1

\put {$_{x_0}$} [ct] at 0 -.22
\put {$_{x_0-1}$} [cb] at -1 -.9
\put {$_{x_0+1}$} [cb] at 1 -.9
\put {$_{x_0-1-d}$} [ct] at -1.8 -.17
\put {$_{x_0+1+d}$} [ct] at 1.8 -.17
\put {$_{x_0-1-\beta}$} [cb] at -3.5 -.9
\put {$_{x_0+1+\beta}$} [cb] at 3.5 -.9
\put {$_{a(x-x_0)}$} [rc] at 0.7 2.6
\put {$_{1}$} [ct] at -5.3 2.1
\put {$_{D}$} [ct] at -5.3 2.5

\setdots <1.2pt>
\setquadratic
\plot 1 2  2 0.55  3.5 0 /
\plot 0.6 2.32  0.8 2.24  1 2 /

\plot -1 2  -2 0.55  -3.5 0 /
\plot -0.6 2.32  -0.8 2.24  -1 2 /

\setlinear
\plot -0.6 2.32  0.6 2.32 / 

\setdots <2.5pt>
\setlinear 
\plot -1 0  -1 2 /
\plot 1 0  1 2 /
\plot -5 2  1 2 /
\plot -5 2.32  0 2.32 /

\endpicture
\] 
\caption{ Location of the support of the cut-off function $a$ in the proof of Proposition~\ref{type 1 quantitative}
($d$~is fixed and $\beta$ is one of the optimization parameters).}
\label{fig3}
\end{figure}

It follows from (\ref{estTypeILemma}) and (\ref{function a q=0}) for $\eta=1$ that,
for all $t\in [t_0,T)$,
\begin{equation}\label{typeIprelim} 
\big[1 - u\big(t,x_0\pm(1+d)\big)\big]^{p+1} \ge (p+1)\left(\dfrac{\beta-d}{\beta}\right)^{p+1} \eps\, \mu (T-t),
\end{equation}
where $\eps=\eps_1$, given in (\ref{epsilon 1b}). 

On the other hand, assuming $x_0\ge 0$ without loss of generality
and recalling the definition of $S$ in (\ref{H and G}),
it follows from the comparison properties for the Dirichlet heat semigroup in 
Proposition~\ref{prop semigroup comparison} that, for all $ \lambda\in [1,R-x_0]$,
$$
\begin{array}{ll}
\left(e^{t\Delta_\Omega}\chi_{I_0}\right) (x) 
&\geq e^{-\frac{\pi^2}{ 4(R-x_0)^2}}\Bigl[1-e^{-(R-x_0-1)(R-x_0-\lambda)/t}\Bigr]\left(e^{t\Delta_\mathbb{R}}\chi_{I_0}\right) (x) \\
\noalign{\vskip 1mm}
&= S(t,\lambda-1)\left(e^{t\Delta_\mathbb{R}}\chi_{I_0}\right) (x)
\quad\hbox{ for all $x\in [x_0-\lambda,x_0+\lambda]$.}
\end{array}
$$
 This along with estimate \eqref{estimate h t0} guarantees that
\begin{equation}\label{compsemigroup}
\begin{array}{lll}
\displaystyle \inf_{|x-x_0|\le \lambda} \dfrac{[e^{t_0\Delta_{\Omega}}\chi_{I_0}](x)}{h(u(t_0,x))a( |x-x_0|)}
& \ge \dfrac{S(t_0,\lambda-1)}{W(0,\tau,K)} 
\ \ \displaystyle \inf_{|x-x_0|\le \lambda} \dfrac{[e^{t_0\Delta_\R}\chi_{I_0}](x)}{a( |x-x_0|)} \\
& = \dfrac{S(t_0,\lambda-1)}{W(0,\tau,K)}
  \ \ \displaystyle \inf_{|y|\le \lambda}\dfrac{[e^{t_0\Delta_\R}\chi_{(-1,1)}](y)}{a( |y|)}.
\end{array}
\end{equation}
Applying (\ref{compsemigroup}) with $\lambda=1+\beta$,  \eqref{Wcalcul}, (\ref{quotient inside B q=0})
and (\ref{quotient outside B}), it follows that
\begin{equation}\label{ineqeps1} 
\eps_1 \ge \frac12 \dfrac{S(t_0,\beta)}{ K+\tau^{-p}} \min \lbrace  
G(t_0,\beta,K), H(t_0,\beta)\rbrace.
\end{equation}
Combining (\ref{typeIprelim}) and (\ref{ineqeps1}),
the conclusion then follows by taking the supremum over 
$(\beta,K)\in\mathcal{A}_1$.
\end{proof}

\begin{remark}
(a) The choice $\eta=1$ allows one to get rid of condition $\eps\leq \eps_2$ in~(\ref{epsilon 1}), leading to 
an important 
simplification of the expression for $\hat{\eps}$ in Proposition~\ref{type 1 quantitative}.

(b) Concerning the definition of the subregions in \eqref{DefSubsetsSigma}, we observe that it would not be of any use 
to separate the cases $1-u\ge \eta$ and $1-u<\eta$ outside the interval $I_0=(x_0-1,x_0+1)$. 
Indeed, this would not lead to a better function $a(x)$ outside $I_0$,
since the supremum in \eqref{condaexterior} is achieved for $u\sim 1$.
\end{remark}

\section{Proof of Theorems~\ref{global result dim 1}, \ref{global result dim 1a}  and \ref{global result dim 1b}}

The proofs  will use the following no-touchdown criterion, which enables one to exclude touchdown 
on a given subinterval $D$ of $\Omega$, under a type I estimate on $\partial D$ and a suitable smallness assumption on $f$ in $D$.

\begin{lemma}\label{basic supersol}
Let $\Omega=(-R,R)$, $\tau\in (0,1)$ and $u$ be the solution of problem~(\ref{quenching problem}). 
Let either 
$$D= (x_0-b,x_0+b)\subset\subset\Omega \quad\hbox{ and }\quad \Gamma=\partial D,
\leqno(i)$$
or
$$D=(a,R)\ \hbox{ for some $a\in(-R,R)$} \quad\hbox{ and }\quad \Gamma=\{a\}.
\leqno(ii)$$
 Let $\tau\in (0,1)$, $t_0 = t_0(\tau) = \frac{1-\tau^{p+1}}{(p+1)\|f\|_\infty}$  and assume that
\begin{equation}\label{hyp basic supersol}
(1-u)^{p+1} \geq k(T-t) \quad\hbox{ on $[t_0,T)\times \Gamma$}
\end{equation}
for some $k>0$. If
\begin{equation}\label{hypf basic supersol}
\|f\|_{L^\infty(D)}<\frac{1}{p+1}\min \left\{ k, \frac{\tau^{p+1}}{T-t_0}\right\},
\end{equation}
then $ \mathcal{T} \cap  D=\emptyset$. 
 In addition, $R\notin \mathcal{T}$ in case (ii).
\end{lemma}

The proof of this lemma is given in~\cite{ES2} 
for the case when the type I estimate \eqref{hyp basic supersol} is satisfied in the whole time interval $[0,T)$. 
Here, we make a slight modification of the proof in order to use Propositions~\ref{type 1 quantitative}, \ref{type 1 quantitative1}, \ref{type 1 quantitative1b}, where the quantitative type I estimate only holds in the interval $[t_0,T)$.

\begin{proof}
We define the comparison function
$$w(t,x) := y(t) \psi(x)  \quad \text{ for } (t,x)\in 
[t_0,T)\times \overline{D},$$
where $y(t)$ is defined by
$$
y(t) =1-A(T-t)^{\frac{1}{p+1}},
\qquad\hbox{ with }
A = \min \left\{ k^{\frac{1}{p+1}} , \frac{\tau-\sigma}{1-\sigma}(T-t_0)^{-\frac{1}{p+1}} \right\}
$$
for some $\sigma\in (0,\tau)$ to be chosen later, and $\psi (x)$ is given by 
$\psi(x) := 1- \sigma \left( 1-\dfrac{(x-x_0)^2}{b^2}\right)$ in case (i), and by  $\psi(x) := 1-\frac{\sigma}{R-a} (x-a)$ in case (ii).

Considering $\sigma$ small enough, we obtain, in $(t_0,T)\times D$
\begin{equation}\label{eqnwcomp}
w_t - w_{xx} -f(x)(1-w)^{-p} \ge \left( \frac{A}{p+1}(1-\sigma) - f(x)A^{-p}\right) T^{-\frac{p}{p+1}} -  \psi''(x) \ge 0,
\end{equation}
 noting that $\|f\|_{L^\infty(D)}< \frac{A^{p+1}}{p+1}$ by \eqref{hypf basic supersol} for $\sigma$ small enough.

We next look at the parabolic boundary of $[t_0,T)\times D$. On the one hand,  by Lemma~\ref{estimates for T}(i), we have
\begin{equation}\label{boundcondw1}
w(t_0,x) \geq (1-A(T-t_0)^{\frac{1}{p+1}}) (1-\sigma) \geq 1-\tau \geq u(t_0,x) \quad \text{in} \ \overline{D}.
\end{equation}
On the other hand, using $\psi =1$ on $\Gamma$ and $A\leq k^{\frac{1}{p+1}}$, we apply \eqref{hyp basic supersol} to obtain
\begin{equation}\label{boundcondw2}
\begin{array}{lll}
w(t,x) 
= 1 - A (T-t)^{\frac{1}{p+1}}\geq 1 - k^{\frac{1}{p+1}} (T-t)^{\frac{1}{p+1}} \geq u(t,x)
\quad\hbox{ in $[t_0,T)\times \Gamma$.}
\end{array}
\end{equation}
In case (ii), we also note that due to the boundary conditions on $u$, we have
\begin{equation}\label{boundcondw3}
w(t,x) \ge 0 = u(t,x) \quad\hbox{in $[t_0,T)\times \{R\}$.}
\end{equation}

By (\ref{eqnwcomp}), (\ref{boundcondw1}), (\ref{boundcondw2}) and (\ref{boundcondw3}) (in case (ii)),
along with the comparison principle and $y(t)\leq 1$ for all $t\in[0,T)$, we conclude that
\begin{equation}\label{usmallD}
u(t,x) \leq w(t,x)\leq \psi(x) \quad\hbox{ in $(0,T)\times D$.}
\end{equation}
In both cases, since $\psi$ is uniformly smaller than $1$ in compact subsets of $D$,
it follows from \eqref{usmallD} that $\mathcal{T} \cap D=\emptyset$. We also see that in case (ii),  
$\psi$ is uniformly smaller than $1$ in a neighborhood of $\{R\}$, so we can rule out touchdown at this point.
\end{proof}

\begin{proof}[Proof of Theorems~\ref{global result dim 1} and \ref{global result dim 1a}]
Let $\rho, \mathcal{A}$ be given by (\ref{simplified problem}). 
We first claim that $\mathcal{A}$ is nonempty, so that $\rho$ is well defined and positive.
We notice that, due to the assumption $\mu>\mu_1(p)$, there exists $\tau$ such that
\begin{equation}\label{condmutau}
\dfrac{\mu}{2\mu-\mu_1(p)} \leq \tau <1.
\end{equation}
Next we see that the condition $\delta(\beta,K)\leq 1$ is equivalent to
\begin{equation}\label{condphi12}
\phi_1(K):= \arctan\biggl[\sqrt{\dfrac{p+1}{p\mu\beta^2}(K^{-1}+p+1)}\biggr] 
\le \sqrt{\frac{p\mu K(K+\theta)}{(1+K)^2}}=:\phi_2(K),
\end{equation}
where $\theta=(p+1)^{-1}\in (0,1)$. We have
$$\frac{\partial}{\partial K}\biggl[\frac{K(K+\theta)}{(1+K)^2}\biggr]
=\frac{(2K+\theta)(1+K)-2K(K+\theta)}{(1+K)^3}>\frac{2(1-\theta)K}{(1+K)^3}>0,$$
so that $\phi_2(K)$ is monotonically increasing on $(0,\infty)$.
Moreover, $\phi_1(K)$ is monotonically decreasing on $(0,\infty)$.
Therefore, there exists $K>0$ satisfying \eqref{condphi12} if and only if
$$\lim_{K\to\infty} \phi_2(K)>\lim_{K\to\infty} \phi_1(K) \Longleftrightarrow
\sqrt{p\mu}> \arctan\Bigl[\dfrac{p+1}{\beta\sqrt{p\mu}}\Bigr]
 \Longleftrightarrow \beta > \frac{p+1}{\sqrt{p\mu}}\,\overline\cot\bigl[\sqrt{p\mu}\bigr].
$$
By assumption \eqref{hypmu1d0}, we may thus
find $\beta\in (d,d_0]$ satisfying  this condition. 
Moreover, \eqref{condphi12} is  then true for any sufficiently large $K>0$.
It follows that $\mathcal{A}$ is nonempty.
 
Next set
$$
\left\lbrace \begin{array}{ll}
D=(x_0+1+d,R), \quad \Gamma=\{x_0+1+d\} &\quad\hbox{ if $f$ satisfies (\ref{condition 1 on f})}, \\
\noalign{\vskip 2mm}
D=(x_1+1+d,x_0-1-d), \quad \Gamma=\partial D &\quad\hbox{ if $f$ satisfies (\ref{condition 2 on f}).} \\
\end{array}\right.
$$
By our assumption on $f$, we may select $(\tau,\beta,K)\in \mathcal{A}$ such that
\begin{equation}\label{finftyD}
\|f\|_{L^\infty(D)}\le \mu{\hskip 1pt}\hat\eps(\beta,K),
\end{equation}
with
$$\hat\eps(\beta,K) = \dfrac{1}{2} \left(\dfrac{\beta-d}{\beta}\right)^{p+1}
 \dfrac{S(t_0,\beta)}{K+\tau^{-p}} \min \bigl\{ H(t_0,\beta), G(t_0,\beta,K) \bigr\}
 \quad\hbox{ and }\quad t_0=t_0(\tau)=\dfrac{1-\tau^{p+1}}{(p+1)\| f\|_\infty}.$$

Under assumption~(\ref{condition 1 on f}), it follows from Proposition~\ref{type 1 quantitative} that
\begin{equation}\label{boundaryGamma}
(1 - u)^{p+1} \geq  (p+1)\mu{\hskip 0.5pt}\hat\eps(\beta,K)  (T-t)
\quad\hbox{ on $[t_0,T)\times\Gamma$.}
\end{equation}
Moreover, \eqref{boundaryGamma} remains true under assumption~(\ref{condition 2 on f}). 
Indeed, we may apply Proposition~\ref{type 1 quantitative}
with $x_0$ replaced by $x_1$, recalling $|x_1|\le |x_0|$ and noticing that $S(t,\beta)$ in (\ref{H and G}), 
hence $\hat\eps$ depends on $d_0=R-1-|x_0|$ in a monotonically increasing way.

 \smallskip

Now, in view of applying Lemma~\ref{basic supersol}, we claim that
\begin{equation}\label{simplification claim}
 (p+1) \mu{\hskip 0.5pt}\hat\eps(\beta,K) \le\dfrac{\tau^{p+1}}{T-t_0}.
\end{equation}
Indeed, since $H(t_0,\beta)\le  \text{\rm erf} \bigl( \frac{1}{\sqrt{t_0}}\bigr) \le 1$ by (\ref{H and G}), $K>0$ 
and $S(d_0+1,t_0,d_0-\beta)\le 1$, we have $\hat\eps\le \frac{\tau^{p}}{ 2}$. 
Therefore, using \eqref{condmutau},  $\mu_1(p)=2\mu_0(p)$ and $(p+1)(\mu - \mu_0(p))T\le 1$ by  Lemma~\ref{estimates for T}(ii), we obtain
$$ (p+1) \mu{\hskip 0.5pt}\hat\eps(\beta,K) \le  2(p+1)(\mu-\mu_0(p))\tau\hat\eps(\beta,K) \le  \dfrac{ 2\hat\eps(\beta,K)\tau}{T}
\le \dfrac{\tau^{p+1}}{T-t_0}.$$

By Lemma~\ref{basic supersol}, it follows that that $\mathcal{T}\cap D=\emptyset$,
 and that $R\not\in\mathcal{T}$ in case of (\ref{condition 1 on f}).
 Finally, let us show that $\mathcal{T}\cap\Gamma=\emptyset$.
By the continuity of $f$, assumptions~(\ref{condition 1 on f}) and (\ref{condition 2 on f})
 remain true for some $\tilde d<d$ close to $d$.
 Also, since the set $\mathcal{A}$ and the quantity $\bigl(\frac{\beta-d}{\beta}\bigr)^{p+1}$ increase as $d$ decreases,
it follows that the supremum in (\ref{simplified problem}) is a nonincreasing function of $d\in (0,d_0)$, 
so we deduce that $\mathcal{T}\cap\Gamma=\emptyset$. This concludes the proof.
\end{proof}


In view of the proof of Theorem~\ref{global result dim 1b},
we first establish the following quantitative type~I estimate,
which is a version of Proposition~\ref{type 1 quantitative},
 with the three parameters $(\beta, K, \eta)$ instead of $(\beta, K)$.

\begin{proposition}\label{type 1 quantitative1}
 Under the assumptions of Proposition~\ref{type 1 quantitative}, $u$ satisfies the type~I estimate \eqref{typeIest}
where now
\begin{equation}\label{def-epsilonbar1}
\overline\eps= \sup_{(\beta,K,\eta)\in\widehat{\mathcal{A}}}\  \hat\eps(\beta,K,\eta),
\end{equation} 
with
\begin{equation}\label{defepshat1}
\hat\eps(\beta,K,\eta) = \textstyle\frac12 \bigl(\textstyle\frac{\beta-d}{\beta}\bigr)^{p+1} 
\min \Bigl\{ \textstyle\frac{S(\overline T,0)}{K+\eta^{-p}}G( \overline T,\beta,K,\eta), \ 
\textstyle\frac{S(t_0,\beta)}{K+\tau^{-p}} \min \bigl[ H(t_0,\beta), G(t_0,\beta,K,\eta) \bigr]\Bigr\},
\end{equation}
$$\widehat{\mathcal{A}}=\mathcal{A}_2:=
\Bigl\{ (\beta,K,\eta)\in (d,d_0)\times(0,\infty)\times(0,1],\ 
K\ge \textstyle\frac{p\eta}{\mu\beta^2}-\textstyle\frac{1}{(p+1)\eta^p},\ \delta(\beta,K,\eta) \leq 1\Bigr\},$$
 where $\overline T= \frac{1}{(p+1)(\mu - \mu_0(p))}$,
the functions $S, H$ are defined in (\ref{H and G}) and $G, \delta$ are defined in \eqref{H and G2},
\end{proposition}

\begin{proof}
Let $d\in (0,d_0)$ and $\tau\in (0,1)$.
Let $(\beta,K,\eta)\in\mathcal{A}_2$ and let $a(r)=a_{\beta,K,\eta}(r)$ be given by  (\ref{function a q=0}).
As in the proof of Proposition~\ref{type 1 quantitative}, 
 we shall rely on Lemma~\ref{control Theta} with $q=0$, and we first express the infima in 
(\ref{epsilon 1b}) and (\ref{epsilon 1}) in terms of the error function.
Denote by $e^{t\Delta_\mathbb{R}}$ the heat semigroup on $\mathbb{R}$.
By the proof of Lemma~\ref{lemTransfErr} we have
\begin{equation}\label{quotient inside B q=0-1}
\displaystyle\inf_{0<r<1} \dfrac{e^{t\Delta_\mathbb{R}}\chi_{(-1,1)}(r)}{a(r)}
=\dfrac{1}{2} G(t,\beta,K,\eta),\quad t>0,
\end{equation}
and
\begin{equation}\label{quotient outside B2} 
\displaystyle\inf_{1<r<1+\beta} \dfrac{e^{t\Delta_\mathbb{R}}\chi_{(-1,1)}(r)}{a(r)} 
=\dfrac{1}{2} H(t,\beta),\quad t>0.
\end{equation}
Moreover,
\begin{equation}\label{Gmonot1}
t\mapsto S(t,\beta),\ G(t,\beta,K,\eta) \ \hbox{ are nonincreasing for $t>0$.}
\end{equation}
The proof of Proposition~\ref{type 1 quantitative1}
is then completely similar to that of Proposition~\ref{type 1 quantitative},
applying Lemma~\ref{control Theta} with the function $a(r)$ given by Lemma~\ref{Explicita0}.
The only difference is that, since now $\eta$ may be less than $1$, 
we need to choose $\eps=\min(\eps_1,\eps_2)$ in Lemma~\ref{control Theta}.
To estimate $\eps_2$, we use \eqref{epsilon 1},  \eqref{Wcalcul}, (\ref{quotient inside B q=0}),
 \eqref{compsemigroup} with $\lambda=1$, and \eqref{quotient inside B q=0-1} to deduce
$$
\begin{array}{lll}
\eps_2
&=\dfrac{1}{W(0,\eta,K)} \displaystyle\inf_{(t,x)\in [t_0,\overline T]\times I_0} \dfrac{e^{t\Delta_\Omega}\chi_{I_0}(x)}{a(|x-x_0|)} \\
&\ge\dfrac{1}{2} \dfrac{1}{\eta^{-p} + K}\displaystyle\inf_{t\in [t_0,\overline T]} S(t,0) \,G(t,\beta,K,\eta) 
=\dfrac{1}{2} \dfrac{S(\overline T,0)}{\eta^{-p} + K} \, G(\overline T,\beta,K,\eta).
\end{array}
$$
The conclusion follows by taking the supremum over $(\beta,K,\eta)\in\mathcal{A}_2$.
\end{proof}

\begin{proof}[Proof of Theorem~\ref{global result dim 1b}]
Let $\mathcal{A}$ be given by \eqref{condition on parameters 2}. 
Picking any  $\tau\in (0,1)$, $\beta\in (d,d_0)$,  $K>0$ 
and then $\eta\in (0,1)$ small,
we see that $\mathcal{A}$ is nonempty, so that $\rho$ is well defined and positive.
 
Now arguing as in the second paragraph of the proof of Theorems~\ref{global result dim 1}-\ref{global result dim 1a}, 
 we may select $(\tau,\beta,K,\eta)\in \mathcal{A}$ satisfying  \eqref{finftyD} and \eqref{boundaryGamma}, 
 where  $\hat\eps=\hat\eps(\beta,K,\eta)$ is now given by (\ref{defepshat1}) and $t_0$ is given by \eqref{deft0}.
Applying Lemma~\ref{basic supersol}
 in the same way as before, we finally conclude that $\mathcal{T}\cap \overline{D}=\emptyset$.
\end{proof}

\section{Proof of Theorems~\ref{global result dim 1c}-\ref{global result dim 1c2}}

 It is based on refinements of various ingredients of the proof of Theorems~\ref{global result dim 1} and \ref{global result dim 1a}.
Namely, we introduce a more  precise cut-off 
function $a(x)$, corresponding to the choice $q=1$ in 
the function $h$ (cf.~\eqref{defhKq} and Lemma~\ref{control Theta}).
We also use an improved lower estimate on $u_t$ (see Lemma~\ref{low bound u_t2}), 
and an upper estimate of $u$ for $t$ small at points of small permittivity (see Lemma~\ref{u upper estimate}).

\subsection{Determination of the family of functions $a( r)$ in the case $q=1$}
\label{construction a q=1}

In this subsection, we compute the solution of the auxiliary problem \eqref{diffeqa4}-\eqref{diffeqa3} 
in the case $q=1$ and $\eta =1$, which will be used to prove 
  Theorems~\ref{global result dim 1c}-\ref{global result dim 1c2}.
 The following Lemma~\ref{Explicita1} is the analogue of Lemma~\ref{Explicita0} for $q=0$.
We see that the expression of $a$ in (\ref{function a q=1B}) is significantly more complicated than for $q=0$.
As for the choices of $q\in [0,1]$ other than $q=0$ or $q=1$, they seem quite difficult to investigate and
have been left out of this study.

\begin{lemma}\label{Explicita1}
Let $\mu,\beta,K>0$, $q=1$, $\eta=1$ 
and let $h, F$ be given by \eqref{defhKq}, \eqref{diffineqa2}. 
 Assume 
\begin{equation}\label{hypKmubeta}
 K\mu\beta^2\le \frac{p(p+2)-K}{p} 
\end{equation}
and let
$$A_0 =\sqrt{\frac{p(1+K)+(p(p+2)-K)K}{p(1+K)^2}},\ \
A_1 = \arctan\left( \sqrt{\frac{p(1+K)}{p(p+2)-K}+K} \right),$$
$$A_2 = \arctan\left( \sqrt{\frac{p}{p(p+2)-K}} \right),\quad A_3 = \arctan\left(\frac{1}{\sqrt{K\mu\beta^2}} \right),$$
$$\delta_1(K)=\dfrac{A_1}{A_0\sqrt{K\mu}},\quad 
\delta_2 (\beta,K) = \dfrac{A_3-A_2}{\sqrt{K\mu}}\ge 0.$$
Assume in addition that $\delta_1+\delta_2\le 1$.
Then there exist $r_0\in [0,1)$ and a solution $a\in  W^{2,2}([0,1+\beta])$ 
of \eqref{diffeqa4}-\eqref{diffeqa3},~\eqref{amonotone}.
The pair 
$(r_0,a)$ is then unique and it is given by
\begin{equation}\label{function a q=1B}  
a_{\beta,K}(r)= \left\lbrace \begin{array}{ll}
D_1 , &  r \in [0, r_0), \\
\noalign{\vskip 1mm}
D_1 \cos^\alpha\bigl(A_0\sqrt{K\mu}(r-r_0)\bigr), &  r\in [r_0,r_1), \\
\noalign{\vskip 2mm}
D_2 \cos^{p+1} \left(\sqrt{K\mu}(r-1)+A_3\right), &  r\in [r_1,1], \\
\noalign{\vskip 1mm}
\left(\dfrac{1+\beta -r}{\beta}\right)^{p+1}, & r\in (1,1+\beta],
\end{array}\right.
\end{equation}
where $r_0 = 1-\delta_1-\delta_2$, $r_1=1-\delta_2$,
$$\alpha = \frac{p+1}{(1+K)A_0^2}, \quad D_2= \left(1+\frac{1}{K\mu\beta^2}\right)^{\frac{p+1}{2}}, \quad D_1 =  D_2 \frac{D_{11}^{\alpha}}{D_{12}^{p+1}},$$
$$D_{11} = \sqrt{1+K + \frac{p(1+K)}{p(p+2)-K}}, \quad
D_{12} = \sqrt{1 + \frac{p}{p(p+2)-K}}.$$
\end{lemma}

 In view of the proof of Lemma~\ref{Explicita1}, we first compute the function $F$ in \eqref{diffineqa2}.

\begin{lemma}\label{Explicita1F}
 Let $\mu,\beta,K>0$, $q=1$, $\eta=1$ 
and let $h, F$ be given by \eqref{defhKq}, \eqref{diffineqa2}. 
Assume \eqref{hypKmubeta} and set
$$m_1=\frac{(p-K)^2}{p(p+1)(1+K)},\ \ M_1=\frac{(p+1)K\mu}{1+K},\ \ m_2=\frac{p}{p+1},\ \ M_2=(p+1)K\mu.$$
Then we have
\begin{equation}\label{orderMm12}  
0<M_1<M_2,\quad 0\le m_1<m_2<1,
\end{equation}
\begin{equation}\label{Frgeq1}
F(r,\xi) = \frac{p}{p+1},\quad r>1,
\end{equation}
and
\begin{equation}\label{Fsplit}
F(r,\xi) = \left\{ \begin{array}{ll}
m_1 \xi^2 - M_1  &  \text{if} \  |\xi| <  \xi_0 \\
\noalign{\vskip 1.5mm}
m_2 \xi^2 - M_2  &  \text{if} \  |\xi| \geq \xi_0,
\end{array}\right.
\qquad 0\le r<1,
\end{equation}
where 
\begin{equation}\label{DefXi0}
\xi_0:=\sqrt{\frac{M_2-M_1}{m_2-m_1}}=(p+1)\sqrt{\frac{pK\mu}{p(p+2)-K}}.
\end{equation}
\end{lemma}

\begin{proof}
\textit{ Step 1: Computation of $F(r,\xi)$ for $r>1$.} 
In this step, we keep $q\in (0,1]$, since the computation requires the same amount of effort.
 Since $K<p(p+2)$ by \eqref{hypKmubeta}, we have in particular
\begin{equation}\label{construction of a 1b}
K\leq \dfrac{p(p+q+1)}{q}.
\end{equation}
We claim that,  under condition \eqref{construction of a 1b}, we have
\begin{equation}\label{condaexterior2}
\displaystyle\sup_{u\in (0,1)}\left[\dfrac{h'^2(u)}{hh''(u)}\right]=\dfrac{h'^2(1)}{hh''(1)}=\dfrac{p}{p+1},
\end{equation}
which immediately yields \eqref{Frgeq1}.

To show \eqref{condaexterior2}, we compute
$$hh'' =(1-u)^{-2p-2} \bigl[1 + K(1-u)^{q+p}\bigr] \bigl[p(p+1) + Kq(q-1)(1-u)^{q+p}\bigr]$$
and
$$h'^2=(1-u)^{-2p-2}\bigl[p-Kq(1-u)^{q+p}\bigr]^2,$$
hence
\begin{equation}\label{calculQX}
\dfrac{h'^2}{hh''}=
\dfrac{\bigl[p-Kq(1-u)^{q+p}\bigr]^2}{\bigl[1 + K(1-u)^{q+p}\bigr] \bigl[p(p+1) + Kq(q-1)(1-u)^{q+p}\bigr]}.
\end{equation}
Now setting $X=K(1-u)^{q+p}$, we see that \eqref{condaexterior2} is equivalent to
$$Q(X):=p(1 + X)[p(p+1) + q(q-1)X]-(p+1)(p-qX)^2\ge 0\quad\hbox{ for all $X\in [0,K]$.}$$
Since 
\begin{eqnarray*}
Q(X)&=&\bigl[pq(q-1)+p^2(p+1)+2(p+1)pq\bigr]X -\bigl[pq(1-q)+(p+1)q^2\bigr]X^2 \\
&=& p\bigl[q^2+p^2+2pq+p+q\bigr]X -q(p+q)X^2 
=(p+q)X\bigl[p(p+q+1) -qX\bigr]
\end{eqnarray*}
and $0\le X\le K$, claim \eqref{condaexterior2} follows from (\ref{construction of a 1b}).

\textit{ Step 2: Computation of $F(r,\xi)$ for $r<1$.} 
In view of \eqref{hypKmubeta}, properties \eqref{orderMm12} and \eqref{DefXi0} follow from
\begin{equation}\label{calculM1M2}
M_2-M_1=\frac{(p+1)K^2\mu}{1+K}>0,
\qquad m_2-m_1 
=\frac{(p(p+2)-K)K}{p(p+1)(1+K)}.
\end{equation}
 Next, taking $q=1$ and setting $y=(1-u)^{p+1}$, we have
$$F(r,\xi)= \sup_{y\in (0,1)}\hat F(\xi,y),\qquad
\hat F(\xi,y):= \dfrac{(p-Ky)^2}{p(p+1)(1+Ky)} \xi^2
- \dfrac{(p+1)K\mu}{1+Ky}. 
$$
Setting $s=\frac{1}{1+Ky}$, i.e., $Ky=s^{-1}-1$, we may rewrite
$$\hat F(\xi,y)=\dfrac{s(s^{-1}-p-1)^2}{p(p+1)} \xi^2-(p+1)K\mu s 
=C_1(\xi)s+C_2(\xi)\, +\,\xi^2\dfrac{s^{-1}}{p(p+1)}.$$
Since, for each fixed $\xi$, $\hat F$ is a convex function of $s$, it follows that
$\sup_{y\in [0,1]}\hat F(\xi,y)$ is achieved for $y=0$ or $1$.
Consequently, 
\begin{equation}\label{a(x) diff ineq q=1}
F(r,\xi)= \max \left\lbrace m_1\xi^2-M_1, \ m_2\xi^2 - M_2 \right\rbrace,
\end{equation}
which immediately yields \eqref{Fsplit}.
\end{proof}

\begin{proof}[Proof of Lemma~\ref{Explicita1}]
We split the proof in three steps.

\textit{Step 1: Preliminaries and monotonicity of $\xi(r)$.} 
Assume that there exist $r_0\in [0,1)$ and a solution $a\in  W^{2,2}([0,1+\beta])$ of \eqref{diffeqa4}-\eqref{diffeqa3}, \eqref{amonotone}.
We note in particular that $a\in C^1([r_0,1+\beta])$.
We shall show that $a$ is necessarily given by \eqref{function a q=1B}. 
It will then be a simple matter to show that this is indeed a solution.

First, in view of \eqref{Frgeq1}, it follows from the proof of Lemma~\ref{Explicita0} 
that $a$ is necessarily given by \eqref{function a q=1B} on $[1,1+\beta]$.

 We next claim that
\begin{equation}\label{ximonot}
\hbox{ $\xi(r) := \dfrac{a'(r)}{a(r)}\le 0$ is a monotone decreasing function of $r\in [r_0,1]$.}
\end{equation}
Indeed we can compute~(a.e.)
\begin{equation*}
 \xi '(r)  =  \dfrac{a''(r)}{a(r)} -\left(\dfrac{a'(r)}{a(r)}\right)^{2} = F(r,\xi(r)) - \xi^2(r) 
         =  \max \left\lbrace (m_1-1) \xi^2(r)-M_1, \ (m_2-1) \xi^2(r)- M_2 \right\rbrace <0,
\end{equation*}
 by \eqref{a(x) diff ineq q=1} and \eqref{orderMm12}.
On the other hand, in view of \eqref{hypKmubeta} and \eqref{DefXi0}, we have
$\xi (1)= -\frac{p+1}{\beta} \le  -\xi_0$.
By \eqref{ximonot}, it follows that there exists a unique $r_1\in (r_0,1]$ such that 
 $\xi(r_1)= -\xi_0$, with
\begin{equation}\label{Xisplit}
 \left\{ \begin{array}{ll}
\xi(r) > -\xi_0  &  \text{for all} \ r\in[r_0,r_1),\\
\noalign{\vskip 1mm}
\xi(r) < -\xi_0 &  \text{for all} \  r\in(r_1,1],
\end{array}\right.
\end{equation}
and that $r_1<1$ unless we have equality in \eqref{hypKmubeta}.

\textit{Step 2: Determination of $r_1$ and of $a$ in $[r_1,1]$.} 
If we have equality in \eqref{hypKmubeta}, then $\xi (1)= -\xi_0$ and we go directly to Step 3.
In the rest of this step, we thus assume that the inequality in \eqref{hypKmubeta} is strict, hence $r_1<1$.
By \eqref{Fsplit}, \eqref{Xisplit}, Step 1 and the fact that $a\in C^1([r_0,1+\beta])$, 
it follows that $a$ solves the problem
\begin{equation}\label{2eme morceau}
\dfrac{a''(r)}{a(r)} = m_2 \left(\dfrac{a'(r)}{a(r)}\right)^{2} - M_2, \quad \text{for } r_1<r<1, \qquad
a(1-) = 1, \qquad a'(1-) = -\frac{p+1}{\beta}.
\end{equation}
 Similarly as for (\ref{resolphi00}),
setting $\phi_2(r) = a^{1-m_2}(r)$, the ODE in \eqref{2eme morceau} is reduced to the equation
\begin{equation}\label{eqnPhi2}
\phi_2''=(1-m_2)a^{-m_2-1}[a a''-m_2(a')^2]=-(1-m_2)M_2\phi_2.
\end{equation}
 The solution of \eqref{eqnPhi2} is of the form
\begin{equation}\label{formPhi2}
\phi_2(r) = D_2^{1-m_2} \cos \left(\sqrt{M_2(1-m_2)}(r-1)+\theta\right)
=D_2^{1-m_2} \cos \left(\sqrt{K\mu}(r-1)+\theta\right),
\end{equation}
 for some constant $\theta\in[0,2\pi)$.
Since $\phi_2>0$ and $\phi_2'<0$ on $[r_1,1]$, this imposes
$$\Bigl[\sqrt{K\mu}(r_1-1)+\theta,\theta\Bigr]\subset \Bigl(0,\frac{\pi}{2}\Bigr),$$
 that is, 
$$ r_1>1-\dfrac{\theta}{\sqrt{K\mu}}, \quad \theta<\frac{\pi}{2}.$$
In particular, we have
\begin{equation}\label{calculxi2}
\xi(r) = \dfrac{1}{1-m_2}\dfrac{\phi_2'(r)}{\phi_2(r)} = -\sqrt{\dfrac{M_2}{1-m_2}} \tan \big(\sqrt{K\mu}(r-1)+\theta\big),
\quad r_1<r<1.
\end{equation}

By a simple computation, the boundary conditions in~\eqref{2eme morceau} give
\begin{equation*}
\theta = \arctan \left(\frac{p+1}{\beta} \sqrt{\frac{1-m_2}{M_2}}\right) 
= A_3 \in \left(0,\frac{\pi}{2}\right), \qquad
D_2^{1-m_2}    = \sqrt{1+ \frac{1-m_2}{M_2}\left(\frac{p+1}{\beta}\right)^2} = \sqrt{1+\frac{1}{K\mu\beta^2}}.
\end{equation*}
 On the other hand, by \eqref{Xisplit}, we have $\xi(r_1)=-\xi_0$. By \eqref{calculxi2},  $\xi(r_1)= -\xi_0$ and \eqref{calculM1M2}, 
 we deduce
$$\tan \big(\sqrt{K\mu}(r_1-1)+\theta \big) = \sqrt{\frac{(M_2-M_1)(1-m_2)}{(m_2-m_1)M_2}}
=\sqrt{\frac{p}{p(p+2)-K}}$$
hence,
$$r_1 = 1-\dfrac{1}{\sqrt{K\mu}} \left[ \arctan \left( \sqrt{\frac{1}{K\mu\beta^2}} \right)
- \arctan\left(\sqrt{\frac{p}{p(p+2)-K}}\right)\right]=1-\delta_2.$$
 Since we assume $\delta_1+\delta_2\le 1$, hence $0\le \delta_2<1$, 
we note that we do have $0<r_1\le 1$.

 \textit{Step 3: Determination of $r_0$ and of $a$ in $[r_0,r_1]$ and conclusion.}
By \eqref{Fsplit}, \eqref{Xisplit}, Step 2 and the fact that $a\in C^1([r_0,1+\beta])$, 
it follows that $a$ solves the problem
\begin{equation}\label{1ere morceau}
\dfrac{a''(r)}{a(r)} = m_1 \left(\dfrac{a'(r)}{a(r)}\right)^{2} - M_1, \quad \text{for }  r_0<r<r_1, \qquad
 a(r_{1-}) = a(r_{1+}),\qquad  a'(r_{1-}) =  a'(r_{1+}).
\end{equation}
 Similarly as before, setting $\phi_1(r) = a^{1-m_1}(r)$, we are left with
$$\phi_1''=(1-m_1)a^{-m_1-1}[a a''-m_1(a')^2]=-(1-m_1)M_1\phi_1,$$
whose solution is of the form
$$\phi_1(r) = D_1^{1-m_1}\cos \big(\sqrt{M_1(1-m_1)}(r-{\theta}_0)\big),$$
for some constant $\theta_0\in\R$.
 Since $\phi_1>0$ and $\phi_1'<0$ on $(r_0,r_1]$, this imposes (modulo~$2\pi$)
$$
\Bigl(\sqrt{M_1(1-m_1)}(r_0-\theta_0),\sqrt{M_1(1-m_1)}(r_1-\theta_0)\Bigr]\subset \Bigl(0,\frac{\pi}{2}\Bigr).
$$
Since $\theta_0=r_0$ owing to $\phi_1'(r_0)=0$, we thus have
$$\Bigl(0,\sqrt{M_1(1-m_1)}(r_1-r_0)\Bigr]\subset \Bigl(0,\frac{\pi}{2}\Bigr).$$

Now, we use the boundary condition $a'(r_{1-}) = a'(r_{1+})$ to obtain the value of~$r_0$.
 Since
\begin{equation*}
a'(r_{1-})  =   \frac{1}{1-m_1}\Bigl[\phi_1'\phi_1^{\frac{m_1}{1-m_1}}\Bigr](r_1)=
-D_1 \sqrt{\frac{M_1}{1-m_1}} \Bigl[ \sin\cos^{\frac{m_1}{1-m_1}}\Bigr]\big(\sqrt{M_1(1-m_1)}(r_1-r_0)\big)
\end{equation*}          
and
\begin{equation*}
a'(r_{1+})  =   -\xi_0 a(r_{1+}) =  -\sqrt{\frac{M_2-M_1}{m_2-m_1}}  a(r_1) 
            =  -\sqrt{\frac{M_2-M_1}{m_2-m_1}} D_1 \cos^{\frac{1}{1-m_1}} \big( \sqrt{M_1(1-m_1)} (r_1-r_0)\big),
\end{equation*}
 $r_0$ is determined by
\begin{equation}\label{r0}
 \tan\big(\sqrt{M_1(1-m_1)}(r_1-r_0)\big) = \sqrt{\frac{(M_2-M_1)(1-m_1)}{(m_2-m_1)M_1}} 
\end{equation}
hence, 
$$r_0 = r_1 - \frac{1}{\sqrt{M_1(1-m_1)}} 
\arctan \left(\sqrt{\frac{(M_2-M_1)(1-m_1)}{(m_2-m_1)M_1}}\right)
= r_1-\delta_1,$$
where we used \eqref{calculM1M2} in the last equality.

Finally, to obtain the value of $D_1$ we use $a(r_{1-}) = a(r_{1+})$. Using
\begin{eqnarray*}
a(r_{1-}) &=& D_1 \cos^{\frac{1}{1-m_1}} \big(\sqrt{M_1(1-m_1)}(r_1-r_0)\big) 
=  D_1 \cos^{\frac{1}{1-m_1}} \left(\arctan\sqrt{\frac{(M_2-M_1)(1-m_1)}{(m_2-m_1)M_1}}\right), \\
a(r_{1+}) &=& D_2 \cos^{p+1} \big(\sqrt{K\mu}(r_1-1)+\theta \big) 
=  D_2 \cos^{p+1} \left(\arctan\sqrt{\frac{(M_2-M_1)(1-m_2)}{(m_2-m_1)M_2}}\right),
\end{eqnarray*}
we obtain the expression for $D_1$ in the statement
 after a straightforward calculation.

 We have thus proved that $(r_0,a)$ is necessarily given by \eqref{function a q=1B}.
Conversely, an immediate inspection of the above proof shows that \eqref{function a q=1B}
does define a solution $a\in  W^{2,2}([0,1+\beta])$ 
of \eqref{diffeqa4}-\eqref{diffeqa3},~\eqref{amonotone}.
\end{proof}

\subsection{Improved lower bound for~$u_t$}

In this subsection, we improve the lower bound on~$u_t$ used in \eqref{low bound u_t} for Propositions~\ref{type 1 quantitative}, \ref{type 1 quantitative1},
by exploiting the contribution coming from the nonlinear term in \eqref{eqnut}.
This will be used in the proof of Theorems~\ref{global result dim 1c}-\ref{global result dim 1c2}.

\begin{lemma}\label{low bound u_t2}
Let $x_0\in\Omega=(-R,R)$ with $\tilde R:=R-|x_0|>1$.
Set $I_\beta=(x_0-1-\beta,x_0+1+\beta) \subset \Omega$ for $\beta\ge 0$.
Assume that $f$ satisfies 
$$f\ge \mu>0  \quad\hbox{ in $I_0$.}$$
Then the solution $u$ of problem (\ref{quenching problem}) satisfies
\begin{equation}\label{ut improved est}
u_t(t,x) \geq \big( 1+p\mu \Lambda_0 (t,x)\big) \mu e^{t\Delta_{I_\beta}} \chi_{I_0} 
\qquad \text{in } (0,T)\times I_\beta,
\end{equation}
where $\Lambda_0(t,x)$ is given by
\begin{equation}\label{Lambda ut estimate}
\Lambda_0(t,x) = \int_0^t   \bigl(1-(p+1)\mu  K_s s\bigr)^{-\frac{p}{p+1}} \, e^{(t-s)\Delta_{I_\beta}}\left[ \dfrac{\chi_{I_0}}{1-(p+1)\mu\varphi (s,\cdot)}\right] ds.
\end{equation}
Here $K_s= \displaystyle\inf_{0<\tau<s} \left(e^{\tau\Delta_{I_\beta}} \chi_{I_0}\right)( x_0+1),\
 \varphi(s,\cdot) = \displaystyle\int_0^s e^{\tau\Delta_{I_\beta}} \chi_{I_0} d\tau$,
 where $1-(p+1)\mu  K_s s>0$ for $s\in (0,T)$ and the denominator in 
\eqref{Lambda ut estimate} is positive in $(0,T)\times I_\beta$.

In particular, we can estimate
\begin{equation}\label{Lambda est1}
\Lambda_0(t,x) \geq \int_0^t   \bigl(1-(p+1)\mu  K_s s\bigr)^{-\frac{p}{p+1}-1} \, e^{(t-s)\Delta_{I_\beta}} \chi_{I_0} ds
\geq \varphi (t,x)
\qquad \text{in } (0,T)\times I_\beta.
\end{equation}
\end{lemma}

\begin{proof}
\textit{Step 1.} We first claim that
\begin{equation}\label{ut estimate step1}
(1-u)^{p+1} \leq 1- (p+1) \mu \int_0^t \big( e^{s\Delta_{I_\beta}} \chi_{I_0}\big) ds
\quad\hbox{ in $(0,T)\times I_\beta$},
\end{equation}
 which in particular guarantees that the denominator in 
\eqref{Lambda ut estimate} is positive in $(0,T)\times I_\beta$
and $1-(p+1)\mu  K_s s>0$ for $s\in (0,T)$.

Indeed, let $w = \frac{1-(1-u)^{p+1}}{p+1} \geq 0.$ We compute
$$w_t = (1-u)^p u_t, \qquad w_x = (1-u)^p u_x, \qquad w_{xx} = (1-u)^p u_{xx}-p(1-u)^{p-1} u_x^2.$$
Hence,
$$w_t -w_{xx} \geq f(x) \geq \mu \chi_{I_0}, \qquad \text{with} \quad w|_{\partial I_\beta}\ge 0, \quad w|_{t=0}=0.$$
Therefore, $w\geq \mu\displaystyle\int_0^t \big(e^{(t-s)\Delta_{I_\beta}} \chi_{I_0}\big) ds = 
\mu\displaystyle\int_0^t \big(e^{s\Delta_{I_\beta}} \chi_{I_0}\big) ds$, and \eqref{ut estimate step1} follows.

\textit{Step 2.} Let $v:=u_t$. We claim that
\begin{equation}\label{ut estimate step2 a}
v(t) \ \ge\  \mu e^{t\Delta_{I_\beta}} \chi_{I_0} + p\mu \int_0^t e^{(t-s)\Delta_{I_\beta}} \left[ \dfrac{\chi_{I_0} v(s,\cdot)}{1-(p+1)\mu\int_0^s (e^{\tau\Delta_{I_\beta}} \chi_{I_0}) d\tau}\right] ds
\quad\hbox{ in $(0,T)\times I_\beta$}.
\end{equation}
By \eqref{eqnut} and \eqref{ut estimate step1}, $v$ satisfies
$$v_t-v_{xx} = p(1-u)^{-p-1} f(x) v \geq \dfrac{p \mu \chi_{I_0}v}{1- (p+1)\mu \int_0^t (e^{s\Delta_{I_\beta}} \chi_{I_0}) ds}
\quad\hbox{ in $(0,T)\times I_\beta$}.$$
Therefore, \eqref{ut estimate step2 a} follows from the variation of constants formula.

\textit{Step 3.} We next claim that
\begin{equation}\label{ut estimate step2 a2}
v(t) \geq \mu \gamma(t) e^{t\Delta_{I_\beta}}\chi_{I_0}\quad\hbox{ in $(0,T)\times I_0$},
\end{equation}
with  $\gamma(s) = \bigl(1-(p+1)\mu  K_s s\bigr)^{-\frac{p}{p+1}}$.
Since 
\begin{equation}\label{etDeltasymmetric}
\hbox{$e^{s\Delta_{I_\beta}} \chi_{I_0}$ 
is  symmetric w.r.t.~$x_0$ and decreasing in $|x-x_0|$,}
\end{equation}
we have
$\int_0^s  \big( e^{\tau\Delta_{I_\beta}} \chi_{I_0}\big)(x) d\tau \geq  sK_s$
for all $0<s<T$ and all $x\in I_0$.
Since the numerator of the bracket in \eqref{ut estimate step2 a} is supported in $I_0$,
it follows from \eqref{ut estimate step2 a} that
$$v(t) \geq \mu e^{t\Delta_{I_\beta}} \chi_{I_0} + 
p\mu\int_0^t \theta (s) \big( e^{(t-s)\Delta_{I_\beta}} (\chi_{I_0} v(s,\cdot))\big) ds
\quad\hbox{ in $(0,T)\times I_\beta$},$$
where $\theta(s) = (1- (p+1)\mu K_ss)^{-1}.$ 
In particular, $\phi := \chi_{I_0}v$ satisfies
\begin{equation}\label{ut estimate step2 b}
\phi(t) \geq \mu \chi_{I_0} e^{t\Delta_{I_\beta}} \chi_{I_0} + p\mu \chi_{I_0} \int_0^t \theta(s) \big( e^{(t-s)\Delta_{I_\beta}} \phi(s)\big)ds
\quad\hbox{ in $(0,T)\times I_\beta$}.
\end{equation}
We want to show that
$\psi(t) = \mu \gamma(t) \chi_{I_0} e^{t\Delta_{I_\beta}} \chi_{I_0}$
is a subsolution of \eqref{ut estimate step2 b}.
This is equivalent to
\begin{equation}\label{ut estimate step2 subsol}
\mu \gamma(t) \chi_{I_0} e^{t\Delta_{I_\beta}} \chi_{I_0} \leq 
\mu \chi_{I_0} e^{t\Delta_{I_\beta}}\chi_{I_0} + p \mu^2 \chi_{I_0} \int_0^t \theta(s)\gamma(s)
e^{(t-s)\Delta_{I_\beta}} \left(\chi_{I_0} e^{s\Delta_{I_\beta}}\chi_{I_0}\right) ds
\quad\hbox{ in $(0,T)\times I_\beta$}.
\end{equation}
Next, for any $s>0$, using \eqref{etDeltasymmetric}, it follows from Lemma~\ref{semigroup comparison Lemma} that,
 for all $t>s$,
 $$
e^{(t-s)\Delta_{I_\beta}}\bigl[\chi_{I_0} e^{s\Delta_{I_\beta}}\chi_{I_0}\bigr]
\ge \Bigl(e^{(t-s)\Delta_{I_\beta}}\chi_{I_0}\Bigr) 
\Bigl(e^{(t-s)\Delta_{I_\beta}}\bigl[e^{s\Delta_{I_\beta}}\chi_{I_0}\bigr]\Bigr)
= \Bigl(e^{(t-s)\Delta_{I_\beta}}\chi_{I_0}\Bigr) 
\Bigl(e^{t\Delta_{I_\beta}}\chi_{I_0}\Bigr)
\quad\hbox{ in $I_\beta$.}
$$
Therefore, a sufficient condition for \eqref{ut estimate step2 subsol} is given by
$$\mu \gamma(t)\chi_{I_0} e^{t\Delta_{I_\beta}} \chi_{I_0} \leq \left[ 1+ p\mu \displaystyle\int_0^t
\theta(s) \gamma(s) \big( e^{(t-s)\Delta_{I_\beta}} \chi_{I_0}\big) ds \right] \mu \chi_{I_0} e^{t\Delta_{I_\beta}} \chi_{I_0}
\quad\hbox{ in $(0,T)\times I_\beta$},$$
which is equivalent to
$$\gamma(t)\leq 1+ p\mu \displaystyle\int_0^t
\theta(s) \gamma(s) \big( e^{(t-s)\Delta_{I_\beta}} \chi_{I_0}\big)(x) ds
\quad\hbox{ in $(0,T)\times I_0$}.$$
For this, by \eqref{etDeltasymmetric}, it is sufficient to have
\begin{equation}\label{CondGamma}
\gamma(t) \leq 1+ p\mu K_t \int_0^t \theta (s) \gamma(s) ds\quad\hbox{ for all $t\in(0,T)$.}
\end{equation}
 Note that $K_t$ is continuous and nonincreasing w.r.t.~$t>0$.
Now, for each $0<t\le \tau<T$, we set
$$\gamma_\tau(t)= \left[ 1-(p+1)\mu K_\tau t\right]^{-\frac{p}{p+1}},\qquad
\theta_\tau(t)= \left[ 1-(p+1)\mu K_\tau t\right]^{-1},$$
which are well defined by Step~1 and the monotonicity of $K_t$.
We have
$\gamma_\tau' (t) = p\mu K_\tau \theta_\tau(t)\gamma_\tau(t)$, hence 
$$\gamma_\tau(t) \le 1+ p\mu K_\tau \int_0^t\theta_\tau(s)\gamma_\tau(s)\,ds,\qquad 0<t\le \tau<T.$$
Letting $\tau\to t$ and using the continuity and monotonicity of $K_t$, we obtain
$$\gamma(t) \le 1+ p\mu K_t \int_0^t\theta_t(s)\gamma_t(s)\,ds
\le 1+ p\mu K_t \int_0^t\theta(s)\gamma(s)\,ds,\quad 0<t<T.$$
Therefore, \eqref{CondGamma}, hence \eqref{ut estimate step2 subsol}, is satisfied.
Property \eqref{ut estimate step2 a2} then follows from the comparison principle (in variation of constants form).

\textit{Step 4:} Since, for any $s>0$, $\left(1- (p+1)\mu \int_0^s (e^{\tau\Delta_{I_\beta}} \chi_{I_0}) d\tau\right)^{-1} \chi_{I_0}$ 
 is  symmetric w.r.t.~$x_0$ and decreasing in $|x-x_0|$, 
we deduce from Lemma~\ref{semigroup comparison Lemma} that, for all $t>s$,
\begin{eqnarray*}
e^{(t-s)\Delta_{I_\beta}} \left[ \dfrac{\chi_{I_0} v(s,\cdot)}{1- (p+1)\mu \int_0^s (e^{\tau\Delta_{I_\beta}} \chi_{I_0}) d\tau} \right]
&\geq & 
\mu \gamma(s) e^{(t-s)\Delta_{I_\beta}} \left[ \dfrac{\chi_{I_0} e^{s\Delta_{I_\beta}} \chi_{I_0}}{1- (p+1)\mu \int_0^s (e^{\tau\Delta_{I_\beta}} \chi_{I_0}) d\tau} \right] \\
&\geq &
\mu \gamma(s) \big( e^{t\Delta_{I_\beta}}\chi_{I_0} \big)e^{(t-s)\Delta_{I_\beta}} \left[ \dfrac{ \chi_{I_0}}{1- (p+1)\mu \int_0^s (e^{\tau\Delta_{I_\beta}} \chi_{I_0}) d\tau} \right].
\end{eqnarray*}
Estimate \eqref{ut improved est} then follows from \eqref{ut estimate step2 a} and \eqref{ut estimate step2 a2}.
Finally, the first part of estimate \eqref{Lambda est1} follows from \eqref{etDeltasymmetric},
and the second inequality from the fact that $0<\gamma (s)\le 1$.
\end{proof}

\begin{lemma}
\label{ut estimate}
Under the assumptions of Lemma~\ref{low bound u_t2}, the solution $u$ of problem (\ref{quenching problem}) satisfies
\begin{equation}\label{ut improved est 2}
u_t(t,x) \geq \big( 1+p\mu \tilde\Lambda (t,x)\big) \mu S(t,\beta) e^{t\Delta_{\mathbb{R}}} \chi_{I_0} \qquad \text{in } (0,T)\times  I_\beta,
\end{equation}
where $S$ is defined in \eqref{H and G} and $\tilde\Lambda(t,x)$ is given by
\begin{equation}\label{Lambda ut estimate 2}
\tilde\Lambda(t,x) = S(t,\beta) \int_0^t  \bigl(1-Y(s)\bigr)^{-\frac{p}{p+1}}
e^{(t-s)\Delta_{\mathbb{R}}} \left[ \dfrac{\chi_{I_0}}{1-(p+1)\mu S(t,0)\psi (s,\cdot)}\right] ds,
\end{equation}
with  $Y(s)=S(s,0)\,\text{\rm erf} \left(\frac{1}{\sqrt{s}}\right)\frac{(p+1)\mu}{2}\,s$ and
$ \psi(s,\cdot) = \displaystyle\int_0^s e^{\tau\Delta_{\mathbb{R}}} \chi_{I_0} d\tau.$
 
In particular, we can estimate
\begin{equation}\label{Lambda est1 2}
\tilde\Lambda(t,x) \geq S(t,\beta) \int_0^t  \bigl(1-Y(s)\bigr)^{-\frac{p}{p+1}-1} e^{(t-s)\Delta_{\mathbb{R}}} \chi_{I_0} ds
\qquad \text{in } (0,T)\times  I_\beta.
\end{equation}
\end{lemma}

\begin{proof}
This follows {by combining Lemma~\ref{low bound u_t2} and Proposition~\ref{prop semigroup comparison}},
 using in particular
\begin{eqnarray*}
\left( e^{s\Delta_{I_\beta}} \chi_{I_0}\right)(x_0+1)
&\ge& S(s,0)\left( e^{s\Delta_{\R}} \chi_{(-1,1)}\right)(1)
\ge \dfrac{S(s,0)}{\sqrt{4\pi s}} \int_{-1}^1 e^{-\frac{(1-y)^2}{4s}} dy \\
&=& \dfrac{S(s,0)}{\sqrt{\pi}} \int_0^{\frac{1}{\sqrt{s}}} e^{-z^2} dz 
= \frac{S(s,0)}{2} \text{\rm erf} \left(\dfrac{1}{\sqrt{s}}\right).
\end{eqnarray*}
\end{proof}

\subsection{Control of $u$ at $t=t_0(\tau)$ for points of small permittivity}

Our goal in this subsection is to take advantage of the smallness assumption in \eqref{condition 1 on f thm21} 
to improve the upper estimate  \eqref{compODE} of $u$ that was used in the proof of Propositions~\ref{type 1 quantitative} and \ref{type 1 quantitative1}.
  Estimate \eqref{compODE} followed from a mere comparison with the associated ODE problem $y'(t) = \|f\|_\infty (1-y(t))^{-p}$
and thus did not take advantage of the possible smaller values of $f$.
The following lemma provides a more precise control of $h(u(t_0,\cdot))$, which will allow a better lower estimate of the ratio 
$\eps_1$ in \eqref{epsilon 1b} (cf.~the proof of Proposition~\ref{type 1 quantitative1b}).

\begin{lemma}\label{u upper estimate}
Let $I\subset \Omega=(-R,R)$, $N\in (0,\|f\|_\infty]$ and assume that
\begin{equation}\label{condition 1 on f2}
f\le N \quad\hbox{ in $\overline I$.} 
\end{equation}
Then we have
\begin{equation}\label{y psi0}
u(t,x) \leq y(t)\theta_N(x)  \quad \text{ in } [0,T_*)\times\Omega,
\end{equation}
where $c_p=\textstyle\frac{(p+1)^{p+1}}{p^p}$, $T_*:= \frac{1}{(p+1)\| f\|_{\infty}}\le T$ and
\begin{equation}\label{y psi}
y(t) = 1-\bigl(1-(p+1)\|f\|_\infty t\bigr)^{\frac{1}{p+1}}, \quad
\theta_N(x) = \frac{N}{\|f\|_\infty} + \left(1-\frac{N}{\|f\|_\infty}\right)
\dfrac{1}{\cosh \bigl[\sqrt{c_p\|f\|_\infty}{\rm dist}(x,\Omega\setminus I)]}.
\end{equation}
\end{lemma} 

\begin{proof}
By Lemma~\ref{estimates for T}(i), we have $T_*\ge T$, as well as $u(t,x)\le y(t)$ in $[0,T_*)$.
It suffices to show that, if $z_0\in \Omega$ and $\delta>0$ are such that
\begin{equation}\label{hyp control u}
f(x)\leq N \quad \text{ for all } x\in [z_0-\delta, z_0+\delta] \cap [-R,R],
\end{equation}
then
\begin{equation}\label{hyp control u2}
u(t,z_0) \leq  y(t)
\left[\frac{N}{\|f\|_\infty} + \left(1-\frac{N}{\|f\|_\infty}\right)\dfrac{1}{\cosh \bigl[\sqrt{c_p\|f\|_\infty}\delta\bigr]}\right]
 \quad \text{ for all } t\in [0,T_*).
 \end{equation}

To prove \eqref{hyp control u2}, we look for a supersolution of problem \eqref{quenching problem} 
in $\mathcal{Q} := [0,T_*)\times ((z_0-\delta, z_0+\delta) \cap \Omega)$. We define the comparison function
$$
v(t,x):= y(t)\psi(x), \qquad \text{for } (t,x)\in \tilde{\mathcal{Q}} := [0,t_0(\tau)]\times [z_0-\delta, z_0+\delta],$$
where $\psi$ is a function to be chosen, satisfying $0\leq \psi \leq 1$, $\psi(z_0\pm\delta)=1$ 
and $\psi''\geq 0$.
By  \eqref{ODE}, \eqref{hyp control u} and using $0\le \psi(x)\le 1$ and $0\le y(t)<1$, we have
$$
Pv:=
v_t-v_{xx} -  \dfrac{f(x)}{(1-v)^p} =  y'(t)\psi(x) - y(t)\psi''(x) -  \dfrac{f(x)}{(1-y(t)\psi)^p} 
\geq \dfrac{\|f\|_\infty\psi(x)-N}{(1-y(t))^p} -y(t)\psi''(x).
$$
Since $\psi''\ge 0$, a sufficient condition to guarantee $Pv\ge 0$ in $\tilde{\mathcal{Q}}$ is thus
$$
\| f\|_\infty \psi (x) - \psi''(x) \displaystyle\max_{t\in[0,T)}  y(t)(1-y(t))^p  \geq N.
$$
An elementary computation shows that 
$$\displaystyle\max_{0\le s\le 1}s(1-s)^p=\displaystyle\max_{0\le X\le 1}X^p-X^{p+1}=\dfrac{p^p}{(p+1)^{p+1}},$$
so we are left with the following differential inequality for $\psi$:
$$
-\frac{\psi''(x)}{c_p\|f\|_\infty} + \psi(x) \geq \frac{N}{\|f\|_\infty}.
$$
 Solving the corresponding ODE, symmetrically in $[z_0-\delta, z_0+\delta]$, we obtain the solution
$$\psi (x) = \frac{N}{\|f\|_\infty} + B \cosh \Bigl[ \sqrt{c_p\|f\|_\infty} (x-z_0)\Bigr],$$
where $B$ is a constant. From the boundary conditions $\psi(z_0\pm\delta)=1$, we finally get
$$\psi (x) = \frac{N}{\|f\|_\infty} + \left(1-\frac{N}{\|f\|_\infty}\right)
\dfrac{\cosh \bigl[\sqrt{c_p\|f\|_\infty} (x-z_0)\bigr]}{\cosh \bigl[\sqrt{c_p\|f\|_\infty}\delta\bigr]},
\qquad x\in [z_0-\delta, z_0+\delta],$$
and we see that the requirements $0\leq \psi\leq 1$ and $\psi''\geq 0$ are satisfied.
Therefore, $Pv\ge 0$ in $\tilde{\mathcal{Q}}$.

Now, we look at the parabolic boundary of $\mathcal{Q}$. 
On the one hand we have $y(0)=0$, so $v(0,x)=u(0,x)=0$. 
On the other hand, for any point $(t,x)$ on the lateral boundary of $\mathcal{Q}$, we have either $x=\pm R$ or $x=z_0\pm\delta$.
In the first case, we have $v(t,x) = y(t)\geq 0=u(t,x)$. 
In the second case, since $y(t)$ is a supersolution of \eqref{quenching problem} 
on $[0,t_0(\tau)]\times \Omega$, we have $v(t,x) = y(t)\geq u(t,x)$.

We thus deduce from the comparison principle that $u\le v$ in $\mathcal{Q}$, and the lemma follows.
\end{proof}

\subsection{Proof of Theorems~\ref{global result dim 1c}-\ref{global result dim 1c2}}

We first establish the following quantitative type~I estimate,
which is the analogue of Propositions~\ref{type 1 quantitative} and \ref{type 1 quantitative1}.

\begin{proposition}\label{type 1 quantitative1b}
 Consider problem (\ref{quenching problem}) with $\Omega=(-R,R)$ and let $\xi\in\Omega$ 
 satisfy $|\xi|<R-1$. 
Let $0<d<d_0:=R-1-|\xi|$, 
$\mu > \mu_0(p)$, $e\in (d,\infty]$ and $\tau,\lambda\in (0,1)$.
Set $I_\ell=(\xi-1-\ell,\xi+1+\ell)$ and assume that
\begin{equation} \label{smallnessf2}
f\ge \mu \quad\hbox{ in $\overline I_0$}
\qquad\hbox{and}\qquad
f<\lambda\mu \quad\hbox{ in $(\overline\Omega\cap\overline I_e) \setminus I_d$.} 
\end{equation}
Then the touchdown time $T$ verifies $T>t_0:=\frac{1-\tau^{p+1}}{(p+1)\| f\|_\infty}$ and $u$ satisfies the type~I estimate
\begin{equation} \label{typeIest21}
\big[1 - u\big(t,\xi\pm(1+d)\big)\big]^{p+1} \geq (p+1)\overline\eps \mu (T-t)
\quad\hbox{ for all $t\in [t_0,T)$,}
\end{equation}
with
\begin{equation}\label{defepshat3}
\overline\eps=\sup_{(\beta,K)\in\widehat{\mathcal{A}}}\  \hat\eps(\beta,K),\qquad
\hat\eps(\beta,K) =  \dfrac{1}{2}\left(\dfrac{\beta-d}{\beta}\right)^{p+1} S(t_0,\beta)
G_1^*(\tau,t_0,\beta,K,\lambda),
\end{equation}
\begin{equation}\label{defCalA3}
\widehat{\mathcal{A}}=\mathcal{A}_3:=
\left\{ (\beta,K)\in (d,\bar d)\times(0,p], \ \ 
 K\mu\beta^2 \leq \textstyle\frac{p(p+2)-K}{p}, \ \ \delta_1(K)+\delta_2(\beta,K) \leq 1\right\},
\end{equation}
where $\bar d=\min\bigl[d_0,(d+e)/2\bigr]$, $S$ is defined in (\ref{H and G}) and $G^*_1,\delta_1,\delta_2$ are defined in Theorem~\ref{global result dim 1c}.
\end{proposition}
\goodbreak

\begin{proof}
Let $(\beta,K)\in\mathcal{A}_3$ and let $a(r)=a_{\beta,K}(r)$ be given by \eqref{function a q=1B}.
By the proof of Lemma~\ref{control Theta} with $q=1$, $\eta=1$, 
using the lower bound \eqref{ut improved est 2}, \eqref{Lambda est1 2} for $u_t$ in 
Lemma~\ref{ut estimate}, instead of  \eqref{low bound u_t}, 
we have
$$(1-u(t,x))^{p+1} \geq (p+1) {\hskip 0.5pt}\eps {\hskip 0.5pt}\mu{\hskip 0.5pt} a(|x-\xi|) (T-t) 
\quad\hbox{ in $[t_0,T)\times I_\beta$},$$
with 
\begin{equation}\label{defeps2}
\eps=S(t_0,\beta) \inf_{x\in I_\beta} 
\Bigl(1+p\mu \tilde{\Lambda} (t_0,x)\Bigr)\,
\dfrac{ \,e^{t_0\Delta_\R}\chi_{I_0}(x)}{h(u(t_0,x))a(|x-\xi|)}.
\end{equation}
In particular, by \eqref{function a q=1B}, we have
$$
\big[1 - u\big(t,\xi\pm(1+d)\big)\big]^{p+1} \ge (p+1)\left(\dfrac{\beta-d}{\beta}\right)^{p+1} \eps\, \mu (T-t),
\quad\hbox{ for all $t\in [t_0,T)$.}
$$

 We now estimate $\eps$ from below. Recalling 
\begin{equation}\label{heat-erf} 
e^{t\Delta_\R}\chi_{(-1,1)}(r)
=\dfrac{1}{\sqrt{4\pi t}}  \displaystyle\int_{-1}^1  e^{-\frac{(r-y)^2}{4t}}\,dy
=\dfrac{1}{\sqrt{\pi}}\int_{\frac{r-1}{2\sqrt{t}}}^{\frac{r+1}{2\sqrt{t}}} e^{-Z^2}\,dZ
=\dfrac{1}{2}\Bigl[{\rm erf}\Bigl(\dfrac{r+1}{2\sqrt{t}}\Bigr)+{\rm erf}\Bigl(\dfrac{1-r}{2\sqrt{t}}\Bigr)\Bigr]
\end{equation}
and setting $r=|x-\xi|$, we have
$$
\tilde{\Lambda} (t,x) \geq 
 \dfrac{1}{2}S(t,\beta) \int_0^t  \bigl(1-Y(s)\bigr)^{-\frac{p}{p+1}-1}\left[ \text{\rm erf} \left(\frac{r+1}{2\sqrt{t-s}}\right) 
 + \text{\rm erf} \left(\frac{1-r}{2\sqrt{t-s}}\right)  \right] ds 
= S(t,\beta) \Lambda(t,r),  
$$
where $\Lambda$ is defined in the statement of Theorem~\ref{global result dim 1c}. 

We next proceed to estimate the factor $h(u(t_0,x))$ in \eqref{defeps2} from above.
We recall from \eqref{Wcalcul} that, since $K\leq p$, the function $h$ is monotone increasing as a function of $u$.
Hence, we shall use the upper estimate given in Lemma~\ref{u upper estimate} in $I=\bigl(\xi+1+d,\min(\xi+1+e,R)\bigr)$ with $N=\lambda\mu$.
For this, we note that, for all $x\in [\xi,\xi+1+\beta]$, we have 
${\rm dist}(x,\Omega\setminus I)=(x-\xi-1-d)_+$ if $R\le \xi+1+e$, and that this is still true if $R>\xi+1+e$ due to $\beta<\bar d\le (d+e)/2$.
It follows from \eqref{y psi0} that, for all $x\in [\xi,\xi+1+\beta]$, 
we have 
$$u(t_0,x)\le y(t_0)\theta_{\lambda\mu}(x)=(1-\tau)\tilde u(|x-\xi|),$$
hence 
\begin{equation}\label{upperestimh}
h(u(t_0,x)) \le W(|x-\xi|),
\end{equation}
where $\tilde u$ and $W=W_{\tau,K,\lambda}$ are respectively given by \eqref{DefTildeu} and \eqref{W q=1}.
Moreover, \eqref{upperestimh} remains true for all $x\in I_\beta$ (applying Lemma~\ref{u upper estimate} in $I=\bigl(\max(\xi-1-e,-R),\xi-1-d\bigr))$.

Since $e^{t\Delta_\R}\chi_{(-1,1)}(r)$ and $\Lambda(t,r)$ are even and nonincreasing with respect to $r>0$
and $a(r)$ is even and constant on $[0,r_0]$, it follows from \eqref{upperestimh} that
$$\eps\geq S(t_0,\beta) \inf_{r\in (r_0,1+\beta)} 
\Bigl(1+p\mu S(t_0,\beta)\Lambda (t_0,r)\Bigr)\,
\dfrac{ \,e^{t_0\Delta_\R}\chi_{(-1,1)}(r)}{W(r)a(r)}.$$
This combined with \eqref{heat-erf} yields
\begin{equation*}
\eps \geq\dfrac{1}{2}S(t_0,\beta) \inf_{r\in (0,1+\beta)} 
\Bigl(1+p\mu S(t_0,\beta)\Lambda(t_0,r)\Bigr)
\dfrac{{\rm erf}\Bigl(\frac{r+1}{2\sqrt{t}}\Bigr)+{\rm erf}\Bigl(\frac{1-r}{2\sqrt{t}}\Bigr)}{W(r)a(r)} 
=\dfrac{1}{2} S(t_0,\beta)G_1^*(\tau, t_0,\beta,K,\lambda).
\end{equation*}
The conclusion then follows by taking the supremum over $(\beta,K)\in\mathcal{A}_3$.
\end{proof}

\begin{proof}[Proof of Theorems~\ref{global result dim 1c}-\ref{global result dim 1c2}]
 So as to prove both results at the same time, 
we set $m=0$, $d_2=d_0$ and $d_1=\infty$ in the case of Theorem~\ref{global result dim 1c}.
 It suffices to show that there are no touchdown points in $\overline{\mathcal{D}}$, 
where either:
\begin{itemize}
\item[(i)] $\mathcal{D}=(x_{i+1}+1+d,x_i-1-d)$ for some $i\in\{1,\dots,m-1\}$ 
(in Theorem~\ref{global result dim 1c2}), or 
\vskip 1mm
\item[(ii)] $\mathcal{D}=(x_0+1+d,R)$ (the case $\mathcal{D}=(-R,x_m-1-d)$ is similar).
\end{itemize}
\noindent We set $\Gamma=\partial\mathcal{D}$ in case (i) and $\Gamma=\{x_0+1+d\}$ in case (ii).

Let $\mathcal{A}$ be given by \eqref{condition on parameters 2c}. 
We first claim that $\mathcal{A}$ is nonempty, so that $\rho$ is well defined and positive.
Note that if we take 
$$\beta^2 = \dfrac{p(p+2)-K}{\mu K p},$$
then $\delta_2 (\beta,K) = 0$. Now, we can pick $K=p$ and $\beta= \sqrt{\frac{p+1}{p\mu}}\in (d,d_2)$. 
For this choice of $K,\beta$, we have
$$\delta_1(K) = \dfrac{\arctan(\sqrt{p+1})}{\sqrt{p\mu}} < 1,$$
provided $\mu > \frac{\arctan^2(\sqrt{p+1})}{p}$.
The claim follows.

 Next, by our assumption on $f$, we may select $(\beta,K,\tau,\lambda)\in \mathcal{A}$ such that
$$
\|f\|_{L^\infty(D)} < \mu{\hskip 1pt}\min\left\{\tilde\eps(\beta,K,\tau,\lambda),
\dfrac{1}{p+1} \dfrac{\tau^{p+1}}{(T-t_0(\tau))\mu},\lambda\right\}
$$
with
$$\tilde\eps(\beta,K,\tau,\lambda) =
\dfrac12\left(\dfrac{\beta-d}{\beta}\right)^{p+1}S(t_0(\tau),\beta)G^*(\tau,t_0(\tau),\beta,K, \lambda).$$
 Let $j\in \{i,i+1\}$ in case (i) and $j=0$ in case (ii). 
We shall apply Proposition~\ref{type 1 quantitative1b} with $\xi=x_j$ and $e=2d_1-d\ge 2d_2-d>d$.
By assumption \eqref{condition 1 on f thm21}, we have 
$$f<\lambda\mu\quad\hbox{ on } \bigl([x_j-1-e,x_j-1-d]\cup[x_j+1+d,x_j+1+e]\bigr)\cap \overline\Omega=
(\overline I_e\cap \overline\Omega)\setminus I_d\subset D.$$
Since $\bar d=\min(R-1-|\xi|,d_1)\ge \min(d_0,d_1)=d_2>\beta>d$,
it follows from \eqref{typeIest21}-\eqref{defCalA3} that
$$
(1 - u)^{p+1} \geq (p+1)\mu{\hskip 0.5pt}\tilde\eps(\beta,K,\tau,\lambda)  (T-t)
\quad\hbox{ on $[t_0,T)\times\Gamma$.}
$$
As $\|f\|_{L^\infty(D)}<\min\bigl\{\mu\tilde\eps(\beta,K,\tau,\lambda),
\frac{1}{p+1} \frac{\tau^{p+1}}{T-t_0(\tau)}\bigr\}$, we may apply Lemma~\ref{basic supersol},
to deduce that $\mathcal{T}\cap \mathcal{D}=\emptyset$.

Finally, let us show that $\mathcal{T}\cap \Gamma=\emptyset$.
By the continuity of $f$, assumption (\ref{condition 1 on f}) remains true for some $\tilde d<d$ close to $d$.
Moreover, since $\lambda\mu<\mu\le\|f\|_\infty$ and since $h(u)=K(1-u)+(1-u)^{-p}$ is nondecreasing on $[0,1)$ due to $K\le p$,
the function $W_{\tau,K,\lambda}(r)$ defined by \eqref{W q=1}, \eqref{DefTildeu} is increasing with respect to $d>0$.
It follows that the supremum in (\ref{simplified problem}) is a nonincreasing function of $d\in (0,d_2)$,
and we deduce that $\mathcal{T}\cap\Gamma=\emptyset$. This concludes the proof.
\end{proof}

\begin{remark}\label{global statement remark}
(i) Due to the search for a more precise control of $h(u(t_0,\cdot))$ to increase $\eps$ in \eqref{defeps2}, we have to make a smallness assumption on $f$ on both sides of the bump in Proposition~\ref{type 1 quantitative1b}
 (cf.~\eqref{smallnessf2}). For this reason, excluding touchdown on a single interval in 
 Theorems~\ref{global result dim 1c}-\ref{global result dim 1c2} would require a smallness condition on some additional intervals, 
 unlike in Theorems~\ref{global result dim 1}-\ref{global result dim 1a}. 
For simplicity, we have refrained from giving such a formulation of 
 Theorems~\ref{global result dim 1c}-\ref{global result dim 1c2}
and have restricted ourselves to a more global statement. 

(ii) We could use the more precise formula \eqref{Lambda ut estimate 2} instead of 
\eqref{Lambda est1 2}. However, numerical tests indicate that
the difference is extremely small, while the computational time is considerably larger.
\end{remark}


\section{Numerical procedures} 

\subsection{Iterative procedure for the optimization problem} \label{numerical procedure}

In this section, we describe the iterative procedure that we use to find  an accurate lower estimate for the solution of the optimization problem (\ref{simplified problem})  giving the threshold $\rho$ in Theorems~\ref{global result dim 1}-\ref{global result dim 1a}. 
The procedure consists of three steps:

\textit{Step 1: First exploration.}
 We first use a simple discretized exploration of the optimization set $\mathcal{A}$ 
 for $(\tau,\beta,K)$.
For this, we iterate in $\beta,\tau$ and $K$ as follows: 

$\bullet$  We initialize $\beta$ by setting
$$ \beta_0 = \min \left(1+d,\dfrac{d_0+d}{2}\right).$$
Then, for a chosen  value $\eps_\beta$ of the discretization parameter in $\beta$, we increment $\beta$ in the interval $(d,d_0)$,
 first increasingly (i.e., $\beta_{i+1} = \beta_i + \eps_\beta$) and then decreasingly (i.e., $\beta_{j+1} = \beta_j - \eps_\beta$).
 Note that additional stopping conditions will be given below.

$\bullet$ For each $\beta$,  we initialize $K$ by setting 
$$K_0 = \max \left( \dfrac{p}{\mu\beta^2}-\dfrac{1}{p+1},  \eps_K \right),$$
 where $\eps_K$ is a chosen value of the discretization parameter in $K$,
and then increment $K$ increasingly, i.e. $K_{i+1} = K_i + \eps_K$.

$\bullet$ For each couple $(\beta,K)$, we compute $\delta(\beta,K)$ defined in (\ref{H and G}).  
If it is  less than or equal to $1$, we then iterate in $\tau$.
This is done by picking $n_\tau$ equidistant points in the interval $\bigl(\frac{\mu}{2\mu-\mu_0}, 1\bigr)$, 
where $n_\tau$ is a chosen number of discretization points.
(The corresponding $\eps_\tau$ is thus $\bigl(1-\frac{\mu}{2\mu-\mu_0}\bigr)/n_\tau$.)

$\bullet$ For each such $(\beta,K,\tau)$, we then compute an approximation of $\rho$,  given by
\begin{equation}\label{def approx rho 2}
\tilde\rho:=\dfrac{1}{2}
   \left( \dfrac{\beta-d}{\beta}\right)^{p+1} \dfrac{S\bigl(t_0(\tau),\beta\bigr)}{K+\tau^{-p}} 
   \min \Bigl\{ \tilde H\bigl(t_0(\tau),\beta\bigr), \tilde G\bigl(t_0(\tau),\beta,K\bigr) \Bigr\},
\end{equation}
where 
$\tilde H(t,\beta)$ (resp., $\tilde G(t,\beta,K)$) is a  suitable approximation of $H(t,\beta)$
(resp., $G(t,\beta,K)$),  and $S(t,\beta)$ is given in (\ref{H and G}).
 To define $\tilde G$, $\tilde H$, we set
\begin{equation}\label{defNHDH}
\begin{array}{ll}
&N_H(x)=\text{\rm erf} \left( \dfrac{1}{\sqrt{t}}\left(1+\dfrac{\beta x}{2}\right)\right) 
- \text{\rm erf} \left(\dfrac{\beta x}{2\sqrt{t}}\right),\quad 
D_H(x)=(1-x)^{p+1},\\
\noalign{\vskip 2mm}
&N_G(x)=\text{\rm erf} \left( \dfrac{2-(1-x)\delta}{2\sqrt{t}}\right)
+ \text{\rm erf} \left(\dfrac{(1-x) \delta}{2\sqrt{t}}\right),\quad 
 D_G(x)= (\Gamma^2+1)^{\frac{\alpha}{2}}\cos^\alpha (Ax)
\end{array}
\end{equation}
 and recall that
$$H(t,\beta) = \displaystyle\inf_{0<x<1} \frac{N_H(x)}{D_H(x)},\qquad G(t,\beta,K) =
 \displaystyle\inf_{0<x<1} \frac{ N_G(x)}{D_G(x)},$$
where $\delta=\delta(\beta,K),\Gamma,A,\alpha$ are defined in (\ref{H and G}).
 We then set
\begin{equation}\label{def approx H G 2}
\tilde H(t,\beta):=\min_{0\le i\le n_x} \frac{N_H(x_{i})}{D_H(x_i)},\qquad 
\tilde G(t,\beta,K):=\min_{0\le i\le n_x} \frac{ N_G(x_i)}{D_G(x_{i})}, 
\end{equation}
 where $x_i=i/n_x$ for $i=0,1,\cdots,n_x$, and $n_x$ is the chosen number of discretization points in the interval $[0,1]$.
For each $i$, the corresponding quotients are computed using the error function provided by $Matlab$. 

 $\bullet$ We define the variables $\rho_{opt}$ and $\beta_1^*,K_1^*,\tau_1^*$,
which respectively stand for the largest value of $\tilde\rho$ obtained so far,
and for the corresponding values of the parameters $\beta, K, \tau$.
These variables are updated after each iteration.

$\bullet$  To avoid unnecessary computations, we also observe that we can (dynamically) further restrict the 
ranges of $\beta, K, \tau$, as follows:
\begin{equation}\label{restriction-range}
\beta\ge\frac{d}{1-(2\rho_{opt})^{1/( p+1)}},\qquad
K\le\frac{1}{2\rho_{opt}}-1,\qquad
\tau\ge(2\rho_{opt})^{1/p}.
\end{equation}
Indeed, since 
\begin{equation}\label{restriction-range2}
\bigl(\textstyle\frac{\beta-d}{\beta}\bigr)^{p+1}\le 1,\qquad S(t_0(\tau),\beta)\le 1
\end{equation}
and
\begin{equation}\label{restriction-range3}
H(t,\beta)\le \frac{N_H(0)}{D_H(0)}=\text{\rm erf} \Bigl(\frac{1}{\sqrt{t}}\Bigr)\le 1,\quad
G(t,\beta,K)\le \frac{N_H(1)}{D_H(1)}=\text{\rm erf} \Bigl(\frac{1}{\sqrt{t}}\Bigr)\le 1,
\end{equation}
we have
$$\tilde\rho\le \frac12\min\Bigl[\Bigl(1-\frac{d}{\beta}\Bigr)^{p+1},\frac{1}{K+1},\tau^p\Bigr],$$
so that any choice of $(K,\beta,\tau)$ violating at least one of the conditions in \eqref{restriction-range} will lead to values $\tilde\rho<\rho_{opt}$.

\textit{Step 2: Refined exploration.}  We 
make a finer second exploration near the parameters  $\beta_1^*,K_1^*$ and $\tau_1^*$ 
obtained in  Step 1.
 This is done by repeating Step 1 on the new ranges
\begin{equation}\label{newranges}
[\beta_1^*-\eps_\beta, \beta_1^*+\eps_\beta],\quad
[K_1^*-\eps_K, K_1^*+\eps_K],\quad
[\tau_1^*-\eps_\tau, \tau_1^*+\eps_\tau],
\end{equation}
taking a chosen number of equidistant points in each interval.
The values of $\beta, K, \tau$ corresponding to the largest $\tilde\rho$ obtained are denoted $\beta^*,K^*,\tau^*$.

\textit{Step 3: Lower estimate  of $\rho$.} 
Finally, for the parameters $(\beta^*,K^*,\tau^*)$ selected in  Step 2, 
we recompute a  ``safer'' approximation of the supremum $\rho$ by looking this time for a  lower estimate. This is done by setting
\begin{equation}\label{def approx rho}
\overline\rho:=\dfrac{1}{2}
   \left( \dfrac{\beta^*-d}{\beta^*}\right)^{p+1} \dfrac{S\bigl(t_0(\tau^*),\beta\bigr)}{K^*+\tau^{*-p}} 
   \min \Bigl\{\overline H\bigl(t_0(\tau^*),\beta^*\bigr), \overline G\bigl(t_0(\tau^*),\beta^*,K^*\bigr) \Bigr\},
\end{equation}
where $\overline H(t,\beta)$ (resp., $\overline G(t,\beta,K)$) is now a suitably accurate lower bound of $H(t,\beta)$
(resp., $G(t,\beta,K)$). 

To compute $\overline G$, $\overline H$, this time we choose another (larger)
number $n_x$ of equidistant discretization points of the interval $[0,1]$, we denote $x_i=i/n_x$ for $i=0,1,\cdots,n_x$,
and then set
\begin{equation}\label{def approx H G}
\overline H(t,\beta):=\min_{0\le i\le n_x-1} \frac{N_H(x_{i+1})}{D_H(x_i)},\qquad 
\overline G(t,\beta,K):=\min_{0\le i\le n_x-1} \frac{ N_G(x_{i+1})}{ D_G(x_i)}, 
\end{equation}
 where $N_H, D_H, N_G, D_G$ are given by \eqref{defNHDH}.
For each $i$, the corresponding quotients are computed using the error function provided by $Matlab$. 
Observe that the functions $N_H(x), D_H(x)$, $ N_G(x),  D_G(x)$) are monotonically decreasing
 in $[0,1]$ (owing to $A\in (0,\pi/2)$ and $\delta\in (0,1]$ for $(\beta,\tau,K)\in\mathcal{A}$).
Consequently,
$$\displaystyle\inf_{x_i\le x\le x_{i+1}} \frac{N_H(x)}{D_H(x)} \ge \frac{N_H(x_{i+1})}{D_H(x_i)},\quad
\displaystyle\inf_{x_i\le x\le x_{i+1}}  \frac{ N_G(x)}{ D_G(x)} \ge \frac{ N_G(x_{i+1})}{ D_G(x_i)},$$
hence
\begin{equation}\label{def approx H G2}
H(t,\beta)\ge \overline  H(t,\beta),\qquad G(t,\beta,K)\ge \overline G(t,\beta,K).
\end{equation}
Moreover, the discretization errors can be estimated by
$$H-\overline H\le  \frac{N_H(x_{i_0})-N_H(x_{i_0+1})}{D_H(x_{i_0})},
\qquad G-\overline G\le  \frac{N_G(x_{i_1})-N_G(x_{i_1+1})}{D_G(x_{i_1})},$$
where $i_0, i_1$ are the indices for which the respective minima in (\ref{def approx H G})
are achieved, so that $n_x$ can be adjusted to guarantee a satisfactory error level (say, $10^{-4}$).

\begin{remark}
(i)  In Step 3, it is consistent to choose a larger $n_x$ in order to have good lower estimates of $H,G$,
while in Steps 1 and 2 we  need to choose coarser partitions of the interval $[0,1]$, 
 in order to keep the computational cost of the method within feasible limits.

(ii)  In the above procedure, the only possible sources of errors in excess on $\overline\rho$ 
are the round-off machine errors and the numerical errors in 
the $Matlab$ evaluations (for instance those of $\text{\rm erf}$).
In principle this can be guaranteed with any reasonably prescribed safety margin.

\end{remark}

 In the following table, for $p=2$ and each of the values of $\mu, \|f\|_\infty, d, d_0$ considered in Table~\ref{tableex1},
we give the numerical approximation of the optimal parameters $\tau^*,\beta^*,K^*$ found by the above procedure,
as well as the lower bounds $\overline H(t_0(\tau^*),\beta^*),$
 $\overline G(t_0(\tau^*),\beta^*,K^*)$ for $H, G$ and the approximated semigroup comparison constant $S\bigl(t_0(\tau^*),\beta^*\bigr)$. 

\begin{table}[h]
\begin{center}
\begin{tabular}{|c|cccc|ccc|ccc|c|}
\hline
$p$ & $\mu$ & $\|f\|_\infty$   & $d$   & $d_0$   & $\tau^*$   & $\beta^*$    &   $K^*$    & $\overline H$ & $\overline G$ & $S$ & $\overline \rho$   \\
\hline   
2 & 1   & 1.1              & 0.1  & 5       & 0.7904   & 1.7400    & 1.4787     & 0.9220 & 0.9140   & 0.8452 & \textbf{0.1050}                  \\
\hline
2 & 1.25     & 1.3              & 0.1  & 3      & 0.8094   & 1.5600    & 1.1117     & 0.8807  & 0.8754     &  0.7429 & \textbf{0.1010}  \\
\hline
2 & 2     & 2.25            & 0.1 & 4     & 0.8111   & 1.2200    & 0.7184     & 0.7629 & 0.7650    & 0.8966    & \textbf{0.1182}                  \\                  
\hline
2 & 2     & 2.25            & 0.05 & 4     & 0.8201   & 1.1900    & 0.8228     & 0.7825 & 0.7869    & 0.9004    & \textbf{0.1341}                  \\
\hline
2 & 3     & 3.5            & 0.01 & 5    & 0.8036   & 0.9900    & 0.7402     & 0.7757 & 0.7710    & 0.9510 & \textbf{0.1554}                  \\
\hline
2 & 4     & 4.1            & 0.05 & 5    & 0.8001   & 0.9100     & 0.7407     & 0.8211 & 0.8182    & 0.9574  & \textbf{0.1436}                  \\
\hline
2 & 4   & 4.1              & 0.01 & 5   & 0.8286   & 0.8700     & 0.6705     & 0.7582 & 0.7517    & 0.9623 &\textbf{0.1643}    \\ 
\hline
2 & 4   & 7            & 0.01 & 5  & 0.7905   & 0.7300     & 0.9385     & 0.7137 & 0.7132   & 0.9739 &  \textbf{0.1313}    \\ 
\hline
2 & 6   & 6.2            & 0.01 & 10  & 0.8063   & 0.7300     & 0.7879     & 0.8252 & 0.8223  & 0.9917 &\textbf{0.1682}    \\ 
\hline
2 & 10   & 10            & 0.005 & 10  & 0.8037   & 0.5850     & 0.6331     & 0.7794 & 0.7832  & 0.9948 &\textbf{0.1732}    \\ 
\hline
1.5 & 10   & 10            & 0.005 & 10  & 0.7461   & 0.5850     & 0.6298     & 0.8390 & 0.8335  & 0.9932 &\textbf{0.1857}    \\ 
\hline
1 & 10   & 10            & 0.005 & 10  & 0.6611   & 0.5850     & 0.6000     & 0.8643 & 0.8643  & 0.9909 &\textbf{0.1992}    \\ 
\hline
0.5 & 10   & 10            & 0.005 & 10  & 0.5724   & 0.5850     & 0.4800     & 0.7972 & 0.7991  & 0.9877 &\textbf{0.2157}    \\ 
\hline
\end{tabular}
\end{center}
\caption{ Numerical parameters corresponding to the examples for Theorems~\ref{global result dim 1}-\ref{global result dim 1a}.}
\label{simplified lower estimates}
\label{Table Thm 11}
\end{table}

In practice we use $\eps_\beta=\eps_K = 0.1, n_\tau = 10$ for  Step 1, 
 whereas for Step 2 we take 10 equidistant points in the intervals  \eqref{newranges}.
As for the approximations of $H,G$ we take $n_x = 20$  in Steps 1 and 2. 
For the lower estimates in Step 3 we have chosen $n_x = 50000$ to compute $\overline H$ and $n_x = 2000$ to compute $\overline G$, which guarantees an error level no larger than $10^{-4}$.

\subsection{Numerical lower estimates for Theorem~\ref{global result dim 1b}}

 The numerical procedure is similar to that in Section~6.1, 
this time for the optimization problem (\ref{defrho1b}).
 The main difference is that we also need to
iterate in the parameter $\eta\in (0,1)$ for each couple $(\beta,K)$. 
As before, if $\delta (\beta,K,\eta)\le 1$, we then iterate in $\tau$, but now in the whole interval $(0,1)$.

 In the following table, for $p=2$ and each of the values of $\mu, \|f\|_\infty, d, d_0$ considered in Table~\ref{tableex2},
we give the numerical approximation of the optimal parameters $\tau^*,\eta^*,\beta^*,K^*$ found by the above procedure,
as well as the lower bounds $\overline H(t_0(\tau^*),\beta^*),$
 $\overline G(t_0(\tau^*),\beta^*,K^*,\eta^*)$,  $\overline G(\bar T,\beta^*,K^*,\eta^*)$ for $H, G$ 
 and the approximated semigroup comparison constants $S\bigl(t_0(\tau^*),\beta^*\bigr)$, $S\bigl(\bar T),\beta^*\bigr)$.
We also give the second term  of the minimum in \eqref{defrho1b}:
\begin{equation}\label{rho2star}
\rho_2(\tau^*) := \dfrac{1}{p+1} \dfrac{\tau^{p+1}}{(T-t_0(\tau^*))\mu}.
\end{equation}
In practice, this term  is observed to be larger than the first one.
However,  we have been unable to find a proof of this without assuming $\tau \geq \frac{\mu}{2\mu-\mu_1}$, 
which induces the extra hypothesis $\mu>\mu_1$ (cf.~Theorems~\ref{global result dim 1}-\ref{global result dim 1a}).

\begin{table}[h]
\begin{center}
\resizebox{\textwidth}{!}{
\begin{tabular}{|cccc|cccc|cccccc|c|}
\hline
$\mu$ & $\|f\|_\infty$   & $d$   & $d_0$   & $\tau^*$ & $\eta^*$  & $\beta^*$  &   $K^*$  &  $\overline G(\overline T)$ & $\overline H (t_0)$ & $\overline G(t_0)$ & $S(\overline T)$ & $S(t_0)$ &  $\rho_2(\tau)$  & $\overline{\rho}$   \\
\hline   
0.7   & 0.8              & 0.01  & 8       & 0.5800   &  0.8000    & 2.7100    & 0.8000   & 0.6352 & 0.7322 & 0.9465 & 0.6152 & 0.8492  &  0.1405   & \textbf{0.0815}                  \\
\hline
0.6     & 0.65       & 0.05  & 10    & 0.5600   &    0.8400   &  3.0500      & 1.000 & 0.5770 & 0.7230 & 0.9311 & 0.6289 &  0.8712  &  0.0977  & \textbf{0.0714}                  \\
\hline
0.5     & 0.6          & 0.001 & 6     & 0.3800   &   0.8800     &  3.8010   & 1.1000  & 0.4824 & 0.3440 & 0.9323 & 0.1357 & 0.6551 &   0.0187  &   \textbf{0.0137}                  \\
\hline
0.5     & 0.5            & 0.01 & 7    & 0.4400   &  0.8400      &  4.0100   & 0.9000  & 0.4907 & 0.4068 & 0.8983 & 0.2167 & 0.6865  &  0.0304    & \textbf{0.0228}                  \\
\hline
\end{tabular}
}
\end{center}
\caption{ Numerical parameters corresponding to the examples for Theorem~\ref{global result dim 1b}.}
\label{general approx table}
\end{table}

\subsection{Numerical lower estimates for 
Theorems~\ref{global result dim 1c}-\ref{global result dim 1c2}}

We use a similar numerical procedure to that presented in Section~\ref{numerical procedure}, this time for the optimization problem \eqref{defrho1c}. 
As a difference, the range for $K$ is now
\begin{equation}\label{newrangeK}
\eps_K\le K\le \min\Bigl(p,\frac{p(p+2)}{1+p\mu\beta^2}\Bigr).
\end{equation}
Also, we need to test the condition $\delta_1(K)+\delta_2(\beta,K)\le 1$ instead of $\delta(\beta,K)\le 1$.
Moreover, the additional stopping conditions in \eqref{restriction-range} are not available.
However, unlike in Section~6.1, 
this is not essential since the range for $K$ is already {\rm bounded, owing to} \eqref{newrangeK}.

 We also need to iterate the procedure with respect to the additional parameter $\lambda$.
Namely, we carry out the exploration of the parameter $\lambda$ in a certain subinterval of $(0,1)$ for each fixed admissible $(\beta,K,\tau)$.
To restrict the range of $\lambda$, let us first rewrite \eqref{defrho1c} under the form $\rho=\min(\hat\rho,\lambda)$
and observe that $G^*$ in \eqref{defrho1c} is a nonincreasing function of $\lambda\in (0,1)$ (due to \eqref{W q=1}, \eqref{DefTildeu}).
Since our numerical tests reveal that the method does not produce values of $\hat\rho$ 
larger than $0.3$ for $p=2$, we intitialize with $\lambda_0=0.3$.
We then compute the corresponding numerical value of $\hat\rho$.
As long as $\hat\rho_i<\lambda_i$, we iterate by setting $\lambda_{i+1}=\lambda_i-0.01$.
Once $\hat\rho_i\ge\lambda_i$, we stop and retain the largest between $\rho_i=\min(\hat\rho_i,\lambda_i)$
and $\rho_{i-1}=\min(\hat\rho_{i-1},\lambda_{i-1})$.
The values of $\lambda$ that we used in pratice remain in the interval $[0.2,0.3]$ (for $p=2$)
and the step $\Delta \lambda=0.01$ turns out to be sufficient since the results are not very sensitive to the variations of $\lambda$.

 Let us turn our attention to the main term in \eqref{defrho1c}, that we rewrite as
\begin{equation}\label{defG*2}
G^* = \displaystyle \inf_{r\in (r_0,1+\beta)}  \mathcal{G}(r),
\qquad\hbox{ with } \mathcal{G}(r)=\Bigl(1+p\mu S(t,\beta)\Lambda(t,r)\Bigr)\,
\dfrac{{\rm erf}\Bigl(\frac{r+1}{2\sqrt{t}}\Bigr)+{\rm erf}\Bigl(\frac{1-r}{2\sqrt{t}}\Bigr)}{W_{\tau,K, \lambda}(r) a_{\beta,K}(r)}
\end{equation}
(for fixed values of the parameters $\tau,t,\beta,K, \lambda$), where
\begin{equation}\label{defLambda2}
\Lambda (t,r) = \frac{1}{2} \int_0^t  \bigl(1-Y(s)\bigr)^{-\frac{p}{p+1}-1}\left[ \text{\rm erf} \left(\frac{r+1}{2\sqrt{ t-s}}\right) + \text{\rm erf} \left(\frac{1-r}{2\sqrt{ t-s}} \right) \right] ds
\end{equation}
and the functions $S, W, a, Y$ are defined in Theorem~\ref{global result dim 1c}.

As for the time integral in $\Lambda(t,r)$,
in the exploration process we use a coarser partition of the interval $[0,t]$ and the Simpson method, in order to have good approximations while keeping a reasonable computational time. 
However, in the last step, once we have chosen the approximated optimal parameters $(\tau,\beta,K,\lambda)$, we use the following monotonicity properties of the integrand and the rectangle rule in order to give a safe lower estimate of the ratio $\rho$. 

First, although $Y(s)$ is not monotone in general, we however note that,
 since $S(s,0)$ decreases with~$s$, we can estimate
\begin{equation}\label{estimY}
Y(s) \geq \tilde{Y} (t,s) = S(t,0)\, \text{erf} \left(\frac{1}{\sqrt{s}}\right) \frac{(p+1) \mu}{2}s,
\qquad 0<s<t.
\end{equation}
We will see numerically in Table~\ref{Table Thm 22} that $S(t_0,\beta)$, which 
satisfies $S(t_0,\beta)\le S(t_0,0)\le 1$, is of the order $\sim 0.99$ in our examples, so that the loss from from estimate~\eqref{estimY} is quite small.
Now, we observe that $\tilde Y (t,s)$ is monotonically increasing with respect to $s\in (0,t)$, due to
$$
\frac{\partial}{\partial s} \left[\sqrt{\pi}\,\text{erf} \left(\frac{1}{\sqrt{s}} \right) s \right] =
2 \int_0^{\frac{1}{\sqrt{s}}} e^{-t^2}dt - \frac{e^{-\frac{1}{s}}}{\sqrt{s}}
\geq 2 \frac{e^{-\frac{1}{s}}}{\sqrt{s}} - \frac{e^{-\frac{1}{s}}}{\sqrt{s}} >0.
$$
For fixed $(t,r)$, let $\{[t_j, t_{j+1}], \ j\in J\}$ be a partition of the interval $[0,t]$,
and set $\tau_j=t_j$ if $r\le1$ and $\tau_j=t_{j+1}$ if $r>1$.
We can therefore estimate $\Lambda(t,r)$ in \eqref{defLambda2} by
\begin{equation}\label{estimLambda}
\Lambda (t,r) 
\ge\dfrac{1}{2}\displaystyle\sum_{j\in J} (t_{j+1} - t_j) (1-\tilde Y (t,t_{j}))^{-\frac{p}{p+1}-1}
\Bigl[\text{\rm erf}\Bigl({\frac{r+1}{2\sqrt{t-t_j}}}\Bigr) + \text{\rm erf}\Bigl({\frac{1-r}{2\sqrt{t-\tau_j}}}\Bigr)\Bigr]_+.
\end{equation}
Since we are computing a lower estimate only in the last step, we can choose a much finer partition of the interval $[0,t]$, so as to keep enough accuracy in the rectangle rule.
For the examples of Table~\ref{Table Thm 22} we have used a partition of $[0,t]$ in 100 equal 
subintervals, while for the exploration process, we have used only 10.

Finally, we observe in \eqref{defLambda2} that $\Lambda(t,r)$ is monotonically decreasing as a function of $r>0$, as well as the function
$W_{\tau,\beta,K}$ defined in \eqref{W q=1} (owing to $K\le p$). Therefore, like in the numerical procedure described in Section~\ref{numerical procedure}, both the numerator and denominator 
 of $\mathcal{G}(r)$ in \eqref{defG*2}
are monotonically decreasing as functions of $r$. 
Hence, applying the same strategy as in Section~\ref{numerical procedure} 
(cf.~(\ref{def approx H G})-(\ref{def approx H G2}))
and then using \eqref{estimLambda} at the discretization
points $r=r_i$, we can compute a safe lower estimate of the infimum
  $G^\ast$ in \eqref{defG*2} once the parameters $\beta,K,\tau$ are chosen. 
For the examples in Table~\ref{Table Thm 22} we 
used a partition of 5000 subintervals, while in the exploration process, we used only 20. 

The following Table~\ref{simplified lower estimates} is the analogue for 
Theorems~\ref{global result dim 1c}-\ref{global result dim 1c2}
of Table~\ref{simplified lower estimates}, using the numerical procedure described above.  
We point out that, due to the finer mesh in the final step, the difference between the lower estimate $\overline \rho (\tau^*, \beta^*, K^*)$ and the explored value $\tilde\rho (\tau^*, \beta^*, K^*)$ is no larger than $10^{-3}$ for any of these examples.

\begin{table}[h]
\begin{center}
\begin{tabular}{|c|cccc|ccc|cccc|c|}
\hline
$p$ & $\mu$ & $\|f\|_\infty$   & $d$   & $d_0$   & $\tau$  & $\beta$ &  $K$ & $\overline {G^*}$   & $S(t_0,\beta)$ & $\rho_2 (\tau)$ 
&  $\overline \lambda$ &  $\overline \rho$   \\
\hline   
2 & 1   & 1.1              & 0.1  & 5       & 0.8000    &    1.6600 & 0.6800   & 0.5474 & 0.9899  &  0.4521  &  0.23  &  \textbf{0.2249}                      \\
\hline
2 & 1.25     & 1.3         & 0.1  & 3       & 0.8000   &  1.5400    & 0.5600   & 0.5735 & 0.9809  &  0.5423 &  0.24   & \textbf{0.2299}                   \\
\hline
2 & 2     & 2.25            & 0.1  & 4     & 0.8000    &  1.1400    & 0.6200   & 0.5601 & 0.9929 &  0.6482  &  0.22   & \textbf{0.2111}                     \\
\hline
2 & 2     & 2.25            & 0.05 & 4     & 0.7800    &  1.2300    & 0.5000   & 0.5712 & 0.9923 &  0.6272  &  0.26    & \textbf{0.2502}                 \\
\hline
2 & 3     & 3.5             & 0.01 & 5     & 0.8000    &  0.9300    & 0.6400   & 0.5591 & 0.9968 &  0.7106  &  0.28      & \textbf{0.2698}              \\
\hline
2 & 4     & 4.1             & 0.05 & 5     & 0.7800    &  0.9100    & 0.4200   & 0.5930 & 0.9971 &  0.8071  &  0.26  & \textbf{0.2495}                        \\
\hline
2 & 4     & 4.1             & 0.01  & 5    & 0.7200    &  1.0100    & 0.3200   & 0.5726 & 0.9965 &  0.7631  &  0.28   & \textbf{0.2769}    \\ 
\hline
2 & 4     & 7               & 0.01  & 5    & 0.7400    &  0.7700    & 0.6600   & 0.4653 & 0.9981 &  0.5327  &  0.23  & \textbf{0.2232}    \\ 
\hline
2 & 6     & 6.2             & 0.01  & 10   & 0.7800    &  0.7300    & 0.4600   & 0.5957 & 0.9994 &  0.8529  &  0.29   & \textbf{0.2856}    \\ 
\hline
2 & 10   & 10               & 0.005 & 10   & 0.8000    &  0.5450    & 0.5200   & 0.6007 & 0.9997 &  0.9310  &  0.3   & \textbf{0.2921} \\ 
\hline
1.5 & 10   & 10               & 0.005 & 10   & 0.7400    &  0.5450    & 0.3800   & 0.6349 & 0.9996 &  0.9074  &  0.32   & \textbf{0.3101} \\ 
\hline
1 & 10   & 10               & 0.005 & 10   & 0.6800    &  0.5250    & 0.3000   & 0.6762 & 0.9995 &  0.8755  &  0.34   & \textbf{0.3315} \\ 
\hline
0.5 & 10   & 10               & 0.005 & 10   & 0.6200    &  0.4650    & 0.2600   & 0.7503 & 0.9993 &  0.8231  &  0.37   & \textbf{0.3689} \\ 
\hline
\end{tabular}
\end{center}
\caption{Numerical parameters corresponding to the examples for 
Theorems~\ref{global result dim 1c}-\ref{global result dim 1c2}.}
\label{Table Thm 22}
\end{table} 

\begin{figure}
  \centering
\scalebox{0.6}[0.4]{\includegraphics[scale=1]{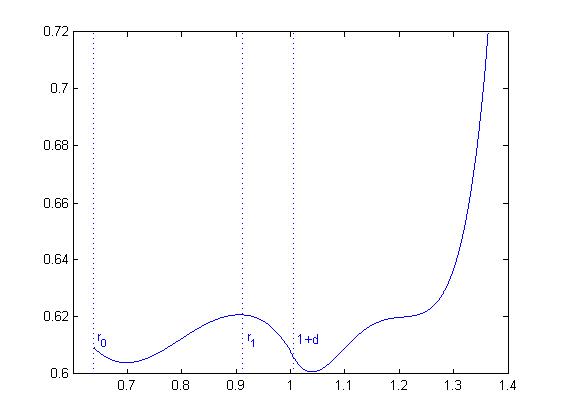}}
  \caption{The function $\mathcal{G}(r)$ in \eqref{defG*2} for $p=2$, $\mu=\|f\|_\infty=10$, 
  $d=0.005$, $d_0=10$.}
\label{FigGr}
\end{figure}

\begin{remark} \label{disclimitations}
(i) We observe from Tables~\ref{Table Thm 11} and \ref{Table Thm 22} 
that, in our examples, the factors $\bigl(\textstyle\frac{\beta-d}{\beta}\bigr)^{p+1}$ and $S=S(t_0,\beta)$
in formulae (\ref{simplified problem}) and \eqref{defrho1c}
are close to $1$ and have only small influence on the value of~$\rho$. 
We also see that the quantity $\rho_2(\tau)$ in \eqref{rho2star} does not affect $\rho$
(indeed, it corresponds to the second term in \eqref{defrho1c}, and turns out to be larger than the first term).
It follows that $\rho\approx \frac12 G^*$ in~\eqref{defrho1c},
once the optimal numerical parameters have been selected.

 (ii) Recall that $G^*$ is obtained as the infimum of the 
function $\mathcal{G}(r)$ in \eqref{defG*2}.
A plot of $\mathcal{G}(r)$ is given in Figure~\ref{FigGr} 
(for the last example with $p=2$ in Table~\ref{Table Thm 22},
and following the numerical method described above).
We observe that $\mathcal{G}(r)$ appears to be neither monotone nor convex.

(iii) The quantity $G^*$ could possibly be increased by taking into account the enhancing effect on $u_t$ of the positive values of $f$ outside the interval $I_0=(x_0-1,x_0+1)$, so as to improve estimate \eqref{ut improved est 2}
 and the ratio in \eqref{defeps2}.
This could be done at the expense of {\bf lower} assumptions on $f$ 
in the region where we want to rule out touchdown.
We have not pursued this further, since this would deviate too much from the main line of the article.
\end{remark} 

\goodbreak


\section{Appendix 1: Comparison estimates for the heat semigroup}

In this appendix, we establish the quantitative comparison properties for the heat semigroup,
that we have used in order to express the infima in (\ref{epsilon 1b}) and (\ref{epsilon 1}) in terms of the error function.
Here $e^{t\Delta_\Omega}$ and $e^{t\Delta_\mathbb{R}}$
respectively denote the Dirichlet heat semigroup on $\Omega$ and the heat semigroup on $\mathbb{R}$.

\begin{proposition}\label{prop semigroup comparison}
Let $\Omega = (-R,R)$ and assume that $I_0:=(x_0-1,x_0+1)\subset\subset \Omega$.
Let $\phi \in L^\infty (\mathbb{R})$ be a nonnegative function, symmetric with respect to $x_0$, nonincreasing in $|x-x_0|$ and supported in $I_0$.
Then, for all $t>0$, we have
\begin{equation}\label{compuwB}
\left(e^{t\Delta_\Omega}\phi \right) (x) 
\geq e^{-\lambda_\ell t}\Bigl[1-e^{-d_1(R-x)/t}\Bigr]\left(e^{t\Delta_\mathbb{R}}\phi \right) (x),
\quad x_0\le x\le R,
\end{equation}
and
\begin{equation}\label{compuwA}
\left(e^{t\Delta_\Omega}\phi \right) (x) 
\geq e^{-\lambda_\ell t}\Bigl[1-e^{-d_2(R+x)/t}\Bigr]\left(e^{t\Delta_\mathbb{R}}\phi \right) (x),
\quad -R\le x\le x_0,
\end{equation}
where $\ell=R-|x_0|$, $d_1=R-x_0-1$, $d_2=R+x_0-1$ and $\lambda_\ell=\Bigl(\dfrac{\pi}{2\ell}\Bigr)^2$.
\end{proposition}

\begin{remark}
(i) We note that similar {\it qualitative} results follow from known estimates of the Dirichlet heat kernel estimates from~\cite{Zh}:
\begin{equation}\label{estimZh}
G(t,x,y)\ge c_1\min\biggl[1,\frac{\delta(x)\delta(y)}{t}\biggr](4\pi t)^{-n/2}e^{-c_2|x-y|^2/t},\quad 0<t<T.
\end{equation}
The lower bound (\ref{estimZh}) is valid in any sufficiently smooth bounded domain $\Omega\subset \mathbb{R}^n$
and the corresponding upper bound is also true.
Although estimate (\ref{estimZh}) is quite powerful (and requires sophisticated methods), it is not suitable for our needs, 
since the constants $c_1, c_2,  T>0$  depending on $\Omega$ in~\cite{Zh} are not quantitatively estimated.

(ii) Another key feature of estimates (\ref{compuwB})-(\ref{compuwA}) in view of 
Theorems~\ref{global result dim 1}-\ref{global result dim 1a} is that they imply
\begin{equation}\label{compuwBasympt}
\frac{\left(e^{t\Delta_\Omega}\chi_{I_0}\right) (x)}{\left(e^{t\Delta_\mathbb{R}}\chi_{I_0}\right) (x)}\to 1,
\quad\hbox{ as $t/\delta(x)\to 0$},
\end{equation}
with quantitative control of the convergence, whereas (\ref{compuwBasympt}) does not follow from (\ref{estimZh}). 
Qualitative properties similar to (\ref{compuwBasympt}), valid also in higher dimensions, were obtained in~\cite{MS00} by different methods.
 \end{remark}

For the proof of Proposition~\ref{prop semigroup comparison} we will use the following lemma.

\begin{lemma}\label{semigroup comparison Lemma}
Let $\Omega = (x_0-\ell,x_0+\ell)$ 
 and $\phi_1,\phi_2\in L^\infty(\mathbb{R})$ be nonnegative, symmetric with respect to $x_0$ and nonincreasing in $|x-x_0|$.
Then we have
\begin{equation}\label{ProdSemigroup0}
e^{t\Delta_\Omega} (\phi_1\phi_2)\geq \left(e^{t\Delta_\Omega}\phi_1\right) \left(e^{t\Delta_\mathbb{R}} \phi_2\right)
\ge  \left(e^{t\Delta_\Omega}\phi_1\right) \left(e^{t\Delta_\Omega} \phi_2\right)
\quad\hbox{ in $(0,\infty)\times\Omega$.}
\end{equation}
\end{lemma}

\begin{proof}
Assume $x_0=0$ without loss of generality and set 
$$v(t,\cdot)=e^{t\Delta_\Omega}\phi_1,\quad w(t,\cdot) =e^{t\Delta_\mathbb{R}} \phi_2,\quad 
\underline{\varphi}(t,x)=v(t,x)\,w(t,x),\qquad t>0,\quad x\in \Omega.$$
Note that for each $t>0$, the functions $x\mapsto v(t,x)$ and $x\mapsto w(t,x)$ are even in $x$ and nonincreasing for $x\in [0,\ell]$.
Therefore,
$$\underline{\varphi}_t - \underline{\varphi}_{xx} = v(w_t-w_{xx})+w(v_t-v_{xx})-2v_x w_x = -2v_x w_x \leq 0,
\qquad t>0,\quad x\in (-\ell,\ell).$$
Since $\underline{\varphi}(0,\cdot)=\phi_1\phi_2$ and $\underline{\varphi}(t,\pm \ell)=0$, it follows from the maximum principle
that $e^{t\Delta_\Omega}  (\phi_1\phi_2)\ge \underline{\varphi}$  
in $[0,\infty)\times (-\ell,\ell)$,
 i.e. the first inequality in (\ref{ProdSemigroup0}). The second follows from the maximum principle.
\end{proof}

\begin{proof}[Proof of Proposition~\ref{prop semigroup comparison}]
It suffices to prove (\ref{compuwB}) (changing $x$ to $-x$). Let 
$v(t,x)= \left(e^{t\Delta_\Omega}\phi\right)(x)$
and let $w$ be the solution of the problem
\begin{equation*}
\left\lbrace \begin{array}{ll}
w_t - w_{xx} = 0, &  t>0,\ x\in (-\infty,R), \\
w(t,R) = 0, & t>0, \\
w(0,x) = \phi(x), & x\in (-\infty,R).
\end{array}\right.
\end{equation*}
Set $\Omega_1=(x_0-\ell,x_0+\ell)\subset\Omega$. By the maximum principle and (\ref{ProdSemigroup0}), we have
\begin{equation}\label{compuw1a}
v \ge e^{t\Delta_{\Omega_1}}\phi
 \ge \left(e^{t\Delta_{\Omega_1}}\chi_{\Omega_1}\right) \left(e^{t\Delta_\mathbb{R}}\phi\right).
 \end{equation}
On the other hand, setting $\varphi(x) = \cos\Bigl(\dfrac{\pi (x-x_0)}{2\ell}\Bigr)$, we have
\begin{equation}\label{compuw1b}
e^{t\Delta_{\Omega_1}}\chi_{\Omega_1}
\ge e^{t\Delta_{\Omega_1}}\varphi=e^{-\lambda_\ell t}\varphi.
\end{equation}
In particular, it follows from (\ref{compuw1a}), (\ref{compuw1b}) and the maximum principle that
$$v(t,x_0)\ge e^{-\lambda_\ell t} \left(e^{t\Delta_\mathbb{R}}\phi\right)(x_0)\ge e^{-\lambda_\ell t} w(t,x_0).$$
For each $t>0$, we thus have 
$$v(s,x_0)\ge e^{-\lambda_\ell t}w(s,x_0),\quad 0<s\le t.$$
It then follows from the maximum principle, applied to 
$z(s,x):=v(s,x)-e^{-\lambda_\ell t}w(s,x)$ on $[0,t]\times [x_0,R]$ for each $t>0$, that
\begin{equation}\label{compuw1}
v(t,x)\ge e^{-\lambda_\ell t}w(t,x),\quad t>0,\ x_0\le x\le R. 
\end{equation}
Now, $w$ admits the representation
$$w(t,x)=\int_{x_0-1}^{x_0+1} K(t,x,y) \phi(y) \, dy,
\quad\hbox{where }
K(t,x,y)=(4\pi t)^{-1/2}e^{-(x-y)^2/4t}\Bigl[1-e^{-(R-x)(R-y)/t}\Bigr]$$
is the Dirichlet heat kernel of the half-line $(-\infty, R)$.
For all $t>0$ and $x\in [x_0,R)$, we have
$$
w(t,x) \geq \Bigl[1-e^{-(R-x_0-1)(R-x)/t}\Bigr] \displaystyle \int_{x_0-1}^{x_0+1} (4\pi t)^{-1/2}e^{-(x-y)^2/4t} \phi(y) \, dy 
=\Bigl[1-e^{-(R-x_0-1)(R-x)/t}\Bigr]  \left(e^{t\Delta_\mathbb{R}}\phi\right) (x).
$$
This combined with (\ref{compuw1}) yields the desired estimate.
\end{proof}

\section{Appendix 2: Optimality of the  cut-off functions $a(r)$}

We here justify the claim, made in Section~3.3, 
about the optimality of the functions $a(r)$
involved in the main auxiliary functional $J$ from \eqref{defJLambda},
among all possible solutions of the differential inequality \eqref{diffineqa1}.

\begin{proposition}\label{optimal prolongation lemma}
Let $q\in [0,1]$, $\eta\in (0,1]$, $ R_1>1$, $K>0$, with $K\leq (p+q+1)p/q$ if $q>0$.
Let $h=h(u)$ and $F=F(r,\xi)$ be defined by \eqref{defhKq} and \eqref{diffineqa2}.
Assume that there exists a solution $a\in  W^{2,2}([0,R_1])$ of 
\begin{equation}\label{diffineqa1-app}
a''(r) \ge a(r)\,F\Bigl(r,\dfrac{a'(r)}{a(r)}\Bigr),\quad 0\le r<R_1,
\end{equation}
\begin{equation}\label{diffineqa2-app}
a'(0)=0,\quad a(1)=1,\quad a(R_1)=0,
\end{equation}
\begin{equation}\label{diffineqa3-app}
a>0\ \hbox{ and }\ a'\le 0 \quad\hbox{ in $[0,R_1)$.}
\end{equation}

(i) Then there exist $r_0\in [0,1)$, $\beta>0$ with $1+\beta\le R_1$,
 and a solution $\bar a\in  W^{2,2}([r_0,1+\beta])$ of 
\begin{equation}\label{diffeqa1-app}
\bar a''(r) = \bar a(r)\,F\Bigl(r,\dfrac{\bar a'(r)}{\bar a(r)}\Bigr),
\quad\hbox{ for a.e. $r\in(r_0,1+\beta)$,} \qquad
\bar a'(r_0)=0,\quad\bar a(1)=1,\quad \bar a(1+\beta)=0,
\end{equation}
such that
$$0<\bar a\le a\ \hbox{ and }\ \bar a'\le 0 \quad \hbox{ in $[r_0,1+\beta)$}.$$

(ii) Let $\bar a$ be extended by setting $\bar a(r)=\bar a(r_0)$ on $[0,r_0]$.
Then $\bar a\in  W^{2,2}([0,1+\beta])$ and $\bar a$ is a solution of \eqref{diffineqa1-app}-\eqref{diffineqa3-app}
with $R_1$ replaced by $1+\beta$.
Moreover, for any open interval $\Omega$ containing $(-1-\beta,1+\beta)$,  any $\ell\in [1,1+\beta]$ and any $t>0$, we have
\begin{equation}\label{compeps1}
\inf_{x\in ( -\ell,\ell)} \dfrac{e^{t\Delta_\Omega}\chi_{(-1,1)}(x)}{\bar a(|x|)}
\ge \inf_{x\in ( -\ell,\ell)} \dfrac{e^{t\Delta_\Omega}\chi_{(-1,1)}(x)}{a(|x|)}.
\end{equation}
\end{proposition}

\begin{proof}
(i) {\bf Step 1}. {\it Preliminaries.}
Set 
\begin{equation}\label{defF1}
F_1(X)=\displaystyle\sup_{u\in (1-\eta,1)}f_1(u,X),\qquad f_1(u,X)=m_1(u)X^2-M_1(u),
\end{equation}
where
\begin{equation}\label{defF1b}
m_1(u)=\dfrac{h'^2(u)}{hh''(u)},\quad M_1(u)=\dfrac{(p+q)K\mu}{(1-u)^{p+1-q}h(u)}.
\end{equation}
We claim that 
\begin{equation}\label{defF1c}
\hbox{$F_1$ is locally Lipschitz continuous on $\R$.}
\end{equation}

We have $m_1\in C([0,1])$,  owing to \eqref{calculQX} and using $K\leq (p+q+1)p/q\le p(p+1)/q(1-q)$ if $q\in (0,1)$. 
On the other hand, if $q\in  [0,1)$, 
then $M_1\in C([0,1))$ and $\lim_{u\to 1}M_1(u)=+\infty$,
whereas $M_1\in C([0,1])$ if $ q=1$. 
In both cases, for all $X\in\R$, there exists $u(X)\in [ 1-\eta,1]$ 
such that $F_1(X)=f_1(u(X),X)$ (with $u(X)\in [ 1-\eta,1)$ if $q\in [0,1)$). 
This, combined with \eqref{defF1} and \eqref{defF1b}, yields
$$\begin{array}{ll}
F_1(X)-F_1(Y)&=f_1(u(X),X)-f_1(u(Y),Y)\le f_1(u(X),X)-f_1(u(X),Y) \\
\noalign{\vskip 1mm}
&=m_1(u(X))(X^2-Y^2)\le \|m_1\|_\infty(|X|+|Y|)|X-Y|.
\end{array}$$
Exchanging the roles of $X, Y$, claim \eqref{defF1c} follows.

{\bf Step 2}. {\it Resolution of \eqref{diffeqa1-app} for $r<1$ and comparison.}
Set
$\hat F_1(\phi):=F_1(\phi)-\phi^2.$
Let $\psi$ be the maximal solution of the Cauchy problem
$$
\psi'=\hat F_1(\psi),\quad r<1,\qquad\hbox{ with }
\psi(1)=a'(1).
$$
Denote by $r^*\in [0,1)$ the endpoint of its interval of existence and set
$\bar a(r)=\exp\bigl[\int_1^r \psi(\tau)\,d\tau\bigr]>0.$
We claim that there exists $r_0\in [0,1)$ such that $ \bar a\in C^2([r_0,1])$, 
\begin{equation}\label{compabar0}
\frac{\bar a''}{\bar a}=F_1(\psi),\quad r_0<r<1,\qquad \bar a'(r_0)=0,\ \bar a(1)=1
\end{equation}
and
\begin{equation}\label{compabar}
\hbox{ $\bar a\le a$ and $\bar a'\le 0$ on $[r_0,1]$.}
\end{equation}

Using $\psi=\frac{\bar a'}{\bar a}$, we have
$\frac{\bar a''}{\bar a}=\psi'+\psi^2=\hat F_1(\psi)+\psi^2=F_1(\psi)$ for all $r\in(r^*,1)$.
Setting $\phi=(\log a)'=\frac{a'}{a}\le 0$, by \eqref{diffineqa1-app}, we have
$\phi'=\frac{a''}{a}-\phi^2\ge \hat F_1(\phi)$ for all $r\in (0,1)$,
and $\phi(1)=a'(1)<0$. 
Using \eqref{defF1c} and the fact that $\phi, \psi$ are locally bounded on $(r^*,1]$, 
for each $r_2\in (r^*,1)$, it follows that there exists $L>0$ such that
$$(\phi-\psi)_+'(r)\ge [\hat F_1(\phi(r))-\hat F_1(\psi(r))]\chi_{\{\phi>\psi\}}\ge - L(\phi-\psi)_+ 
\quad\hbox{ a.e. in $(r_2,1)$.}$$
Since $(\phi-\psi)(1)=0$, we conclude that $(\phi-\psi)_+=0$ on $(r^*,1]$, hence
\begin{equation}\label{compphibar}
\phi(r)\le \psi(r) ,\quad r^*<r<1.
\end{equation}
It follows that
$$a(r)=\exp\Bigl[\int_1^r \phi(\tau)\,d\tau\Bigr]\ge \exp\Bigl[\int_1^r \psi(\tau)\,d\tau\Bigr]= \bar a(r),
\quad r^*<r<1.$$

Now first consider the case when 
\begin{equation}\label{discrstar}
\hbox{$r^*=0$ and $\psi$ is continuous on $[0,1]$.}
\end{equation}
Then, since $\psi(0)\ge \phi(0)=0$ and $\psi(1)<0$, it follows that there exists a largest $r_0\in [0,1)$ such that $\psi(r_0)=0$.

Next consider the case when \eqref{discrstar} is not true. Then we must have $\lim_{r\to r^*}|\psi(r)|=\infty$.
On the other hand, since $m_1(u)\le 1$ by \eqref{condaexterior} and \eqref{condaexterior2},
we have $\psi'=\hat F_1(\psi)=F_1(\psi)-\psi^2\le 0$ on $(r^*,1]$.
It follows that $\lim_{r\to r^*}\psi(r)=+\infty$
and there again exists a largest $r_0\in (r^*,1)$ such that $\psi(r_0)=0$.

In both cases we have $ \bar a\in C^2([r_0,1])$, \eqref{compabar0} and \eqref{compabar}. 

{\bf Step 3}. {\it Resolution of \eqref{diffeqa1-app} for $r>1$ and comparison.}
By \eqref{condaexterior}, \eqref{condaexterior2}, we have 
$$F(r,X)=F_2(X):=\frac{p}{p+1}X^2, \qquad r>1.$$
In particular, we have $a''(r)\ge 0$ due to \eqref{diffineqa1-app}.
Since $a(R_1)=0<a(1)$, it follows that $a'(1)<0$.
Now set 
$$\bar a(r)=\Bigl(\frac{1+\beta-r}{\beta}\Bigr)^{p+1},\quad 1\le r\le1+\beta
\qquad\hbox{ where }\ \beta=-\frac{p+1}{\bar a'(1)}>0.$$
An immediate computation shows that
\begin{equation}\label{eqnabar1}
\frac{\bar a''}{\bar a}=F_2\Bigl(\frac{\bar a'}{\bar a}\Bigr),\quad 1<r<1+\beta.
\end{equation}
Moreover, $\bar a'(1)=-(p+1)/\beta=a'(1)$ and $\bar a\in  W^{2,2}(1,1+\beta)$.
Now define $\hat F_2(X)=F_2(X)-X^2=-\frac{X^2}{p+1}$, along with
$$\phi=\frac{a'}{a},\ \ 1\le r  <  R_1
\quad\hbox{ and }\quad 
\psi(r)=\frac{\bar a'(r)}{\bar a(r)}=-\frac{p+1}{1+\beta-r},\ \ 1\le r  <  1+\beta.$$  
Using \eqref{diffineqa1-app} and \eqref{eqnabar1}, we easily obtain
$\phi'\ge \hat F_2(\phi),\ 1\le r  <  R_1$, 
and $\psi'=\hat F_2(\psi),\ 1\le r  < 1+\beta,$ 
with $\psi(1)=\phi(1)$. 
Applying the argument leading to \eqref{compphibar}, this time for $r>1$, we obtain
$\phi(r)\ge \psi(r)$ for all $r\in (1,\min(R_1,1+\beta))$,
hence
\begin{equation}\label{compabar2}
a(r)=\exp\Bigl[\int_1^r \phi(\tau)\,d\tau\Bigr]\ge \exp\Bigl[\int_1^r \psi(\tau)\,d\tau\Bigr]= \bar a(r)>0,
\quad 1<r<\min(R_1,1+\beta).
\end{equation}
Since $a(R_1)=0$, it follows that $R_1\ge 1+\beta$.

Finally, since $\bar a\in C^{2}([r_0,1])$, $\bar a\in  W^{2,2}([1,1+\beta])$ and $\bar a'_-(1)=a'(1)=\bar a'_+(1)$, 
we have $\bar a\in  W^{2,2}([r_0,1+\beta])$.
In view of \eqref{compabar0}, \eqref{compabar}, \eqref{eqnabar1}, \eqref{compabar2}, and noting that $\bar a(1+\beta)=0$,
this completes the proof of assertion (i).
\smallskip

(ii) Since $\bar a\in  W^{2,2}([r_0,1+\beta])$ and $\bar a'_+(r_0)=0$, we have $\bar a\in  W^{2,2}([0,1+\beta])$.
By~\eqref{diffineqa2}, for all $r\in (0, r_0)$, 
we have $F(r,0)\le 0$, hence $\bar a''(r)=0\ge F(r,0)=F(r,\frac{\bar a'}{\bar a})$.
This along with  Steps~2 and 3 
 guarantees that $\bar a$ is a solution of \eqref{diffineqa1-app}-\eqref{diffineqa3-app}
with $R_1$ replaced by $1+\beta$.

Next, since $a(r)\ge a(r_0)\ge \bar a(r_0)=\bar a(r)$ on $[0,r_0]$
due to $a'\le 0$, we deduce from \eqref{compabar}, \eqref{compabar2} that
$0<\bar a\le a$ in $[0,1+\beta)$
and  property \eqref{compeps1} 
follows.
\end{proof}

\end{document}